\documentclass[twocolumn,amsthm]{autart}
\usepackage[T1]{fontenc}
\usepackage[utf8]{inputenc}
\usepackage{amsmath,amsthm,amssymb,mathrsfs}
\usepackage{bm}
\usepackage{graphicx}
\usepackage{placeins}
\usepackage{chngcntr}

\usepackage{booktabs,tabularx,array}
\usepackage{enumitem}
\usepackage{algorithm,algpseudocode}
\usepackage[round,authoryear]{natbib}
\makeatletter
\renewcommand{\bibfont}{\@bibliosize}
\makeatother
\usepackage{xcolor}
\definecolor{linkblue}{RGB}{23,61,94}
\usepackage[colorlinks=true,linkcolor=linkblue,citecolor=linkblue,urlcolor=linkblue]{hyperref}

\hypersetup{pdftitle={Amortized Feedback Planning: Turning Model-Based Rollouts into Executable Policies},pdfauthor={Jeonggyu Huh}}
\setlist{nosep,leftmargin=1.7em}
\numberwithin{equation}{section}
\newtheorem{assumption}{Assumption}[section]
\newtheorem{proposition}[assumption]{Proposition}
\newtheorem{theorem}[assumption]{Theorem}
\newtheorem{corollary}[assumption]{Corollary}
\newtheorem{lemma}[assumption]{Lemma}
\theoremstyle{remark}

\newcommand{\E}{\mathbb E}
\newcommand{\R}{\mathbb R}
\newcommand{\N}{\mathcal N}
\newcommand{\F}{\mathcal F}
\newcommand{\U}{\mathcal U}
\newcommand{\Qfun}{\mathcal Q}
\newcommand{\Adv}{\mathcal A}
\newcommand{\diag}{\operatorname{diag}}
\newcommand{\tr}{\operatorname{tr}}
\newcommand{\vech}{\operatorname{vech}}
\newcommand{\chol}{\operatorname{chol}}
\newcommand{\argmin}{\operatorname*{arg\,min}}

\newcommand{\norm}[1]{\left\lVert #1\right\rVert}

\begin{document}
\begin{frontmatter}
\runtitle{Amortized Feedback Planning}
\title{Amortized Feedback Planning: Turning Model-Based Rollouts into Executable Policies}
\author[skku]{Jeonggyu Huh}\ead{jghuh@skku.edu}
\address[skku]{Department of Mathematics, Sungkyunkwan University, Suwon, Republic of Korea}
\begin{abstract}
Closed-loop planning accounts for future observation-dependent actions but can be expensive to repeat at deployment. Bellman-gradient (BG) refinement differentiates conditional rollouts through an existing actor, corrects the current action, and stores the result in an executable policy. A backward sweep reuses deployed future feedback; full-horizon rollouts in an identified Gaussian belief model need no learned value critic. A nonlinear error recursion links quadratic continuation error, local correction, and policy storage. A controlled non-LQG example exhibits second-order policy accuracy with consistent storage, while ideal affine LQG admits exact backward recovery. In nonquadratic thrust, BG reduces actor cost by 2.47--3.20\% and remains within 0.13--0.53\% of the tested feedback MPC; original policies execute in $5$--$6\,\mu$s in a seed-0 native audit. With matched storage, BG attains competitive plant costs at about one ninth (arm) and one sixtieth (docking) of feedback-teacher-plus-student construction time on one GPU, excluding shared learning and node preparation. Richer common maps substantially narrow some student--BG gaps. These results expose the roles of learned feedback, local correction, and storage in executable control.
\end{abstract}
\end{frontmatter}

\begin{figure*}[t]
\centering
\includegraphics[width=\linewidth]{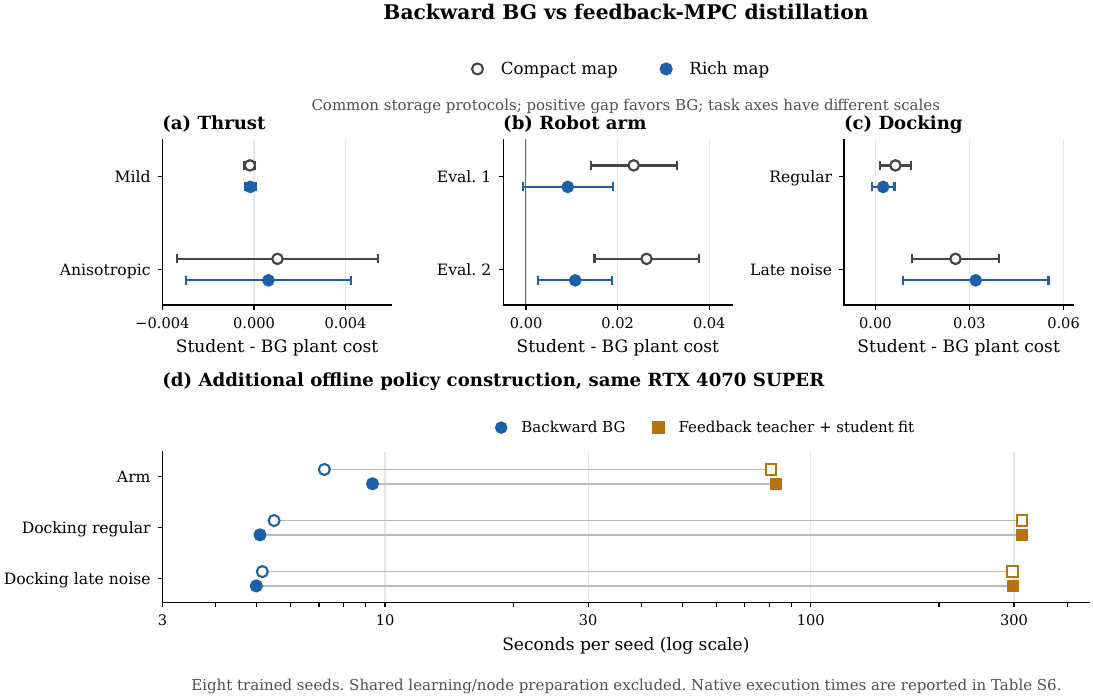}
\caption{BG versus feedback-MPC distillation with matched storage. (a--c) Paired student-minus-BG plant-cost gaps: positive favors BG; whiskers are exploratory 95\% $t$ intervals over eight seeds. Open/filled markers use compact/rich maps; arm panels reuse the same seeds. (d) Additional construction on one RTX~4070 SUPER, excluding shared model/actor learning and node preparation. Thrust totals used different GPUs and are omitted. Section~\ref{sec:matchedstorage} and Table~\ref{tab:matchedstoragefour} give protocols; Online Supplement Table~\ref{tab:matchednative} audits native execution.}
\label{fig:matchedstorage}
\end{figure*}

\section{Introduction}
Future information matters only through decisions that can respond to it. A controller purchasing an observation must anticipate how that observation changes later actions. Optimizing a fixed future action sequence omits this response, even if the controller replans after every measurement. This distinction between open-loop feedback (OLF) execution and closed-loop optimization is central to dual control \citep{barshalom1974,mesbah2018}.

Optimizing future feedback rules accounts for that response but enlarges the decision problem. Our question is whether a trained actor can supply the future feedback structure while a smaller calculation adds local precision. Instead of repeatedly optimizing all future feedback parameters, we differentiate conditional rollouts through an available continuation policy and correct only the current action. Backward construction then uses the policies already corrected at later dates. Regression stores these local decisions so that deployment requires a belief update and one policy evaluation.

This division of work has three parts: learning supplies broad feedback, model derivatives supply local action precision, and regression makes the result reusable. We call the local refinement Bellman-gradient (BG) correction. The base actor is trained by direct policy optimization (DPO), meaning policy search through simulated dynamics, unrelated to direct preference optimization. The pipeline is
\begin{equation}
(Y,C,u)\xrightarrow{\rm fit}\theta\xrightarrow{\rm DPO}\pi^0
\xrightarrow{\rm BG+storage}\bar\pi.
\label{eq:pipeline}
\end{equation}
The model and actor are frozen during refinement. Conditional rollouts retain both the observation response and the state dependence of future actions; no fitted cost-to-go function is needed in the full-horizon construction.

Figure~\ref{fig:matchedstorage} compares BG with feedback-MPC distillation under common map classes and fitting rules. Richer maps narrow several student--BG cost gaps; BG needs less additional construction in the tested arm and docking implementations. Target generation and BG's rebuilt continuation still differ, so this is not an isolated target-quality comparison. Online Supplement Table~\ref{tab:matchednative} audits matched-map execution; Figure~\ref{fig:introcostlatency} separately shows the original task-section policies.

Our central result is a nonlinear error recursion for the stored policies. Optimal-action stationarity makes continuation errors quadratic in future-policy error; a polynomial--kernel storage operator connects node corrections to feedback accuracy, with explicit derivative, safeguard and fixed-ridge terms. Ideal affine LQG gives exact backward recovery and a separate frozen-tail refinement reference. A controlled non-LQG experiment tests the local accuracy regime; scalar docking diagnoses precision lost in regression. These small structured tasks assess capacity, construction cost, latency and the remaining planner gap.

Sections~\ref{sec:beliefmodel}--\ref{sec:metrics} give the model, method and guarantees; Section~\ref{sec:experimentsoverview} evaluates the controllers. Complete proofs and the controlled accuracy example are in Technical Appendix A--E. Protocols, storage comparisons, computational audits and scalar DP are in Online Supplement S1--S6; lettered and S-prefixed references distinguish the two.

\section{Related work and positioning}\label{sec:positioning}
\paragraph{Policy learning and explicit control.}
Explicit MPC moves optimization offline \citep{bemporad2002}; learned approximate MPC fits executable policies under suitable validation conditions \citep{hertneck2018}. Differentiable predictive control learns through model rollouts, including stochastic closed-loop objectives without MPC supervision \citep{drgona2022dpc,drgona2022spdpc}. BG adds current-action refinement to such a base actor, rather than globally retraining its parameters.

\paragraph{Rollout improvement and storage.}
Newton interpretations of policy improvement are established \citep{puterman1979,bertsekas2022}; Kleinman and Hewer connect LQ policy iteration to algebraic Riccati solvers \citep{kleinman1968,hewer1971}. For discounted stationary problems, \citet{santosrust2004} establish local quadratic convergence to a discretized Bellman fixed point with piecewise-linear interpolation. Finite-horizon performance iteration learns current policies through fitted future feedback with sampling and approximation bounds \citep{hure2021}. Our result instead tracks one backward sweep of local action corrections and actual storage in a weighted $C^2$ policy norm, combining quadratic continuation error with derivative, safeguard and polynomial--kernel storage terms, including fixed-ridge bias. Guided policy search learns from optimized trajectories or actions \citep{levine2013gps,levine2014trajectory,zhang2016mpcgps}; POPLIN guides online planning with policies \citep{wang2020}. Our matched-storage comparison uses restricted feedback-MPC teachers.

\paragraph{Conditional derivatives and learned models.}
Stochastic value gradients differentiate simulated returns \citep{heess2015}. Belief-space DDP updates local quadratic values and affine feedback around a trajectory \citep{vandenberg2017}; BG differentiates full conditional rollouts through stored feedback and regresses current-action targets across belief nodes. Dreamer learns actors and values from imagined trajectories, while TD-MPC combines short planning with learned terminal values \citep{hafner2020dreamer,hafner2025dreamer,hansen2022,hansen2024tdmpc2}. \citet{jeon2026adjoint} recover constrained portfolio actions using fixed-latent open-loop BPTT adjoints and a generalized Hamiltonian. BG retains future feedback state derivatives and corrects against deployed continuation, storing policies backward in learned observation-based control. Structured Gaussian identification relates to DeepState and latent control models \citep{rangapuram2018,watter2015,tian2023}; our Gaussian inference requires explicit information closure \citep{sarkka2023}.

\section{Learning a Gaussian belief model}\label{sec:beliefmodel}
The identified model supplies conditional rollouts; the data-generating system is the plant. Fitted-model and plant costs are distinct throughout.

\subsection{Dynamics and information closure}
At dates $t=0,\ldots,T$, let $X_t\in\R^d$ be latent state, $Y_{t+1}\in\R^m$ the next observation, $C_t$ observed predictors and recurrent memory, and $W_t$ observed task variables. The controller chooses $u_t$ using history $\F_t$ before the next observation arrives. With coefficients evaluated at the known $(C_t,u_t)$, the fitted model is
\begin{equation}\label{eq:controlledmodel}
\begin{split}
X_{t+1}&=A_tX_t+a_t+L_t\xi^X_{t+1},\\
Y_{t+1}&=H_tX_{t+1}+D_t\nu_{t+1}.
\end{split}
\end{equation}
Fresh standard Gaussian noises are mutually independent and independent of the past, the initial conditional state is Gaussian, and $D_tD_t^\top\succ0$. Neural coefficient maps permit predictor-dependent dynamics. Thrust changes drift; sensing power changes observation variance.

\begin{assumption}[Predictor information closure]\label{ass:context}
The predictor update is specified by
\begin{equation}
C_{t+1}=\Gamma_t(C_t,Y_{t+1},\zeta_{t+1}),\label{eq:context}
\end{equation}
where fresh $\zeta_{t+1}$ has a specified law independent of $\F_t$, the latent state and model noises. Conditional on $\F_t,Y_{t+1}$, the new predictor supplies no additional unmodeled latent observation. The state $C_t$ includes any required memory.
\end{assumption}
Action-dependent updates use the corresponding extension. Additional latent information must enter the likelihood; an arbitrary encoder does not ensure closure.

\begin{proposition}[Gaussian closure]\label{prop:closure}
Under Assumption~\ref{ass:context} and the preceding noise and initial-state assumptions, $X_t\mid\F_t$ is Gaussian. Its mean $m_t$ and covariance $P_t$ follow the Kalman recursions \eqref{eq:pred}--\eqref{eq:covupdate} in Technical Appendix~\ref{proof:prop:closure}.
\end{proposition}
That appendix gives the conditioning proof and standard formulas. Innovation likelihood identifies parameters by differentiating through the filter and recurrent state \citep{ljung1999,sarkka2023}. Parameters are then frozen. The fitted covariance does not represent parameter uncertainty or misspecification.

\subsection{Conditional simulation and decision state}
Joint rollouts draw the latent state from $\N(m_t,P_t)$, propagate fitted dynamics and observations, and apply the same observation to filtering, task updates and costs. The equivalent innovation sampler is \eqref{eq:predsampler}--\eqref{eq:meansampler}. Unknown future predictors require a causal model; thrust uses a deterministic schedule, and the arm a fixed goal. One-step Gaussian closure does not require a jointly Gaussian path.

The task update and information state are
\begin{align}
W_{t+1}&=\Psi_t(W_t,C_t,u_t,Y_{t+1}),\label{eq:decisionstate}\\
s_t&=(W_t,C_t,m_t,\vech(P_t)),\label{eq:fullstate}\\
s_{t+1}&=\mathcal T_{\theta,t}(s_t,u_t,\omega_{t+1}),\label{eq:belieftransition}
\end{align}
where $\omega$ has a state-independent base law. The belief dimension is $d+d(d+1)/2$; task state and predictor memory add dimensions. Deterministic action-independent coefficient schedules can make $P_t$ a known function of time.

\paragraph{Nonlinear extension.}
The arm uses $X_{t+1}=F_\theta(X_t,u_t)+L_\theta\xi^X_{t+1}$, linear noisy observations, and an EKF. It predicts $m^-_{t+1}=F_\theta(m_t,u_t)$ and $P^-_{t+1}=J_tP_tJ_t^\top+L_\theta L_\theta^\top$, where $J_t=D_xF_\theta(m_t,u_t)$. Its likelihood and Gaussian posterior are approximate. Conditional rollouts restart from the fitted Gaussian and use the differentiable EKF. Restart and filtering errors supplement parameter error; Proposition~\ref{prop:closure} does not certify this extension.

\section{Model-grounded policy correction}\label{sec:method}
The baseline supplies feedback, local derivatives refine current actions, and regression stores them. The future observation response is retained in all conditional rollouts.
\subsection{Conditional derivatives with a fixed future policy}
Use the cost-minimization convention
\begin{equation}
J_\theta(\pi;s_0)
=\E_\theta^\pi\left[
\sum_{t=0}^{T-1}\ell_t(s_t,u_t,\omega_{t+1})+\Phi(s_T)
\,\middle|\,s_0\right].
\label{eq:objective}
\end{equation}
Here $u_t=\pi_t(s_t)\in\U_t(s_t)$.
Revealed costs must enter the information state. Hidden-state scores in our tasks are evaluator-only; training samples them inside the fitted model. For example, $\Phi(s_T)=\E_\theta[\phi(X_T)\mid s_T]$ for hidden terminal score $\phi$, with analogous stage expectations. DPO differentiates full simulated costs through belief dynamics to train $\pi^0$, without assuming global optimality. Define
\begin{align}
V_t^0(s)&=J_\theta(\pi^0;t,s),\qquad V_T^0=\Phi,\\
\Qfun_t^0(s,u)&=\E\left[\ell_t(s,u,\omega)
+V_{t+1}^0(\mathcal T_{\theta,t}(s,u,\omega))\right].
\label{eq:actionvalue}
\end{align}
The cost-to-go satisfies $V_t^0(s)=\Qfun_t^0(s,\pi_t^0(s))$. Complete conditional rollouts evaluate these functions without a learned critic.

At a correction date $t$, choose an admissible future feedback sequence $\rho_{t+1:T-1}$ and define
\begin{equation}
\Qfun_t^\rho(s,u)=\E_\theta[\ell_t(s,u,\omega)+V_{t+1}^\rho(\mathcal T_{\theta,t}(s,u,\omega))].
\label{eq:tailactionvalue}
\end{equation}
The continuation $\rho$ is the baseline tail for static BG and the completed deployed tail $\bar\pi_{t+1:T-1}$ for backward BG. Both anchor at $u^0=\pi_t^0(s)$. Future coefficients are fixed while their dependence on state inputs remains active.

Let $\mathcal G_t$ contain the fitted model, baseline, correction nodes and completed future maps before the current date's derivative samples are drawn. Conditional on $\mathcal G_t$, fresh base-noise units have the declared law, independent of the variable current action. Let $\mathscr C_t^\rho(s,u,\omega_{t+1:T})$ be the complete sampled cost with current action $u$ and continuation $\rho$. This specifies the conditioning for the pathwise statement \citep{heess2015}.

\begin{proposition}[Conditional closed-loop return derivatives]\label{prop:clderivatives}
Condition on $\mathcal G_t$, so that $(s,u^0,\rho)$ and the fitted model are fixed. Suppose the base-noise law is independent of $u$ and, on an open neighborhood $U$ of $u^0$, $\mathscr C_t^\rho$ is almost surely $C^2$ in $u$, with
\[
\E\!\left[\sup_{u\in U}\left\{|\mathscr C_t^\rho|+\norm{\nabla_u\mathscr C_t^\rho}
+\norm{\nabla_{uu}^2\mathscr C_t^\rho}_F\right\}\,\middle|\,\mathcal G_t\right]<\infty.
\]
Use smooth local coordinates for any differentiated positive-definite covariances. Then
\begin{equation}
\begin{aligned}
g_t&=\nabla_u\Qfun_t^\rho(s,u^0)=\E[\nabla_u\mathscr C_t^\rho\mid\mathcal G_t],\\
B_t&=\nabla_{uu}^2\Qfun_t^\rho(s,u^0)=\E[\nabla_{uu}^2\mathscr C_t^\rho\mid\mathcal G_t].
\end{aligned}
\label{eq:directderivatives}
\end{equation}
Averages over fresh conditionally independent rollout units are unbiased for these derivatives conditional on $\mathcal G_t$.
\end{proposition}
Technical Appendix~\ref{app:derivativealignment} proves the interchange. BPTT retains future-action and filter state dependence, freezing parameters and base noise; it does not differentiate policy construction. Unbiasedness conditions on $\mathcal G_t$ before current samples, not on maps selected using those samples.

\subsection{Local correction and acceptance}
Fix the continuation $\rho$ declared in \eqref{eq:tailactionvalue}. Write $g_t,B_t$ for its current-action gradient and Hessian at $u^0=\pi_t^0(s)$, suppressing $\rho$ in these two symbols. For a positive-definite action metric $M_t$, a shift $\lambda_t\geq0$ gives
\begin{equation}
\begin{aligned}
\widetilde B_t&=B_t+\lambda_tM_t\succ0,\\
d_t&=\argmin_{d:\,u^0+d\in\U_t(s)}
 \{g_t^\top d+\tfrac12d^\top\widetilde B_td\}.
\end{aligned}
\label{eq:bgqp}
\end{equation}
For unrestricted actions the solution is
\begin{equation}
 d_t=-\widetilde B_t^{-1}g_t.
\label{eq:bgsolve}
\end{equation}
Estimated Hessians are symmetrized. Thrust floors eigenvalues at 0.01 and caps displacements at norm 2. This Euclidean radial safeguard is not the norm-constrained metric-QP solution. Rejection returns $u^0$; numerical failure stops construction.

A proposed action takes the form
\begin{equation}
 u_t^{\rm prop}=u^0+\alpha_t d_t,\qquad 0\leq\alpha_t\leq1.
\label{eq:deploy}
\end{equation}
For this same fixed continuation, exact backtracking can enforce
\begin{equation}
\begin{aligned}
\Qfun_t^\rho(s,u^0+\alpha_td_t)&\le\Qfun_t^\rho(s,u^0)+c\alpha_tg_t^\top d_t,\\
&\hspace{35mm}c\in(0,1).
\end{aligned}
\label{eq:armijo}
\end{equation}
The implementation instead accepts one safeguarded proposal only if separate paired Monte Carlo rollouts give lower estimated cost. Both actions share the same tail and comparison noise, independent of derivative noise (thrust: 128 derivative, 256 acceptance paths). This zero-margin check differs from Armijo sufficient decrease; its error enters the comparison bounds.

\subsection{Backward policy construction on sampled belief nodes}\label{sec:backward}
Sample dated belief nodes $\mathcal S_t=\{s_{t,i}\}_{i=1}^N$ from fitted-model baseline trajectories with declared mean perturbations. Descending in time, correct $\pi_t^0(s_{t,i})$ through the chosen fixed tail. Given accepted targets $v_{t,i}$ (baseline after rejection), store
\begin{equation}
 \widehat f_t=\mathsf{Fit}_t\bigl(\{(s_{t,i},v_{t,i})\}_{i=1}^N;
     \mathcal P_t,\mathcal L_t\bigr).
\label{eq:policyregression}
\end{equation}
Here $\mathcal P_t$ is the map class and $\mathcal L_t$ its loss and selection rule. Date-specific residual maps reuse the actor; optional cross-validation selects using targets only. Fitting before a nonlinear action transform can differ from fitting executed actions (Online Supplement~\ref{sup:regressionrefit}).

In the nonquadratic experiment, the executed policy also limits extrapolation relative to the baseline:
\begin{equation}
\begin{aligned}
\bar\pi_t(s)&=\pi_t^0(s)+\mathcal C_r(\widehat f_t(s)-\pi_t^0(s)),\\
\mathcal C_r(d)&=d\min\{1,r/\norm d\},\qquad r=2.
\end{aligned}
\label{eq:meshdeploy}
\end{equation}
with $\mathcal C_r(0)=0$. Acceptance precedes regression, so targets and executed actions may differ. The cap bounds displacement, not cost. Action constraints require a feasible parameterization or projection, included in the deployed tail. Backward reuse of fitted policies relates to performance iteration \citep{hure2021}; here conditional simulation evaluates continuation cost.

\paragraph{Regularity and deployment.}
Hard caps are piecewise smooth, so pathwise automatic differentiation need not estimate the population Hessian without bias. Regularized curvature can instead serve as a metric under Proposition~\ref{prop:inexactmetric}, without asserting Newton convergence. Technical Appendix~\ref{app:regularity} gives a boundary-term example and Lemma~\ref{lem:smoothcap} a smooth alternative. The cost identities permit nonsmooth policies under integrability, with acceptance and storage errors retained. Algorithm~\ref{alg:main} uses actual deployed future maps throughout.

\begin{algorithm}[t]
\caption{Backward BG policy construction}\label{alg:main}
\begin{algorithmic}[1]
\State Fit and freeze the belief model; train baseline $\pi^0$.
\State Sample dated nodes from fitted-model baseline paths.
\For{$t=T-1,\ldots,0$}
 \State Fix deployed tail $\rho=\bar\pi_{t+1:T-1}$; retain its state derivatives, including deployment transforms.
 \For{each node $s_{t,i}$}
  \State Estimate current-action derivatives around $\pi_t^0(s_{t,i})$ with fresh conditional rollouts.
  \State Safeguard the proposal; compare it with the baseline on separate paired rollouts using the same tail.
  \State Accept the target or retain the baseline; stop and report numerical failures.
 \EndFor
 \State Fit and, where specified, target-validate the physical-action map $\bar\pi_t$, including its safeguard.
\EndFor
\State Freeze and evaluate on held-out episodes.
\State Deploy by updating the belief and evaluating $\bar\pi_t$.
\end{algorithmic}
\end{algorithm}

With $N$ nodes per date and $M$ rollout paths per node (derivative and acceptance combined), the backward sweep requires approximately
\begin{equation}
 N M\sum_{t=0}^{T-1}(T-t)=O(NMT^2)
 \label{eq:constructionwork}
\end{equation}
model steps before differentiation, cross-validation, and regression overhead; dense Hessian work also depends on action dimension. Deployment uses only filtering and policy features, but richer fitted maps increase storage and action latency. Total cost therefore includes DPO, correction, map fitting, and deployment volume, whereas independent MPC requires no DPO training.

Model fitting, construction, and testing use separate data and random streams; held-out evaluation measures transfer beyond the sampled nodes.

\section{Guarantees and evaluation criteria}\label{sec:metrics}
All results condition on the fitted model. The nonlinear recursion connects continuation, correction and storage; affine LQG supplies exact references.

\subsection{Backward improvement and deployment error}
Condition on the completed construction: $\bar\pi$ is fixed and evaluation trajectories are fresh. With $\bar V_t=V_t^{\bar\pi}$, define
\begin{align}
\bar\Qfun_t(s,u)&=\E_\theta[\ell_t(s,u,\omega)
 +\bar V_{t+1}(\mathcal T_{\theta,t}(s,u,\omega))],\nonumber\\
d_t^{\rm back}(s)&=\bar\Qfun_t(s,\bar\pi_t(s))
 -\bar\Qfun_t(s,\pi_t^0(s)).\label{eq:backwarddefect}
\end{align}
Both actions use the same actual corrected tail. Set
\begin{equation}
 \pi^{[t]}=(\pi_0^0,\ldots,\pi_{t-1}^0,\bar\pi_t,\ldots,\bar\pi_{T-1}),
 \quad t=0,\ldots,T.\label{eq:hybridpolicies}
\end{equation}
\begin{theorem}[Backward policy comparison]\label{thm:backward}
Let $\pi^0,\bar\pi$ be measurable admissible policies with the same initial state and terminal cost. If the hybrid-policy costs and the following defects are integrable, then
\begin{equation}
 J_\theta(\bar\pi;s_0)-J_\theta(\pi^0;s_0)
 =\E_\theta^{\pi^0}\sum_{t=0}^{T-1}d_t^{\rm back}(s_t).
 \label{eq:backwardperformance}
\end{equation}
If $d_t^{\rm back}\le b_t$ almost surely under the baseline state laws, then
\begin{equation}
 J_\theta(\bar\pi;s_0)-J_\theta(\pi^0;s_0)
 \le\E_\theta^{\pi^0}\sum_{t=0}^{T-1}b_t(s_t).
 \label{eq:backwardbudget}
\end{equation}
In particular, nonpositive defects imply nominal cost nonincrease.
\end{theorem}
\begin{proof}
Adjacent hybrids share the baseline prefix and corrected tail. Conditioning on their common state $s_t$ gives
$J_\theta(\pi^{[t]};s_0)-J_\theta(\pi^{[t+1]};s_0)
=\E_\theta^{\pi^0}d_t^{\rm back}(s_t)$.
Sum over $t$: $\pi^{[0]}=\bar\pi$, $\pi^{[T]}=\pi^0$, and the costs telescope.
\end{proof}
The baseline visitation law matches the construction nodes, but finite-node diagnostics do not certify the deployed policy.

\begin{corollary}[Acceptance and interpolation defect]\label{cor:backwarddefect}
At states almost surely under the baseline laws, let a target rule return the baseline or accept $a_t(s)$ using the actual corrected tail, with indicator $I_t(s)$. At an accepted pair assume absolute action-value approximation errors at most $\varepsilon_t(s)$ and
\[
 \widehat{\bar\Qfun}_t(s,a_t)-\widehat{\bar\Qfun}_t(s,\pi_t^0(s))
 \le-\eta_t(s),\qquad\eta_t(s)\ge0.
\]
If $\bar\Qfun_t(s,\cdot)$ is $L_t^{Q}(s)$-Lipschitz between the target and deployed action, and $e_t^{\rm int}=\norm{\bar\pi_t-a_t}$, then
\begin{equation}
 d_t^{\rm back}\le I_t(2\varepsilon_t-\eta_t)+L_t^{Q}e_t^{\rm int}.
 \label{eq:backwardinterpolation}
\end{equation}
When integrable, this is a budget in \eqref{eq:backwardbudget}.
\end{corollary}
The two accepted-pair errors contribute $2\varepsilon_t$, rejection gives zero target defect, and the Lipschitz bound controls storage. Full proofs and baseline-continuation results are in Technical Appendix~\ref{proof:thm:backward} and~\ref{app:baselinecomparison} (Proposition~\ref{prop:descent}, Theorem~\ref{thm:improve}, Corollary~\ref{cor:approx}). Regression and caps may alter targets without renewed acceptance at unseen states; fresh closed-loop tests evaluate improvement.

\subsection{Nonlinear backward accuracy}\label{sec:nonlinearaccuracy}
Write $\pi^\star=\pi^{\theta,\star}$ and, for $v\ge1$, define
\begin{equation}
 \norm f_{2,v}=\sum_{k=0}^2\frac1{k!}\sup_s\frac{\norm{D^kf(s)}}{v(s)}.
 \label{eq:nlpolicynorm}
\end{equation}
Set $E_t=\norm{\bar\pi_t-\pi_t^\star}_{2,v_t}$. Let $\epsilon_t^0$ be the maximum baseline action error relative to $\pi_t^\star$ on construction nodes. State derivatives enter because earlier rollouts differentiate future feedback.

\paragraph{Regularity and storage assumptions.}
Fix a smooth policy class containing the stored maps and an interior optimum, with integrable envelopes permitting differentiation under expectations. Put $H_{t,\alpha}^\rho(s)=Q_{t,uu}^\star(s,\pi_t^\star(s)+\alpha(\rho_t(s)-\pi_t^\star(s)))$. For weights $w_t\ge \chi_tv_t^2$, assume uniformly over this class and $\alpha\in[0,1]$ that
\begin{align}
 \norm{H_{t,\alpha}^\rho}_{2,\chi_t}
 &\le L_t^\star,\nonumber\\
 \norm{\mathsf P_t^\rho f}_{2,w_t}&\le C_t^{\rm tr}\norm f_{2,w_{t+1}}.
 \label{eq:nlmainregularity}
\end{align}
Here $\mathsf P_t^\rho$ is the one-step expectation operator under policy $\rho$. Finite weighted moments of transition derivatives also bound the action gradient and Hessian of $\E f(\mathcal T_{\theta,t}(s_i,u,\omega))$ by $\Gamma_{g,t}\norm f_{2,w_{t+1}}$ and $\Gamma_{H,t}\norm f_{2,w_{t+1}}$ on node action balls. The explicit moment conditions in Technical Appendix~\ref{app:nlcontinuation} allow unbounded states with suitable growth envelopes.

On each radius-$r_{B,t}$ ball about $\pi_t^\star(s_i)$, assume $2\mu_t I\preceq Q_{t,uu}^\star\preceq L_{*,t}I$ and a Lipschitz actual-tail Hessian. The used matrix satisfies $\widehat B_t\succeq\widehat m_tI$; its Hessian, gradient and solve residual errors relative to that tail are $\delta_{H,t},\delta_{g,t},\delta_{\rm lin,t}$. Proposal shaping changes the raw step by at most $\eta_{\rm node,t}$. A paired acceptance difference has absolute error at most $\tau_t$. Finally, storage must satisfy
\begin{equation}
 E_t\le s_t^{\rm reg}b_t+F_t^{\rm reg}+\eta_{\rm dep,t}.
 \label{eq:nlmainstorage}
\end{equation}
Here $b_t$ bounds node-target error, $s_t^{\rm reg}$ is storage stability, $F_t^{\rm reg}$ includes approximation and regularization bias, and $\eta_{\rm dep,t}$ bounds the deployment-transform error in the norm \eqref{eq:nlpolicynorm}. Equation~\eqref{eq:nlstoragebudget} derives this bound for the polynomial--kernel map, including its intermediate cap and arithmetic error. Node MSE alone does not control derivatives or off-node approximation.

\begin{theorem}[Nonlinear stored-policy error recursion]\label{thm:nonlinearbackward}
Under the preceding assumptions, fix a finite horizon and impose the local radius checks stated below. There are finite nonnegative constants such that
\begin{align}
 E_t\le{}&A_t^{\rm N}(\epsilon_t^0)^2
  +\sum_{j>t}B_{tj}^{\rm tail}E_j^2+F_t^{\rm reg}+\eta_{\rm dep,t}\nonumber\\
 &+C_t^{\rm err}(\delta_{H,t}\epsilon_t^0+\delta_{g,t}
 +\delta_{\rm lin,t}+\eta_{\rm node,t}+\sqrt{\tau_t}).
 \label{eq:nonlinearbackward}
\end{align}
Future errors concern the actual stored and deployed policies.
\end{theorem}
\begin{proof}
\emph{Continuation.} For a tail $\rho$, put $h_t=\rho_t-\pi_t^\star$ and $D_t=V_t^\rho-V_t^\star$. Optimal-action stationarity and Taylor's formula give
\begin{align*}
 A_t^\rho(s)&:=Q_t^\star(s,\rho_t(s))-V_t^\star(s)\\
 &=\int_0^1(1-\alpha)h_t(s)^\top H_{t,\alpha}^\rho(s)h_t(s)\,d\alpha.
\end{align*}
The weighted product bound makes its $C^2$ norm at most $L_t^\star\norm{h_t}_{2,v_t}^2/2$. The Bellman difference
$D_t=Q_t^\star(\cdot,\rho_t)-V_t^\star+\mathsf P_t^\rho D_{t+1}$, $D_T=0$, therefore yields, with $\rho_{>t}=\bar\pi_{>t}$,
\begin{equation}
 \norm{D_{t+1}}_{2,w_{t+1}}
 \le\mathcal R_t=\sum_{j>t}\frac{L_j^\star}{2}
 \Bigl(\prod_{k=t+1}^{j-1}C_k^{\rm tr}\Bigr)E_j^2.
 \label{eq:nlmaintail}
\end{equation}
The action-gradient and Hessian discrepancies are at most $\Gamma_{g,t}\mathcal R_t$ and $\Gamma_{H,t}\mathcal R_t$.

\emph{Correction and acceptance.} Suppress $t$ and set $q=Q_t^\rho(s_i,\cdot)$, $u^*=\pi_t^\star(s_i)$. If $\Gamma_H\mathcal R_t\le\mu$, then $\mu I\preceq q''\preceq LI$ on the ball, where $L=L_*+\Gamma_H\mathcal R_t$. Taylor expansion bounds the raw Newton proposal $n=u^0+d$ by $\norm{n-u^*}\le b_N$, where
\begin{equation}
 b_N=
 \frac{L_h(\epsilon^0)^2/2+\delta_H\epsilon^0+\delta_g+
 \delta_{\rm lin}+\Gamma_g\mathcal R_t}{\widehat m},
 \label{eq:nlmainnode}
\end{equation}
where $L_h$ is the Hessian Lipschitz constant. The shaped proposal $p$ has error at most $b_N+\eta_{\rm node}$. If $\Gamma_g\mathcal R_t<\mu r_B$, the local minimizer $x_\rho$ is interior and $\norm{x_\rho-u^*}\le c=\Gamma_g\mathcal R_t/\mu$. The paired comparison accepts $p$ when $\widehat\Delta<0$ and otherwise returns $u^0$. For $u^0,p$ in the ball, the resulting label $y$ satisfies $q(y)\le q(p)+\tau$, including rejection. With $\kappa=\sqrt{L/\mu}$, strong convexity and smoothness imply
\[
 \norm{y-u^*}\le b=\kappa(b_N+\eta_{\rm node})
 +(\kappa+1)c+\sqrt{2\tau/\mu}.
\]

\emph{Storage and closure.} Substitute this target bound into \eqref{eq:nlmainstorage}, then substitute \eqref{eq:nlmaintail}. Collect coefficients to obtain \eqref{eq:nonlinearbackward}; Technical Appendix~\ref{app:nltargetstorage}--\ref{app:nlclosure} gives constants and storage algebra. Prescribe envelopes and radii $\overline E_t$ before construction. Descending in time, insert later radii in $\mathcal R_t$, check the curvature/interiority conditions and $\epsilon_t^0,b_{N,t}+\eta_{\rm node,t}\le r_{B,t}$, and require the explicit storage bound to be at most $\overline E_t$, preserving the envelopes. This verifies the local regime in the norm used at earlier dates.
\end{proof}

\begin{corollary}[A consistent second-order regime]\label{cor:nonlinearquadratic}
Under Theorem~\ref{thm:nonlinearbackward}, let the constants and inverse-curvature margins be uniform as $\epsilon\to0$, preserve the regularity class, and assume $\epsilon_t^0=O(\epsilon)$. If $\delta_{H,t}=O(\epsilon)$, $\tau_t=O(\epsilon^4)$, and
\[
 \delta_{g,t}+\delta_{\rm lin,t}+\eta_{\rm node,t}
 +F_t^{\rm reg}+\eta_{\rm dep,t}=O(\epsilon^2),
\]
then $E_t=O(\epsilon^2)$ at every date for sufficiently small $\epsilon$ within the local regime.
\end{corollary}
\begin{proof}
At $T-1$, $\mathcal R_{T-1}=0$ gives $E_{T-1}=O(\epsilon^2)$. Second-order future errors make \eqref{eq:nlmaintail} fourth order, so backward induction gives the same current-policy rate.
\end{proof}
The same weighted advantage and transition bounds yield
$0\le J_\theta(\bar\pi;s_0)-J_\theta(\pi^\star;s_0)=O(\epsilon^4)$ after one complete backward sweep at each fixed $s_0$. This is a nominal value-accuracy result, not an unconditional improvement claim or a rate for the learned controllers. Hard-cap crossings can violate derivative regularity; the smooth-cap bound is \eqref{eq:nlsmoothcapaccuracy}. Acceptance margins can replace the sufficient $\tau_t=O(\epsilon^4)$ allocation; see \eqref{eq:nlacceptmargin} and Technical Appendix~\ref{app:nlacceptsampling}.

\paragraph{Storage consistency.}
Fixed direct-action ridge need not reproduce an optimal baseline (Proposition~\ref{prop:nlridgefloor}); its consistency defect does not estimate experimental cost loss. The sufficient representation and ridge scaling in \eqref{eq:nlridgeschedule} are not certified by fixed budgets. Reducing ridge can amplify finite-sample target noise. Since optimal-policy error $\epsilon$ is unobserved, these rates describe a sufficient regime rather than an implemented tuning rule; a node-correction-magnitude proxy is possible but untested and does not certify the regime. Section~\ref{sec:matchedstorage} measures storage capacity empirically.

\paragraph{Controlled non-LQG evidence.}
Table~\ref{tab:nlregimecompact} shows second-order residual storage with decreasing ridge, first-order error with fixed residual ridge, and a fixed-direct-ridge floor. The two-decision test uses a fixed smooth perturbation represented in the residual dictionary and the actual smooth stored tail; its decreasing-ridge continuation-gradient order is 3.970. These sampled local orders use bounded intervals (Technical Appendix~\ref{app:nlregime}).
\begin{table}[H]
\centering\small\setlength{\tabcolsep}{3pt}
\caption{Sampled weighted $C^2$ policy-error orders, $\epsilon:0.0125\to0.00625$. Direct storage includes a reference-policy feature.}
\label{tab:nlregimecompact}
\begin{tabular}{lrr}
\toprule
Storage & Date 0 & Date 1\\
\midrule
Residual, $\lambda=\epsilon$ & 1.993 & 1.983\\
Residual, $\lambda=0.05$ & 1.033 & 0.982\\
Direct, $\lambda=0.01$ & $-0.0001$ & $-0.010$\\
\bottomrule
\end{tabular}
\end{table}

\subsection{Local action correction and alignment}
\begin{proposition}[Inexact metric correction]\label{prop:inexactmetric}
Fix a continuation $\rho$ and $q(u)=\Qfun_t^\rho(s,u)$ with $L_\nabla$-Lipschitz gradient on a convex neighborhood of the feasible segment $[u^0,u^0+d]$. For symmetric $\mathsf M\succ0$ and $\norm{\widehat g-\nabla q(u^0)}\le\delta$, put
\begin{equation}
 \xi=\max\{0,\widehat g^\top d+\tfrac12d^\top\mathsf M d\}.
 \label{eq:metricresidual}
\end{equation}
For $\alpha\in[0,1]$, $q(u^0+\alpha d)-q(u^0)\le b^{\rm met}(\alpha)$, where
\begin{align}
 b^{\rm met}(\alpha)&=\alpha(\xi-\tfrac12d^\top\mathsf M d+\delta\norm d)
 \nonumber\\
 &\quad+\tfrac12L_\nabla\alpha^2\norm d^2.
 \label{eq:metricbudget}
\end{align}
If $d\ne0$ and $\xi+\delta\norm d<d^\top\mathsf M d/2$, small positive $\alpha$ gives strict descent.
\end{proposition}
Thus positive metrics can support descent without estimating the population Hessian. Feasible metric-QP and radially shortened unconstrained steps have $\xi=0$; further projection may not.

\begin{proposition}[Action alignment]\label{prop:alignment}
Let $\U$ be nonempty, closed and convex, $q_\star=\Qfun_t^{\theta,\star}(s,\cdot)$ be $\mu$-strongly convex with minimizer $u_\theta^\star$, and $q_\star,q_\rho=\Qfun_t^\rho(s,\cdot)$ have differentiable neighborhood extensions. For feasible $v$, put $\varepsilon_{\rm cont}^{\rho}(v)=\norm{\nabla q_\rho(v)-\nabla q_\star(v)}$ and
\begin{equation}
 r_\rho(v)=\operatorname{dist}(0,\nabla q_\rho(v)+N_\U(v)),
 \label{eq:alignmenterrors}
\end{equation}
where $N_\U(v)$ is the convex normal cone. Then
\begin{equation}
 \norm{v-u_\theta^\star}\le
 (r_\rho(v)+\varepsilon_{\rm cont}^{\rho}(v))/\mu.
 \label{eq:alignmentbound}
\end{equation}
No convexity of $q_\rho$ is required.
\end{proposition}
The bound uses the population residual after storage and safeguards. Exact selected-tail minimization may retain continuation error. Technical Appendix~\ref{app:metricimplementation} and~\ref{app:derivativealignment} give both proofs and the deployment and continuation identities \eqref{eq:metricbackwardbudget} and~\eqref{eq:continuationgradientgap}.

\subsection{Affine LQG as an exact reference}\label{sec:lqgreference}
For a known deterministic context schedule, take
\begin{align}
 X_{t+1}&=AX_t+Bu_t+b_t+L\xi^X_{t+1},\nonumber\\
 Y_{t+1}&=H_tX_{t+1}+D_t\nu_{t+1},\label{eq:lqgmodel}\\
 \ell_t(X,u)&=\tfrac12X^\top C_t^xX+\tfrac12u^\top C_t^uu,
 \label{eq:lqgcost}
\end{align}
with $\Phi(X)=X^\top C_T^xX/2$, action-independent noise covariances, $C_t^x,C_T^x\succeq0$ and $C_t^u\succ0$. Fitted coefficients define the nominal Riccati reference \eqref{eq:lqgrecursion}; plant coefficients define an evaluator-only oracle.

\begin{proposition}[Exact backward recovery in fitted affine LQG]\label{prop:lqgbackwardexact}
Fix this model, finite horizon and initial covariance, with finite coefficients and covariances and unrestricted actions. At each date, apply one exact Newton step under the constructed tail from any finite $u_t^0(m)$:
\begin{equation}
 v_t(m)=u_t^0(m)-B_t^{-1}g_t,
 \label{eq:lqgbackwardnewton}
\end{equation}
where $g_t,B_t$ are the population action gradient and Hessian at $(m,u_t^0(m))$. Without damping, curvature modification or clipping, exact affine storage (or unregularized full-rank affine least squares) yields
\begin{equation}
 \bar\pi_t(m)=-K_tm-k_t=u_t^{\theta,\star}(m),\quad t<T,
 \label{eq:lqgbackwardexact}
\end{equation}
for all $m$ after one backward sweep. Exact nonincrease acceptance may be included.
\end{proposition}
\begin{proof}
The last action value is strictly convex quadratic, so Newton recovers its affine minimizer and exact storage preserves it. An optimal future tail gives the same property one date earlier. Induction proves the result; Technical Appendix~\ref{proof:prop:lqgbackwardexact} supplies details.
\end{proof}
For a frozen affine baseline tail, Theorem~\ref{thm:lqgrefinement} in Technical Appendix~\ref{app:theoryproofs} instead bounds corrected coefficient error by $C_\delta\epsilon^2$ for initial error $\epsilon\le\delta$, paralleling classical Newton/policy improvement \citep{kleinman1968,hewer1971,puterman1979,bertsekas2022}. Constants need not be uniform in horizon or dimension. Refreshed-tail propagation and realization errors are in \eqref{eq:refreshedaffinebudget} and Corollary~\ref{cor:deployment}.

\subsection{Model transfer and evaluation}
For plant cost $J_\star$ of the complete controller, including its fitted filter, suppose $|J_\star(\pi^j)-J_\theta(\pi^j)|\le e_j$. These bounds guarantee plant improvement whenever the nominal improvement exceeds $e_0+e_1$. They require separate evidence (Technical Appendix~\ref{app:transferaccounting}); fitted covariance does not measure parameter uncertainty.

Plant cost uses each deployed policy's own observation history. Action MSE uses common histories and either the fitted optimum or an evaluator-only plant oracle. Policy and model/filter discrepancies have a cross term, so their squared errors are not additive; see \eqref{eq:erroridentity}. Conditional action diagnostics concern their declared continuation and independent rollout noise, and do not replace deployed-policy evaluation (Online Supplement~\ref{sec:rocketdiagnostics}).

\section{Experiments}\label{sec:experimentsoverview}
The experiments test inexpensive executable feedback in thrust, active-sensing docking, and a nonlinear arm. Each task uses eight models fitted to action--observation data without latent-state labels, with supplied model families and priors. Filtering is exact within fitted thrust/docking models; the arm uses an approximate EKF.

Open-loop feedback (OLF) MPC optimizes action sequences and replans after observations; restricted feedback MPC (FB-MPC in tables) optimizes future belief-dependent actions. Neither uses DPO/BG initialization or continuation. Costs use paired physical trajectories and each controller's own belief history. Unless stated otherwise, intervals are exploratory 95\% Student-$t$ intervals over eight fitted-model seed means. Map families were developed after earlier diagnostics and selected by fitted-model target actions; fresh episodes reuse the training seeds. Full protocols and solver budgets are in Online Supplement~\ref{sup:completeexperiments}.

\paragraph{Reading the time columns.}
Timing markers: $\mathrm C$ denotes date-zero native medians across eight seed medians on one Ryzen~9~7950X3D thread; $\mathrm C_0$ denotes seed~0. $\mathrm G$ denotes archived cold-start A6000 PyTorch planners at dates $0,8,15$; $\mathrm P$, a seed-0 date-zero docking PyTorch CPU pilot; $\mathrm A$, acados CPU execution with panel-specific validation. All exclude filtering, fitting, and construction. Preloading, parity, and audit scopes are detailed in Online Supplement~\ref{sup:currentnative}. Cross-marker quotients are not algorithmic speedups.

\paragraph{Online planning and executable feedback.}
Figure~\ref{fig:introcostlatency} compares costs and available latencies of the original task policies; Figure~\ref{fig:matchedstorage} gives the separate common-map comparison.

\begin{figure*}[t]
\centering
\includegraphics[width=\linewidth]{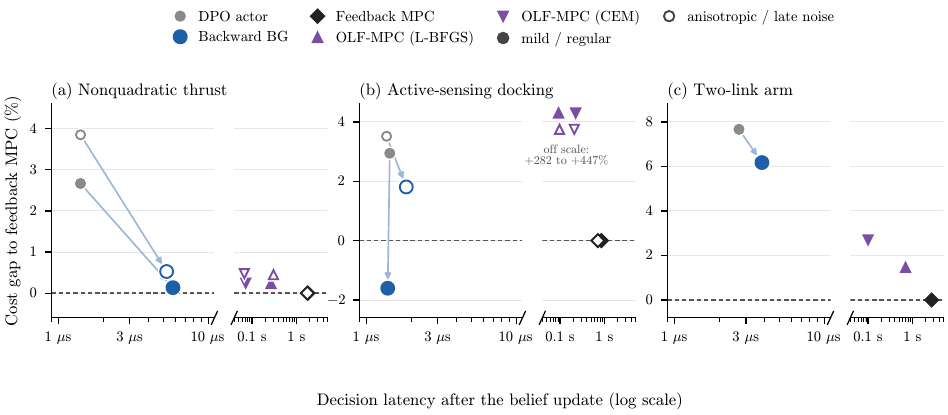}
\caption{\small Original task policies: eight-seed plant-cost gaps $100(J/J_{\mathrm{FB}}-1)$ versus online feedback MPC and available action latencies. Arrows connect DPO to backward BG; filled/open symbols denote mild/anisotropic thrust and regular/late-noise docking. Docking OLF costs are off scale. Timings exclude filtering and follow the task tables' implementation scopes; their quotients are not algorithmic speedups. Historical students appear in Online Supplement Figure~\ref{fig:historicaldeployment}; Figure~\ref{fig:matchedstorage} gives the separate common-map comparison.}
\label{fig:introcostlatency}
\end{figure*}

\subsection{Thrust: exact recovery and inexpensive refinement}\label{sec:thrustfamily}
\subsubsection{LQG reference}\label{sec:rocketlqg}
The thrust model has six latent position/velocity/wind states, two actions and $T=16$. Quadratic costs provide exact fitted Riccati feedback. BG reduces fitted-optimal action MSE by over 99.99\%; plant costs match Riccati feedback at reported precision, approximately realizing Proposition~\ref{prop:lqgbackwardexact}. The analytic controller remains cheaper and more accurate when available. Online Supplement~\ref{sup:lqgpanel} gives the full eight-seed reference panel (Table~\ref{tab:compactlqg}); identification and training are in Section~\ref{sec:rocketprotocol}.

\subsubsection{Nonquadratic cost and the value of local calculation}\label{sec:rocketnonquadratic}
Quartic position and coupled-action penalties remove the Riccati solution. BG stores target-validated maps and rebuilds earlier corrected tails. In Table~\ref{tab:compactthrust}, backward BG lowers DPO cost by 2.47\%/3.20\%; its advantage over the equally enriched static map repeats on fresh paired episodes (Online Supplement~\ref{sup:regressionrefit}).

\begin{table}[!tbp]
\centering\small\setlength{\tabcolsep}{3pt}
\caption{Nonquadratic thrust: eight-seed plant costs; paired cells stack mild above anisotropic. BG uses target-validated maps. Timing: native $\mathrm C$ (eight seeds), $\mathrm C_0$ (seed~0); A6000 planners $\mathrm G$. Fixed-belief acados results have no episode costs (Online Supplement Table~\ref{tab:currentacados}).}
\label{tab:compactthrust}
\begin{tabular}{lrr}
\toprule
Controller  &  Plant cost  &  Action time\\
\midrule
DPO  & \shortstack[r]{1.697575\\2.221337} & \shortstack[r]{$1.407$\\$1.409$}$\,\mu\mathrm s^{\mathrm C}$\\
DPO+150  & \shortstack[r]{1.698865\\2.227725} & \shortstack[r]{$1.448$\\$1.421$}$\,\mu\mathrm s^{\mathrm C}$\\
Static BG  & \shortstack[r]{1.658031\\2.156786} & \shortstack[r]{$5.587$\\$5.385$}$\,\mu\mathrm s^{\mathrm C_0}$\\
Backward BG  & \shortstack[r]{1.655677\\2.150216} & \shortstack[r]{$5.820$\\$5.280$}$\,\mu\mathrm s^{\mathrm C_0}$\\
OLF, L-BFGS  & \shortstack[r]{1.657344\\2.148585} & \shortstack[r]{$269.205$\\$305.617$}$\,\mathrm{ms}^{\mathrm G}$\\
OLF, CEM  & \shortstack[r]{1.657510\\2.149064} & \shortstack[r]{$72.472$\\$67.529$}$\,\mathrm{ms}^{\mathrm G}$\\
FB-MPC  & \shortstack[r]{1.653464\\2.138939} & \shortstack[r]{$1826.580$\\$1771.013$}$\,\mathrm{ms}^{\mathrm G}$\\
\bottomrule
\end{tabular}
\end{table}

BG is within 0.13\%/0.53\% of tested restricted feedback MPC; paired intervals are in Online Supplement~\ref{sup:regressionrefit}. Anisotropic OLF planners remain modestly better; mild OLF contrasts are small. Section~\ref{sec:matchedstorage} compares executable feedback-MPC students under common storage.

\paragraph{Checking a dedicated native solver.}
At fixed held-out beliefs, acados solves the same Gaussian-moment objectives. OLF SQP converges; capped feedback SQP reaches its iteration limit. These fixed-belief results supply no closed-loop costs (Online Supplement Table~\ref{tab:currentacados}).

\subsection{Active-sensing docking}\label{sec:dockingmain}
Docking couples bounded thrust to sensing power that reduces the next position-observation variance. Its three-state Gaussian model tracks position, velocity, and hidden disturbance over 12 decisions; late noise increases observation noise in the second half. Success imposes terminal position/velocity tolerances. BG fits physical actions and rebuilds earlier corrected tails (Online Supplement~\ref{sup:dockingprotocol}).

With future thrust fixed during optimization, OLF assigns no value to sensing and chooses zero power. DPO and BG retain observation-dependent actions; restricted feedback MPC uses tanh-affine belief-mean thrust and a datewise sensing schedule. A validated 128-path acados OLF solver succeeds in all 98,304 calls yet retains the zero-sensing continuation's poor performance (Table~\ref{tab:compactdocking}).

\begin{table}[!tbp]
\centering\small\setlength{\tabcolsep}{3pt}
\caption{Docking: eight-seed means; paired cells stack regular above late noise. BG uses physical-action fits. Timing: native $\mathrm C$ (eight seeds), $\mathrm C_0$ (seed~0); validated acados $\mathrm A$; seed-0 CPU planner pilot $\mathrm P$.}
\label{tab:compactdocking}
\begin{tabular}{lrrr}
\toprule
Controller  &  Plant cost  &  Success (\%)  &  Action time\\
\midrule
DPO  & \shortstack[r]{0.8603\\4.0925} & \shortstack[r]{81.35\\48.05} & \shortstack[r]{$1.434$\\$1.366$}$\,\mu\mathrm s^{\mathrm C}$\\
DPO+extra  & \shortstack[r]{0.8546\\4.0849} & \shortstack[r]{81.79\\48.41} & \shortstack[r]{$1.493$\\$1.404$}$\,\mu\mathrm s^{\mathrm C}$\\
Static BG  & \shortstack[r]{0.8249\\4.0318} & \shortstack[r]{83.13\\50.61} & \shortstack[r]{$1.374$\\$1.399$}$\,\mu\mathrm s^{\mathrm C_0}$\\
Backward BG  & \shortstack[r]{0.8223\\4.0250} & \shortstack[r]{82.57\\50.68} & \shortstack[r]{$1.389$\\$1.848$}$\,\mu\mathrm s^{\mathrm C_0}$\\
OLF, L-BFGS  & \shortstack[r]{4.5741\\15.1212} & \shortstack[r]{35.60\\26.15} & \shortstack[r]{$92.901$\\$97.753$}$\,\mathrm{ms}^{\mathrm P}$\\
OLF, CEM  & \shortstack[r]{4.5660\\15.1397} & \shortstack[r]{36.01\\26.10} & \shortstack[r]{$225.704$\\$209.911$}$\,\mathrm{ms}^{\mathrm P}$\\
OLF, acados  & \shortstack[r]{4.5675\\15.1203} & \shortstack[r]{35.60\\26.07} & \shortstack[r]{$0.827$\\$0.822$}$\,\mathrm{ms}^{\mathrm A}$\\
FB-MPC  & \shortstack[r]{0.8357\\3.9535} & \shortstack[r]{82.42\\49.98} & \shortstack[r]{$855.195$\\$713.144$}$\,\mathrm{ms}^{\mathrm P}$\\
\bottomrule
\end{tabular}
\end{table}

Backward BG reduces DPO cost by $0.03796$/$0.06746$ and raises success by 1.22/2.63 points in regular/late-noise docking. The regular cost gain repeats on fresh episodes; the late-noise change is unresolved. Feedback MPC costs $0.01335$ $[0.00455,0.02215]$ more than BG in regular docking and $0.07158$ $[0.01710,0.12606]$ less under late noise. The fitting loss thus affects the comparison, while DPO already supplies the sensing-dependent continuation missing from OLF (Online Supplement~\ref{sup:regressionrefit}).

\paragraph{Where precision is lost.}
A scalar companion admits numerical belief-space DP (Online Supplement~\ref{sup:scalardp}). Backward BG closes 53.9\% of the eight-seed DPO-to-DP mean cost gap. Terminal-node targets are nearly optimal, but regression loses precision; exact terminal-thrust replacement lowers full-policy cost (Online Supplement Table~\ref{tab:scalardpterminal}). This identifies a storage bottleneck at those nodes, without attributing all three-state docking error to regression. A second sweep gives no reliable improvement.

\subsection{Nonlinear robot arm}\label{sec:robotmain}
The two-link arm has four latent angle/velocity states, noisy angles, bounded two-dimensional torque, and 16 decisions. Its EKF is approximate. Policies use belief means, covariance diagonals, goal, and time; rollouts retain full covariance. Costs combine squared/quartic fingertip, velocity, and torque penalties; attainment is a terminal threshold event. Target validation selects residual/Nystr\"om BG maps (Online Supplement~\ref{sup:robotprotocol}).

\begin{table}[!tbp]
\centering\small\setlength{\tabcolsep}{3pt}
\caption{Two-link arm: eight-seed plant cost and attainment, 256 paired episodes/seed. BG uses target-validated maps. Timing: date-zero native $\mathrm C$ (eight seeds, same-process BG audit); archived A6000 planners $\mathrm G$.}
\label{tab:compactarm}
\begin{tabular}{lrrr}
\toprule
Controller & Plant cost & Attain. (\%) & Action time\\
\midrule
DPO &6.77187&48.49&$2.698\,\mu\mathrm s^{\mathrm C}$\\
DPO+150 &6.74886&49.85&$2.734\,\mu\mathrm s^{\mathrm C}$\\
Static BG &6.70148&62.26&$3.659\,\mu\mathrm s^{\mathrm C}$\\
Backward BG &6.67805&59.23&$3.837\,\mu\mathrm s^{\mathrm C}$\\
OLF, L-BFGS &6.38183&66.06&$699.24\,\mathrm{ms}^{\mathrm G}$\\
OLF, CEM &6.45990&64.01&$100.14\,\mathrm{ms}^{\mathrm G}$\\
FB-MPC, $K=0$ &6.36050&66.11&$700.32\,\mathrm{ms}^{\mathrm G}$\\
FB-MPC &6.28979&66.02&$2703.97\,\mathrm{ms}^{\mathrm G}$\\
\bottomrule
\end{tabular}
\par\smallskip\footnotesize Native parity and scope: Online Supplement~\ref{sup:currentnative}.
\end{table}

Backward BG lowers DPO cost by $0.09382$ (interval $[0.07156,0.11609]$) and raises attainment by 10.74 points $[8.65,12.83]$ (Table~\ref{tab:compactarm}); fresh episodes with the same model seeds give cost $6.68845$. Static BG attains 3.03 points more often but costs $0.02343$ more. BG costs 4.64\% more than OLF L-BFGS; feedback MPC lowers BG cost by 5.81\%. Thrust's near-planner performance therefore does not extend to the arm.

Same-noise finite differences give refreshed-tail Hessian discrepancy up to 0.0874, not a population-bias estimate. With nonsmooth caps, branchwise curvature is interpreted as a regularized metric under Proposition~\ref{prop:inexactmetric}; held-out costs assess the deployed controller. The rich-map student also retains a planner gap (Section~\ref{sec:matchedstorage}), motivating investigation of shared storage constraints and visitation mismatch alongside BG correction error without identifying their contributions. Online Supplement~\ref{sup:armstoragechain} gives construction-node teacher-projection and storage diagnostics.

\subsection{Matched policy-storage comparison}\label{sec:matchedstorage}
Figure~\ref{fig:matchedstorage} summarizes the main BG--student comparison. Students store first actions from finite-budget restricted feedback-MPC solves at DPO-visited beliefs; online MPC instead replans along its own trajectories. Historical distinct-map results and teacher-budget sweeps are in Online Supplement Tables~\ref{tab:feedbackdistillmain} and~\ref{tab:feedbackthruststudents}.

We cross backward BG and student construction with compact/rich common map families. Models, DPO actors, visitation nodes, action bounds/transforms, fitting loss, map-selection rule, and evaluation paths are matched within each task. BG is rebuilt backward with each map's stored future policy; only MPC first actions are distilled, with teacher labels shared between student maps. Thrust/arm validation may select different realized rich features from the common candidate class. This matches storage protocols while retaining different target generators and finite budgets; map changes also alter BG continuations and targets. Historical absolute costs are not paired contrasts with this evaluation.

\begin{table}[!tbp]
\centering\small\setlength{\tabcolsep}{3pt}
\caption{Matched storage: eight-seed plant cost. Thrust uses rebuilt maps and 512 episodes/seed, distinct from Table~\ref{tab:compactthrust}'s 256-episode panel. Arm Eval.~1/2 are archived streams 12/22 on the same seeds. Figure~\ref{fig:matchedstorage} and Online Supplement~\ref{sup:matchedstorage} give paired contrasts.}
\label{tab:matchedstoragefour}
\begin{tabular}{lrrrr}
\toprule
& \multicolumn{2}{c}{Compact map} & \multicolumn{2}{c}{Rich map}\\
Task / condition & BG & Student & BG & Student\\
\midrule
Thrust, mild & 1.629484 & 1.629304 & 1.627984 & 1.627820\\
Thrust, aniso. & 2.086755 & 2.087783 & 2.083432 & 2.084063\\
Arm, Eval. 1 & 6.714405 & 6.737905 & 6.678051 & 6.687144\\
Arm, Eval. 2 & 6.717764 & 6.744071 & 6.688447 & 6.699184\\
Docking, regular & 0.822330 & 0.828730 & 0.818491 & 0.820978\\
Docking, late & 4.025040 & 4.050587 & 4.027685 & 4.059673\\
\bottomrule
\end{tabular}
\end{table}

The rich class reduces the student-minus-BG mean cost gap by 59--61\% in the two arm streams and 61\% in regular docking (Table~\ref{tab:matchedstoragefour}, Figure~\ref{fig:matchedstorage}); all three paired gap-change intervals exclude zero. Rich-map arm still favors BG in cost on one stream and attainment on both; regular docking retains a small BG attainment advantage. Thrust resolves no method difference in either map, without establishing equivalence. Late-noise docking favors BG in both maps, but the map effect remains unresolved. Intervals are exploratory and unadjusted for multiplicity.

Rich-map incremental construction takes BG versus teacher-plus-student fit 9.34/82.78 seconds in arm, 5.09/313.39 in regular docking, and 4.99/298.16 under late noise on the same local GPU. These budget-specific times exclude shared model/actor training, node preparation, evaluation, and export. Thrust reuses teachers from another GPU, precluding a matched total-time ratio. Matched rich-map date-zero native medians are $2.49$--$5.46\,\mu$s over eight seeds/condition, excluding filtering and per-action Python overhead (Online Supplement Table~\ref{tab:matchednative}). Full paired contrasts and timing scopes appear in Online Supplement~\ref{sup:matchedstorage}.

\section{Discussion and conclusion}
BG refines current actions through deployed future feedback and stores the corrections. The nonlinear recursion tracks continuation, derivative, safeguard and representation errors; the controlled non-LQG example demonstrates consistent second-order policy accuracy. Affine LQG provides exact backward recovery.

Thrust BG approaches the tested feedback planner; docking shows the value of future observation response and sensitivity to physical-action fitting. The arm retains a planner gap. Richer matched maps narrow student--BG differences in arm and regular docking; late-noise docking favors BG, while thrust resolves no difference. Scalar DP diagnoses precision lost in regression. 

\paragraph{Scope.}
The tasks use supplied low-dimensional model families, two actions and 12--16 decisions; the arm EKF is approximate. Piecewise-smooth caps motivate Proposition~\ref{prop:inexactmetric}; fixed ridge and cap crossings preclude an implementation-wide rate certificate from Corollary~\ref{cor:nonlinearquadratic}. Finite-node checks certify neither all-state accuracy nor plant transfer. Full rollouts cost $O(NMT^2)$ model steps before derivative/storage overhead; large-model and long-horizon scaling is untested.

\begin{ack}
Jeonggyu Huh acknowledges financial support from the National Research Foundation of Korea (NRF; grant no.~RS-2025-00562904).

\end{ack}

\section*{Declaration of generative AI and AI-assisted technologies}
Under the author's direction, ChatGPT/Codex (OpenAI) and Claude (Anthropic) assisted idea/argument examination and manuscript editing; ChatGPT/Codex also assisted implementation and debugging (Online Supplement~\ref{sup:aidevelopment}). The author retains responsibility for all content.

\onecolumn
\noindent{\Large\bfseries Technical Appendix}\par\medskip

\appendix
\numberwithin{table}{section}
\numberwithin{figure}{section}
This appendix supplies the mathematical references used in the main paper. Section~\ref{app:supporting} proves the local and backward comparisons, Section~\ref{app:theory} develops finite-horizon LQG error propagation, Section~\ref{app:nonlinear} proves the nonlinear stored-policy recursion and its consistency limits, Section~\ref{app:gaussian} states the regularity and belief conditions, and Section~\ref{app:nlregime} tests the local non-LQG accuracy regime. Main-paper equation and result numbers are retained.

\section{Proofs of the main comparisons}\label{app:supporting}
The Gaussian closure proof is in Section~\ref{proof:prop:closure}. Section~\ref{app:derivativealignment} establishes conditional derivative and action-alignment bounds. Section~\ref{app:baselinecomparison} records the baseline-continuation comparison, while Section~\ref{proof:thm:backward} proves the backward comparison and its regression-error bound. The implemented metric safeguards are covered in Section~\ref{app:metricimplementation}; Section~\ref{proof:prop:lqgbackwardexact} proves exact affine LQG backward recovery. Section~\ref{app:transferaccounting} collects model-transfer and action-error accounting.

\subsection{Proof of Proposition~\ref{prop:closure}}\label{proof:prop:closure}
\paragraph{Gaussian filtering and innovations likelihood.}
For the controlled model \eqref{eq:controlledmodel}, substitute the corresponding coefficients evaluated at the known $(C_t,u_t)$ for $A,f_t,L_\theta(C_t),H_t,D_\theta(C_t)$ below; the same conditioning formulas apply.

For completeness, the action-independent core model and its standard filtering calculations are recorded here. Fresh $\xi^X,\nu$ are independent standard Gaussians, $f_t=f_\theta(C_t)$, $H_t=H_\theta(C_t)$, $X_t\mid\F_t\sim\N(m_t,P_t)$, and $L_{Y,t+1}=\chol(\Omega_{t+1})$.
\begin{align}
X_{t+1} &= A X_t+f_\theta(C_t)+L_\theta(C_t)\xi^X_{t+1},
\qquad \xi^X_{t+1}\sim\N(0,I_d), \label{eq:state}\\
Y_{t+1} &= H_\theta(C_t)X_{t+1}
                 +D_\theta(C_t)\nu_{t+1},
\qquad \nu_{t+1}\sim\N(0,I_m), \label{eq:obs}\\
Q_t&=L_\theta(C_t)L_\theta(C_t)^\top,
\qquad R_t=D_\theta(C_t)D_\theta(C_t)^\top\succ0. \label{eq:covs}
\end{align}
\begin{align}
a_{t+1}&=A m_t+f_t,
& P^-_{t+1}&=A P_tA^\top+Q_t, \label{eq:pred}\\
\mu_{t+1}&=H_ta_{t+1},
& \Omega_{t+1}&=H_tP^-_{t+1}H_t^\top+R_t. \label{eq:innovcov}
\end{align}
\begin{align}
K_t^{\mathrm{KF}}&=P^-_{t+1}H_t^\top\Omega_{t+1}^{-1}, \label{eq:kfgain}\\
m_{t+1}&=a_{t+1}+K_t^{\mathrm{KF}}(Y_{t+1}-\mu_{t+1}), \label{eq:meanupdate}\\
P_{t+1}&=P^-_{t+1}-K_t^{\mathrm{KF}}\Omega_{t+1}(K_t^{\mathrm{KF}})^\top. \label{eq:covupdate}
\end{align}
\begin{equation}
p_\theta(Y_{t+1}\mid\F_t,u_t)
=\N(Y_{t+1};\mu_{t+1},\Omega_{t+1}).
\label{eq:predictivedensity}
\end{equation}
\begin{equation}
\mathcal L_{\mathrm{pred}}(\theta)
=\frac12\sum_t\left[
\log\det\Omega_{t+1}
+e_{t+1}^\top\Omega_{t+1}^{-1}e_{t+1}
+m\log(2\pi)\right]
+\lambda_{\mathrm{reg}}\mathcal R(\theta),
\quad e_{t+1}=Y_{t+1}-\mu_{t+1}.
\label{eq:nll}
\end{equation}
\begin{align}
Y_{t+1}&=\mu_{t+1}+L_{Y,t+1}E_{t+1}, \label{eq:predsampler}\\
m_{t+1}&=a_{t+1}+K_t^{\mathrm{KF}}L_{Y,t+1}E_{t+1}. \label{eq:meansampler}
\end{align}
The predictive likelihood differentiates through the filter and recurrent predictor state. The observation score alone need not be the joint likelihood of predictors and observations. The final two equations use fresh $E_{t+1}\sim\N(0,I_m)$; latent-state-dependent scores can instead be evaluated by jointly simulating the latent belief, physical dynamics and observations.

\begin{proof}
Conditional on $\F_t$ and the chosen action, the transition is affine in a Gaussian state with independent Gaussian noise. Hence $(X_{t+1},Y_{t+1})$ is jointly Gaussian. Conditioning on $Y_{t+1}$ gives the displayed mean and covariance. Assumption~\ref{ass:context} ensures that subsequently observing $C_{t+1}$ introduces no omitted likelihood factor. Induction proves the result.
\end{proof}

\subsection{Conditional derivatives and action alignment}\label{app:derivativealignment}
Derivative statements condition on pre-sampling $\mathcal G_t$; performance identities condition on completed policies and fresh evaluation paths.

\begin{proof}[Proof of Proposition~\ref{prop:clderivatives}]
On a smaller neighborhood, the mean-value formula bounds a first difference quotient of $\mathscr C_t^\rho$ by its local gradient envelope. Conditional dominated convergence gives the first interchange in \eqref{eq:directderivatives}. Applying the same argument to the pathwise gradient, using the Hessian envelope, gives the second. The differentiated composition includes the state derivative of every future continuation action, although model and policy parameters are fixed. The common innovation law is independent of the differentiated argument; all action-dependent conditional laws enter through the simulator.
\end{proof}

For a vector $D(\omega)$ containing the selected gradient and Hessian entries at a fixed $(t,s,u)$, suppose $\E[\norm{D}^2\mid\mathcal G_t]<\infty$. Fresh conditionally independent rollout units give
\begin{equation}
\E[\norm{\widehat d_M-d}^2\mid\mathcal G_t]=\frac{\E[\norm{D-d}^2\mid\mathcal G_t]}{M},
\qquad d=\E[D\mid\mathcal G_t],\quad \widehat d_M=M^{-1}\sum_{i=1}^M D_i.
\label{eq:derivativemse}
\end{equation}
Conditional centering removes cross terms. For antithetic simulation, $M$ counts independent pair averages. This pointwise identity for the declared discrete model implies neither a uniform deployment bound nor unbiasedness after inversion, flooring, clipping, or acceptance. In particular, unbiased derivative estimates need not give an unbiased Newton step or justify evaluation at a data-dependent action.

\begin{proof}[Proof of Proposition~\ref{prop:alignment}]
The minimizing variational inequality gives $n_\star=-\nabla q_\star(u_\theta^\star)\in N_\U(u_\theta^\star)$. The normal cone is monotone. For $e=v-u_\theta^\star$ and any $n\in N_\U(v)$, strong convexity implies
\[
\mu\norm{e}^2\leq
\bigl(\nabla q_\star(v)-\nabla q_\star(u_\theta^\star)+n-n_\star\bigr)^\top e
=\bigl(\nabla q_\star(v)+n\bigr)^\top e.
\]
Add and subtract $\nabla q_\rho(v)$, apply Cauchy--Schwarz, and minimize over $n$. If $e=0$ the bound is immediate; otherwise divide by $\norm{e}$. For \eqref{eq:continuationgradientgap}, subtract the two action values: their stage costs cancel and their transition laws coincide. Differentiating the remaining expectation gives the displayed identity. This proof also covers any feasible action satisfying selected-tail stationarity, even when $q_\rho$ is nonconvex.
\end{proof}
Put $\Delta V=V_{t+1}^\rho-V_{t+1}^{\theta,\star}$ and $F(u,\omega)=\mathcal T_{\theta,t}(s,u,\omega)$. When differentiation under the expectation is justified,
\begin{equation}
 \nabla q_\rho(v)-\nabla q_\star(v)
 =\E[F_u(v,\omega)^\top\nabla\Delta V(F(v,\omega))].
 \label{eq:continuationgradientgap}
\end{equation}
This follows by differentiating the common-transition expectation of $\Delta V$. Exact minimization of a selected-tail action value need not remove this discrepancy, and backward refresh need not decrease it monotonically.

\paragraph{What a sampled, safeguarded Newton step leaves unresolved.}
At the fixed state, let $g=\nabla q_\rho(u^0)$ and $B=\nabla^2q_\rho(u^0)$. Suppose the Hessian is $L$-Lipschitz on a neighborhood of the segment from $u^0$ to the feasible returned action $v=u^0+h$. Let $\widehat g,\widehat B$ be the derivative estimates and $\widetilde B$ the matrix actually used after symmetrization and regularization. Put
\begin{align}
\epsilon_g&=\norm{\widehat g-g},\qquad
\epsilon_B=\norm{\widehat B-B}_{\rm op},\qquad
\epsilon_{\rm reg}=\norm{\widetilde B-\widehat B}_{\rm op},\nonumber\\
\zeta(v)&=\operatorname{dist}\bigl(0,\widehat g+\widetilde B h+N_\U(v)\bigr).
\end{align}
The integral Taylor remainder gives the deterministic bound
\begin{equation}
r_\rho(v)\leq\epsilon_g+(\epsilon_B+\epsilon_{\rm reg})\norm{h}
+\zeta(v)+\tfrac L2\norm{h}^2.
\label{eq:newtonresidualbudget}
\end{equation}
The Taylor remainder satisfies $\norm{\nabla q_\rho(v)-g-Bh}\leq L\norm{h}^2/2$; adding estimation and regularization errors and minimizing over the normal cone gives the bound. An exact QP solve on the original feasible set has $\zeta=0$, but damping, clipping, or an added trust constraint can leave a nonzero original-action residual. All terms vanish for an unconstrained quadratic with exact unmodified derivatives and solve. Combining \eqref{eq:newtonresidualbudget} and \eqref{eq:alignmentbound} gives the alignment budget under these error and smoothness conditions.

\subsection{Baseline continuation and descent}\label{app:baselinecomparison}\label{proof:prop:descent}\label{proof:thm:improve}\label{proof:cor:approx}

\begin{proposition}[The exact local direction is a descent direction]\label{prop:descent}
Fix an admissible continuation $\rho$ and take $g_t=\nabla_u\Qfun_t^\rho(s,u^0)$. In \eqref{eq:bgqp}, allow $\widetilde B_t$ to be any symmetric positive-definite matrix; it need not equal or approximate the Hessian. Suppose $\U_t(s)$ is convex, contains $u^0$, and the quadratic minimizer exists. If $d_t\ne0$, then $g_t^\top d_t<0$. If $\Qfun_t^\rho(s,\cdot)$ is continuously differentiable near the feasible segment, sufficiently small positive steps satisfy \eqref{eq:armijo} for that same continuation.
\end{proposition}
\begin{proof}
The zero correction is feasible. Optimality therefore gives
$g_t^\top d_t+\tfrac12d_t^\top\widetilde B_t d_t\leq0$.
Positive definiteness makes $g_t^\top d_t$ strictly negative when $d_t\ne0$. The first-order expansion along the feasible segment proves the Armijo statement.
\end{proof}

\begin{theorem}[Policy comparison using the baseline continuation]\label{thm:improve}
Assume the finite-horizon costs are integrable and the model and state construction are fixed. Define
\begin{equation}
\Adv_t^0(s,u)=\Qfun_t^0(s,u)-V_t^0(s).
\label{eq:advantage}
\end{equation}
For an admissible feedback policy $\pi^1$, deterministic or randomized with fresh internal randomness independent of future transition noise conditional on the current state, let $U_t\sim\pi_t^1(\cdot\mid s_t)$. Then
\begin{equation}
J_\theta(\pi^1;s_0)-J_\theta(\pi^0;s_0)
=\E_\theta^{\pi^1}\left[\sum_{t=0}^{T-1}
\Adv_t^0(s_t,U_t)\right].
\label{eq:performance}
\end{equation}
Consequently, if the exact acceptance test ensures
$\E_{U\sim\pi_t^1(\cdot\mid s)}[\Adv_t^0(s,U)]\leq0$ at every relevant state, then
$J_\theta(\pi^1;s_0)\leq J_\theta(\pi^0;s_0)$.
\end{theorem}
\begin{proof}
Under $\pi^1$, the conditional expectation of
$\ell_t(s_t,U_t,\omega_{t+1})+V_{t+1}^0(s_{t+1})-V_t^0(s_t)$
given $(s_t,U_t)$ is $\Adv_t^0(s_t,U_t)$. Summing and taking expectations telescopes the value terms. Since $V_T^0=\Phi$ and $V_0^0(s_0)=J_\theta(\pi^0;s_0)$, the identity follows. The sign condition gives the inequality.
\end{proof}

Online correction with independent simulation noise is randomized. For offline policies, condition on the completed construction $\mathcal C_{\rm build}$ and use fresh evaluation trajectories; an outer construction expectation is then permitted under integrability. This differs from the pre-sampling $\mathcal G_t$ conditioning in Proposition~\ref{prop:clderivatives}. Equation~\eqref{eq:performance} is the finite-horizon performance-difference identity \citep{kakade2002}. It also holds for backward-constructed $\bar\pi$, but corrected-tail acceptance does not imply its $\Qfun^0$ sign condition. The main paper's Theorem~\ref{thm:backward} addresses that continuation directly.

\begin{corollary}[Approximate action-value comparisons]\label{cor:approx}
Suppose an approximation $\widehat\Qfun_t^0$ has absolute error at most $\varepsilon_t(s)$ at both $u^0$ and an accepted action $u^1$. If
\begin{equation}
\widehat\Qfun_t^0(s,u^1)-\widehat\Qfun_t^0(s,u^0)
\leq-\delta_t(s),
\label{eq:approxaccept}
\end{equation}
let $a_t\in\{0,1\}$ indicate acceptance, and return $u^0$ otherwise. With errors controlled at each realized accepted pair, the deployed policy satisfies
\begin{equation}
J_\theta(\pi^1;s_0)-J_\theta(\pi^0;s_0)
\leq\E_\theta^{\pi^1}\sum_{t=0}^{T-1}
a_t\bigl[2\varepsilon_t(s_t)-\delta_t(s_t)\bigr].
\label{eq:approxbound}
\end{equation}
\end{corollary}
\begin{proof}
At the two accepted actions, the approximation errors add at most $2\varepsilon_t(s)$ to the estimated difference. A rejected proposal returns the baseline and contributes exactly zero advantage. Substitute into \eqref{eq:performance}, including internal randomness when present.
\end{proof}

At the implemented zero margin $\delta_t=0$, this bounds deterioration by $\E^{\pi^1}\sum_t2a_t\varepsilon_t$, without guaranteeing improvement. A finite acceptance sample alone does not control errors along every state visited by the candidate.

\subsection{Backward comparison and regression error}\label{proof:thm:backward}\label{proof:cor:backwarddefect}
\begin{proof}[Proof of Theorem~\ref{thm:backward}]
The policies $\pi^{[t]}$ and $\pi^{[t+1]}$ have the same distribution of $s_t$, generated by the common baseline prefix. Conditional on $s_t$, their costs up to $t-1$ cancel, and their remaining expected-cost difference is $d_t^{\mathrm{back}}(s_t)$ because both use the same corrected tail after date $t$. Hence
\[
 J_\theta(\pi^{[t]};s_0)-J_\theta(\pi^{[t+1]};s_0)
 =\E_\theta^{\pi^0}[d_t^{\mathrm{back}}(s_t)].
\]
Sum over $t=0,\ldots,T-1$. The policy costs telescope to $J_\theta(\bar\pi;s_0)-J_\theta(\pi^0;s_0)$, proving \eqref{eq:backwardperformance}. The sign condition and pointwise upper bounds give the stated inequalities.
\end{proof}

\begin{proof}[Proof of Corollary~\ref{cor:backwarddefect}]
At an accepted target, adding the two approximation errors bounds
$\bar\Qfun_t(s,a_t(s))-\bar\Qfun_t(s,\pi_t^0(s))$ by
$2\varepsilon_t(s)-\eta_t(s)$. At a rejected target that difference is zero.
The remaining change from $a_t(s)$ to $\bar\pi_t(s)$ is at most
$L_t^{Q}(s)e_t^{\mathrm{int}}(s)$; add these bounds.
\end{proof}

\subsection{Inexact metrics and implemented safeguards}\label{app:metricimplementation}
\begin{proof}[Proof of Proposition~\ref{prop:inexactmetric}]
Fix a realized gradient estimate, metric and proposal. The definition of $\xi$ and Cauchy--Schwarz give
\[
 \nabla q(u^0)^\top d
 \leq\widehat g^\top d+\delta\|d\|
 \leq-\tfrac12d^\top\mathsf M d+\xi+\delta\|d\|.
\]
Integrating the gradient along the segment yields the descent-lemma bound
$q(u^0+\alpha d)-q(u^0)\leq
\alpha\nabla q(u^0)^\top d+
\tfrac12L_\nabla\alpha^2\|d\|^2$.
Substitution proves \eqref{eq:metricbudget}. If
$c_d=\tfrac12d^\top\mathsf M d-\xi-\delta\|d\|>0$,
any feasible $0<\alpha\leq1$ with
$L_\nabla\alpha\|d\|^2<2c_d$ gives strict decrease.
For $L_\nabla=0$, every $0<\alpha\leq1$ does.
\end{proof}

This is a realized-error bound. The gradient estimate and metric may use the same noise and need not be independent. Unbiasedness of $\widehat g$ alone does not establish the bound $\delta$, and gradient interchange still needs its own assumptions. Positive definiteness of a floored branchwise matrix supplies a metric, not a bound on population-curvature error. The proposition uses only a Lipschitz population gradient; this regularity is weaker than the pathwise $C^2$ interchange hypotheses but is not automatic for a capped continuation.

\paragraph{Which implemented steps have zero residual?}
For an exact constrained quadratic solve with $0$ feasible, comparison with $0$ gives $\widehat g^\top d+\tfrac12d^\top\mathsf M d\leq0$. More generally, a quadratic objective within $\epsilon_{\rm solve}\geq0$ of its feasible minimum gives $\xi\leq\epsilon_{\rm solve}$. For an unconstrained step $v=-\mathsf M^{-1}\widehat g$ and radial shortening $d=\gamma v$, $0\leq\gamma\leq1$,
\begin{equation}
 \widehat g^\top d+\tfrac12d^\top\mathsf M d
 =-\gamma(1-\gamma/2)v^\top\mathsf M v\leq0.
 \label{eq:radialmetricdecrease}
\end{equation}
Thus thrust's eigenvalue floor and radial shortening give $\xi=0$ in exact arithmetic, even when the input curvature is only branchwise. Linear-solve error can be included by evaluating \eqref{eq:metricresidual} at the returned step.

The arm additionally projects onto its action box after radial shortening. The residual must then be evaluated using the displacement \emph{after} that projection. Euclidean projection does not in general preserve descent in an anisotropic metric. For example, take $\widehat g=(1,1)^\top$, $\mathsf M=\left(\begin{smallmatrix}1&2\\2&5\end{smallmatrix}\right)\succ0$, and $u^0=(-2,0)^\top$ in $[-2,2]^2$. Then $v=(-3,1)^\top$ and $\widehat g^\top v=-2$, but projecting $u^0+\gamma v$ onto the box for $0<\gamma\leq1$ returns displacement $(0,\gamma)^\top$ with positive gradient product. Equation~\eqref{eq:metricresidual} accounts for such a change; it does not assume it is small. Subsequent target regression and the deployed cap are covered separately by $e_t^{\rm int}$. The stored experimental diagnostics do not certify $\delta$, $L_\nabla$, or the final projected-step residual over the baseline state laws.

\paragraph{Acceptance and backward comparison.}
\begin{equation}
 d_t^{\rm back}\le I_t\min\{b_t^{\rm met},2\varepsilon_t-\eta_t\}
 +L_t^{Q}e_t^{\rm int}.
 \label{eq:metricbackwardbudget}
\end{equation}

For an accepted target, Proposition~\ref{prop:inexactmetric} bounds its true action-value change by $b_t^{\rm met}$, while the independent comparison error bound gives $2\varepsilon_t-\eta_t$. Both refer to the same actual corrected continuation and the same accepted action, so their minimum is also an upper bound. Rejection returns the baseline and contributes zero. Adding the action-value change between the target and the deployed regression proves \eqref{eq:metricbackwardbudget}; Theorem~\ref{thm:backward} then sums it under baseline visitation.
These statements condition on completed construction and require their bounds at the realized proposals, almost surely under the indicated state laws. Independent acceptance samples avoid reusing derivative noise, but do not provide these error bounds by themselves. A negative integrated bound certifies fitted-model improvement; a nonnegative bound only limits possible deterioration. Plant transfer still requires \eqref{eq:transfer}.

\subsection{Proof of Proposition~\ref{prop:lqgbackwardexact}}\label{proof:prop:lqgbackwardexact}
\begin{proof}
Under the fixed context schedule, $P_t$ is deterministic and independent of the action. The belief mean has the form
$m_{t+1}=Am_t+Bu_t+b_t+\xi_{t+1}$, with centered Gaussian $\xi_{t+1}$ whose covariance is action independent. Conditional stage cost is
$\tfrac12m^\top C_t^xm+\tfrac12u^\top C_t^uu+\tfrac12\tr(C_t^xP_t)$.
At the terminal date $S_T=C_T^x$, $q_T=0$, and $c_T=\tfrac12\tr(C_T^xP_T)$. Suppose inductively that the constructed continuation is optimal and has value
$\bar V_{t+1}(m)=\tfrac12m^\top S_{t+1}m+q_{t+1}^\top m+c_{t+1}$, with $S_{t+1}\succeq0$.
Gaussian expectation then gives
\begin{align}
 \bar\Qfun_t(m,u)
 &=\tfrac12u^\top G_tu+
 u^\top B^\top\{S_{t+1}(Am+b_t)+q_{t+1}\}+d_t(m),
 \nonumber\\
 G_t&=C_t^u+B^\top S_{t+1}B\succ0,
 \label{eq:lqgbackwardquadratic}
\end{align}
where $d_t$ is independent of $u$. Thus \eqref{eq:lqgbackwardnewton} cancels its starting action and equals the unique minimizer $-K_tm-k_t$. The target is affine. Its coefficients give zero least-squares residual, and full column rank makes them the unique fitted coefficients. Exact fitting therefore reproduces the minimizer for every $m$, not just at the nodes. Its cost is no greater than that of $u_t^0$, so the stated exact acceptance test preserves it. Substitution into the action value gives the Riccati recursion and a quadratic value at date $t$; its state Hessian $S_t$ is positive semidefinite by the Schur complement of the joint state--action quadratic. Induction from $T-1$ to zero completes the proof.
\end{proof}
Pooling context schedules requires a policy representation containing their optimal coefficients and a full-rank pooled design; a truncated feature basis need not suffice. Active sensing, action bounds and nonquadratic costs generally remove the closed-form LQG reference. Small action error on a finite evaluation panel does not establish the coefficient-error hypothesis of Theorem~\ref{thm:lqgrefinement}.

\subsection{Model-transfer and action-error accounting}\label{app:transferaccounting}
\paragraph{Performance transfer under plant--model mismatch.}
Let $J_\star(\pi)$ denote the cost of the same observation-adapted policy in the plant. If, for the two policies under comparison,
$|J_\star(\pi^j)-J_\theta(\pi^j)|\leq e_j$, then elementary addition and subtraction gives
\begin{equation}
J_\star(\pi^1)-J_\star(\pi^0)
\leq J_\theta(\pi^1)-J_\theta(\pi^0)+e_1+e_0.
\label{eq:transfer}
\end{equation}
The bound guarantees transfer when nominal improvement exceeds $e_0+e_1$. These require separate bounds; fitted posterior covariance does not measure parameter uncertainty.

For the backward controller, take $\pi^1=\bar\pi$ and apply Corollary~\ref{cor:backwarddefect}. Under its assumptions and full-policy transfer bounds $e_{\bar\pi},e_{\pi^0}$,
\begin{equation}
 J_\star(\bar\pi)-J_\star(\pi^0)
 \leq\E_\theta^{\pi^0}\sum_{t<T}
 \bigl[I_t(2\varepsilon_t-\eta_t)+L_t^{Q}e_t^{\mathrm{int}}\bigr]
 +e_{\bar\pi}+e_{\pi^0}.
 \label{eq:backwardtransfer}
\end{equation}
Transfer bounds must cover the complete corrected controller and its fitted filter in the plant. A one-action probe followed by DPO is insufficient. Derivatives, acceptance, interpolation, and comparison must use the same corrected tail.

\paragraph{Action errors on common histories.}
Let $h_t$ denote the observed history and let $s_t^\theta(h_t)$ be the controller's fitted-model information state. In the quadratic companion, define the action errors on a specified distribution $\rho$ of history--time pairs as
\begin{align}
\mathcal E_{\theta}(\pi;\rho)&=\E_{\rho}\norm{u_t^\pi(h_t)-u_t^{\theta,\star}(h_t)}^2,\\
\mathcal E_{\star}(\pi;\rho)&=\E_{\rho}\norm{u_t^\pi(h_t)-u_t^{\star}(h_t)}^2.
\label{eq:policyerrors}
\end{align}
The first measures distance from the fitted model's own optimal decision; the second includes disagreement between the fitted model and plant and their conditional beliefs. The plant oracle is evaluator-only. These are action-space errors, not Kalman-gain errors. The squared errors are not additive: with $a=u^\pi-u^{\theta,\star}$ and $b=u^{\theta,\star}-u^\star$,
\begin{equation}
\mathcal E_{\star}(\pi;\rho)=\mathcal E_{\theta}(\pi;\rho)+\E_\rho\norm{b}^2+2\E_\rho[a^\top b].
\label{eq:erroridentity}
\end{equation}
Hence fitted-model-to-plant discrepancy is not a lower bound for every policy. Greedification under a suboptimal frozen tail also differs from full-horizon optimal control.

\section{Finite-horizon error propagation and LQG structure}\label{app:theory}
Section~\ref{app:theoryproofs} proves the frozen-continuation $O(\epsilon^2)$ coefficient bound. Section~\ref{app:refreshedpropagation} tracks local solve and representation errors through the actual corrected affine tail. The Riccati reference in Section~\ref{app:riccatireference} fixes the optimal coefficients, and Section~\ref{app:affineextensions} accounts for numerical deployment errors. These results retain the fixed fitted model and finite horizon.

\subsection{Finite-horizon affine refinement}\label{app:theoryproofs}
In the fitted LQG model, fix the initial covariance and the deterministic context schedule. Gaussian filtering yields a deterministic covariance sequence and the innovation representation
\[
m_{t+1}=Am_t+Bu_t+b_t+\eta_{t+1},\qquad
\E[\eta_{t+1}\mid\F_t,u_t]=0,\quad
\operatorname{Cov}(\eta_{t+1}\mid\F_t,u_t)=\Xi_t,
\]
where the Gaussian innovation law is independent of state and action. The conditional stage cost is $\tfrac12m^\top C_t^xm+\tfrac12u^\top C_t^uu+q_t^{\rm cov}$, where $q_t^{\rm cov}=\tfrac12\tr(C_t^xP_t)$; the terminal conditional cost has the analogous variance term. These statements concern the fitted model's own conditional law.

Set
\begin{equation}
\bar A_t=\begin{pmatrix}A&b_t\\0&1\end{pmatrix},\quad
\bar B=\begin{pmatrix}B\\0\end{pmatrix},\quad
\bar C_t=\diag(C_t^x,0),\quad
\bar W_t=\diag(\Xi_t,0),\quad R_t^u=C_t^u.
\label{eq:lqgaugmentation}
\end{equation}
For $u=-L_tz$, put $A_t^L=\bar A_t-\bar B L_t$. A canonical augmented quadratic matrix $\mathsf S_t^L$ and separate scalar $c_t^L$ satisfy
\begin{align}
\mathsf S_T^L&=\bar C_T,\quad c_T^L=q_T^{\rm cov},\nonumber\\
\mathsf S_t^L&=\bar C_t+L_t^\top R_t^u L_t+(A_t^L)^\top\mathsf S_{t+1}^L A_t^L,\nonumber\\
c_t^L&=q_t^{\rm cov}+c_{t+1}^L+\tfrac12\tr(\mathsf S_{t+1}^L\bar W_t),\qquad
V_t^L(m)=\tfrac12z^\top\mathsf S_t^Lz+c_t^L.
\label{eq:affineevaluation}
\end{align}
Thus $\mathsf S_t^L\succeq0$. Its final row and column include deterministic affine costs; accumulated noise costs remain in $c_t^L$. This convention avoids folding the noise constants into the matrix recurrence. The optimal and baseline-continuation greedy coefficients are
\begin{equation}
G_t^i=R_t^u+\bar B^\top\mathsf S_{t+1}^i\bar B,\qquad
L_t^\star=(G_t^\star)^{-1}\bar B^\top\mathsf S_{t+1}^\star\bar A_t,\qquad
L_t^+=(G_t^0)^{-1}\bar B^\top\mathsf S_{t+1}^0\bar A_t.
\label{eq:affinegreedy}
\end{equation}
The optimal recursion agrees with \eqref{eq:lqgrecursion}; $\mathsf S$ is the augmented matrix, distinct from its physical-state block $S$ there.

\begin{theorem}[Frozen-continuation refinement of an affine LQG policy]\label{thm:lqgrefinement}
Fix a fitted model satisfying \eqref{eq:lqgmodel}--\eqref{eq:lqgcost}, a finite horizon, a known deterministic context schedule, and an initial belief covariance. Actions are unrestricted, all coefficients and noise covariances are finite, and $C_t^u\succeq rI$ for some $r>0$. Write $z=(m^\top,1)^\top$, $u_t^0(m)=-L_t^0z$, and $u_t^{\theta,\star}(m)=-L_t^\star z$. Evaluate the affine baseline exactly once and let $u_t^+(m)=\argmin_u\Qfun_t^0(m,u)=-L_t^+z$. For every fixed radius $\delta>0$, the finite constant $C_\delta$ in \eqref{eq:lqgrefinementconstant} satisfies
\begin{equation}
\epsilon=\max_{t<T}\norm{L_t^0-L_t^\star}_{\rm op}\leq\delta
\quad\Longrightarrow\quad
\max_{t<T}\norm{L_t^+-L_t^\star}_{\rm op}\leq C_\delta\epsilon^2.
\label{eq:lqgrefinement}
\end{equation}
For any fixed distribution $\rho$ of history--time pairs with
$M_\rho=\E_\rho\norm{(m_t^\theta(h)^\top,1)^\top}^2<\infty$, evaluating both actions at the same fitted belief gives
\begin{equation}
\left(\E_\rho\norm{u_t^+(h)-u_t^{\theta,\star}(h)}^2\right)^{1/2}
\leq\sqrt{M_\rho}\,C_\delta\epsilon^2.
\label{eq:lqgrefinementrho}
\end{equation}
\end{theorem}
\begin{proof}[Proof of Theorem~\ref{thm:lqgrefinement}]
Let $D_t=L_t^0-L_t^\star$, $\Delta\mathsf S_t=\mathsf S_t^0-\mathsf S_t^\star$, and $A_t^\star=\bar A_t-\bar B L_t^\star$. Completing the square around the optimal coefficient cancels the linear term in $D_t$ and gives
\begin{align}
\Delta\mathsf S_T&=0,\qquad
\Delta\mathsf S_t=(A_t^0)^\top\Delta\mathsf S_{t+1}A_t^0+D_t^\top G_t^\star D_t,\label{eq:affinequadraticerror}\\
\Delta c_T&=0,\qquad
\Delta c_t=\Delta c_{t+1}+\tfrac12\tr(\Delta\mathsf S_{t+1}\bar W_t).\nonumber
\end{align}
Both differences are nonnegative in their respective matrix/scalar orders. Define the finite quantities
\begin{align}
b&=\norm{\bar B}_{\rm op},\quad a_\star=\max_{t<T}\norm{A_t^\star}_{\rm op},\quad
g_\star=\max_{t<T}\norm{G_t^\star}_{\rm op},\quad a_\delta=a_\star+b\delta,\nonumber\\
H_T(a)&=\sum_{j=0}^{T-1}a^{2j},\qquad
C_\delta=\frac{b a_\star g_\star}{r}H_T(a_\delta).
\label{eq:lqgrefinementconstant}
\end{align}
Since $\norm{A_t^0}_{\rm op}\leq a_\delta$, backward induction in \eqref{eq:affinequadraticerror} gives
\[
\norm{\Delta\mathsf S_t}_{\rm op}
\leq g_\star\epsilon^2\sum_{j=0}^{T-t-1}a_\delta^{2j}.
\]
Subtracting the two stationarity systems in \eqref{eq:affinegreedy} yields the exact identity
\begin{equation}
G_t^0(L_t^+-L_t^\star)=\bar B^\top\Delta\mathsf S_{t+1}A_t^\star.
\label{eq:affinegainerror}
\end{equation}
Because $G_t^0\succeq R_t^u\succeq rI$, the preceding bound proves \eqref{eq:lqgrefinement}. At the final decision $\Delta\mathsf S_T=0$, so its exact correction is optimal even without a small initial error. At other dates a finite $C_\delta$ exists for every finite horizon, although it can grow rapidly with that horizon. Finally,
$\norm{u_t^+(h)-u_t^{\theta,\star}(h)}\leq C_\delta\epsilon^2\norm{z_t(h)}$.
Square, integrate under the same $\rho$, and take a square root to obtain \eqref{eq:lqgrefinementrho}.
\end{proof}

The premise controls all affine coefficients, intercepts, and remaining dates. Small finite-panel action MSE, including the five-date experimental panel, does not establish it. Neural continuations need additional approximation arguments.

The affine augmentation includes a constant coordinate, so its transition matrix has an eigenvalue at one. The bound above is finite-horizon; no contraction or horizon-uniform constant is asserted.

\subsection{Propagation through a refreshed affine tail}\label{app:refreshedpropagation}
The preceding cancellation also quantifies the effect of already-corrected future policies. Retain the fitted LQG assumptions and augmented notation, and condition on an arbitrary sequence of deployed affine maps $\bar\pi_j(m)=-\bar L_jz$. Let $\bar{\mathsf S}_j$ be their exact evaluation matrices, and define the exact greedy coefficient at date $t$ using this actual tail by
\[
 \widetilde L_t=(\bar G_t)^{-1}\bar B^\top\bar{\mathsf S}_{t+1}\bar A_t,
 \qquad \bar G_t=R_t^u+\bar B^\top\bar{\mathsf S}_{t+1}\bar B.
\]
Write $e_j=\norm{\bar L_j-L_j^\star}_{\rm op}$ and
$\eta_t=\norm{\bar L_t-\widetilde L_t}_{\rm op}$. The latter includes departure from the exact current-action solve and its affine representation. Put
$\bar A_j^{\rm cl}=\bar A_j-\bar B\bar L_j$,
$A_t^\star=\bar A_t-\bar B L_t^\star$, and
\[
 \Phi_{t+1,t+1}=I,\qquad
 \Phi_{t+1,j}=\bar A_{j-1}^{\rm cl}\cdots\bar A_{t+1}^{\rm cl}
 \quad(j>t+1).
\]
Then, for $R_t^u\succeq rI$,
\begin{equation}
 e_t\leq\eta_t+
 \frac{\norm{\bar B}_{\rm op}\norm{A_t^\star}_{\rm op}}{r}
 \sum_{j=t+1}^{T-1}
 \norm{\Phi_{t+1,j}}_{\rm op}^2
 \norm{G_j^\star}_{\rm op}e_j^2.
 \label{eq:refreshedaffinebudget}
\end{equation}
Indeed, completion of the square gives
\[
 \bar{\mathsf S}_{t+1}-\mathsf S_{t+1}^\star
 =\sum_{j=t+1}^{T-1}\Phi_{t+1,j}^\top
 (\bar L_j-L_j^\star)^\top G_j^\star
 (\bar L_j-L_j^\star)\Phi_{t+1,j},
\]
and subtraction of the greedy stationarity equations gives
\[
 \bar G_t(\widetilde L_t-L_t^\star)
 =\bar B^\top(\bar{\mathsf S}_{t+1}-\mathsf S_{t+1}^\star)A_t^\star.
\]
Apply the operator norm and $\norm{\bar G_t^{-1}}_{\rm op}\leq1/r$, then add $\eta_t$. The final-date sum is empty, so zero local errors yield exact backward recovery. Nonzero local errors and corrected-tail coefficients determine earlier errors. Node-fit MSE alone does not bound $\eta_t$, and nonlinear or capped maps require further approximation.

\subsection{Riccati reference coefficients}\label{app:riccatireference}
For the main paper's fitted affine LQG model and cost, the deterministic covariance separates from the action optimization. In the main paper's notation, set $S_T=C_T^x$ and $q_T=0$,
\begin{align}
G_t&=C_t^u+B^\top S_{t+1}B,
&K_t&=G_t^{-1}B^\top S_{t+1}A,\nonumber\\
k_t&=G_t^{-1}B^\top(S_{t+1}b_t+q_{t+1}),\nonumber\\
S_t&=C_t^x+A^\top S_{t+1}A
-A^\top S_{t+1}B G_t^{-1}B^\top S_{t+1}A,\nonumber\\
q_t&=A^\top(S_{t+1}b_t+q_{t+1})\nonumber\\
&\quad-A^\top S_{t+1}B G_t^{-1}B^\top(S_{t+1}b_t+q_{t+1}).
\label{eq:lqgrecursion}
\end{align}

The fitted feedback is $u_t^{\theta,\star}(m)=-K_tm-k_t$; plant coefficients instead define the evaluator-only oracle.

\subsection{Numerical deployment of the finite-horizon correction}\label{app:affineextensions}
\begin{corollary}[Numerical deployment of the affine correction]\label{cor:deployment}
Under Theorem~\ref{thm:lqgrefinement}, let a returned action rule $v$ satisfy
$\Delta_\rho=(\E_\rho\norm{v-u^+}^2)^{1/2}<\infty$, including any internal correction randomness in this expectation. Then
\begin{equation}
\left(\E_\rho\norm{v-u^{\theta,\star}}^2\right)^{1/2}
\leq\sqrt{M_\rho}\,C_\delta\epsilon^2+\Delta_\rho.
\label{eq:numericalrefinement}
\end{equation}
\end{corollary}
Minkowski's inequality proves the claim. The term $\Delta_\rho$ includes sampling, solving, regularization, clipping, and rejection errors and requires independent control. For other baseline classes, replace $\sqrt{M_\rho}C_\delta\epsilon^2$ by the exact correction's actual policy error. The same accounting gives the damping bound below.

\paragraph{Damping and the numerical margin.}
In one fixed norm $\norm{\cdot}_N$, for example $L^2(\rho)$ with internal randomness included, define $E_N(w)=\norm{w-u^{\theta,\star}}_N$. If $\Delta_N=\norm{v-u^+}_N$, then for a constant $0\leq\alpha\leq1$ and a feasible mixture,
\begin{equation}
E_N((1-\alpha)u^0+\alpha v)
\leq(1-\alpha)E_N(u^0)+\alpha\{E_N(u^+)+\Delta_N\}.
\label{eq:dampingerrorbudget}
\end{equation}
If the bracket is below $E_N(u^0)$, every constant $\alpha>0$ gives strict improvement. State-dependent acceptance belongs inside $v$ and $\Delta_N$; a varying gain cannot be pulled outside the norm.

\paragraph{Why action accuracy and percentage cost improvement differ.}
For any admissible policy with finite second moments in the fixed LQG model, completion of the optimal Bellman square and telescoping give
\begin{equation}
J_\theta(\pi)-J_\theta(\pi^{\theta,\star})
=\tfrac12\E_\theta^\pi\sum_{t=0}^{T-1}
e_t^\top G_t^\star e_t,
\qquad e_t=u_t-u_t^{\theta,\star}(m_t).
\label{eq:lqgweightedcostgap}
\end{equation}
The optimal advantage is $e^\top G_t^\star e/2$; summing under $\pi$ proves the identity, including covariance-dependent constants. Fitted excess cost is therefore curvature-weighted action error under the policy visitation law. Total cost, plant evaluation, and common-history action MSE use different baselines or measures.

\section{Nonlinear backward accuracy and storage}\label{app:nonlinear}
This appendix proves Theorem~\ref{thm:nonlinearbackward} and Corollary~\ref{cor:nonlinearquadratic}. Section~\ref{app:nlcontinuation} derives continuation errors from policy errors, Section~\ref{app:nltargetstorage} treats correction and the implemented thrust storage, Section~\ref{app:nlclosure} closes the recursion, and Section~\ref{app:nlconsistency} identifies the fixed-ridge obstruction. Throughout, $V^\star,Q^\star,\pi^\star$ abbreviate the fitted-model optimum, not a plant reference.

\subsection{Weighted continuation cancellation}\label{app:nlcontinuation}
Use the norm \eqref{eq:nlpolicynorm} with induced multilinear derivative norms. Its factorial weights imply
$\norm{fg}_{2,vw}\le\norm f_{2,v}\norm g_{2,w}$ for compatible products, by Leibniz' rule. Consider a prescribed class of $C^2$ feedback policies $\rho$ and weights $v_t,w_t,\chi_t\ge1$ with $w_t\ge \chi_tv_t^2$. Assume differentiation under expectations is justified by integrable envelopes, the optimum is interior, and
\begin{equation}
 Q_{t,u}^\star(s,\pi_t^\star(s))=0,\qquad
 \sup_{\rho,\alpha\in[0,1]}
 \norm{Q_{t,uu}^\star(\cdot,\pi_t^\star+\alpha(\rho_t-\pi_t^\star))}_{2,\chi_t}
 \le L_t^\star.
 \label{eq:nlmultiplier}
\end{equation}
The latter is a uniform composition bound, not a small-error assumption. Derivatives of $Q^\star$ through order four and bounded weighted policy derivatives suffice when the corresponding growth bounds hold.

Write $G_t^\rho(s,\omega)=\mathcal T_{\theta,t}(s,\rho_t(s),\omega)$ and $\mathsf P_t^\rho f=\E f(G_t^\rho)$. Assume uniform bounds $p_0,p_1,p_2,p_K$ on the respective ratios
\[
 \frac{\E w_{t+1}(G)}{w_t(s)},\quad
 \frac{\E[w_{t+1}(G)\norm{D_sG}]}{w_t(s)},\quad
 \frac{\E[w_{t+1}(G)\norm{D_sG}^2]}{w_t(s)},\quad
 \frac{\E[w_{t+1}(G)\norm{D_s^2G}]}{w_t(s)}.
\]
The chain rule then gives
\begin{equation}
 \norm{\mathsf P_t^\rho f}_{2,w_t}\le C_t^{\rm tr}\norm f_{2,w_{t+1}},
 \qquad C_t^{\rm tr}=\max\{p_0,p_1+p_K/2,p_2\}.
 \label{eq:nltransition}
\end{equation}
Indeed the zeroth-, first- and half-second-derivative bounds are respectively $p_0a_0$, $p_1a_1$, and $(p_2a_2+p_Ka_1)/2$, where $a_k=\sup\norm{D^kf}/w_{t+1}$.

\begin{lemma}[Quadratic continuation error]\label{lem:nlcontinuation}
Put $E_j^\rho=\norm{\rho_j-\pi_j^\star}_{2,v_j}$. Under \eqref{eq:nlmultiplier}--\eqref{eq:nltransition},
\begin{equation}
 \norm{V_{t+1}^\rho-V_{t+1}^\star}_{2,w_{t+1}}
 \le\mathcal R_t:=\sum_{j>t}\frac{L_j^\star}{2}
       \left(\prod_{k=t+1}^{j-1}C_k^{\rm tr}\right)(E_j^\rho)^2.
 \label{eq:nlcontinuationbudget}
\end{equation}
At construction nodes and their action balls, let $F(u,\omega)=\mathcal T_{\theta,t}(s_i,u,\omega)$ and assume finite uniform bounds
\[
 \Gamma_{g,t}=\sup_{i,u}\E[w_{t+1}(F)\norm{F_u}],\quad
 \Gamma_{H,t}=\sup_{i,u}\E[w_{t+1}(F)(2\norm{F_u}^2+\norm{F_{uu}})].
\]
Then the action-gradient and action-Hessian differences between $Q_t^\rho$ and $Q_t^\star$ are at most $\Gamma_{g,t}\mathcal R_t$ and $\Gamma_{H,t}\mathcal R_t$, respectively.
\end{lemma}
\begin{proof}
For $h=\rho_t-\pi_t^\star$, stationarity and Taylor's formula give
\[
 Q_t^\star(s,\rho_t(s))-V_t^\star(s)
 =\int_0^1(1-\alpha)h(s)^\top Q_{t,uu}^\star(s,\pi_t^\star(s)+\alpha h(s))h(s)\,d\alpha.
\]
The weighted product bound makes its $C^2$ norm at most $L_t^\star(E_t^\rho)^2/2$. Subtracting Bellman equations yields
$D_t=Q_t^\star(\cdot,\rho_t)-V_t^\star+\mathsf P_t^\rho D_{t+1}$, with $D_T=0$ and $D_t=V_t^\rho-V_t^\star$. Backward substitution proves \eqref{eq:nlcontinuationbudget}. Finally differentiate $\E D_{t+1}(F)$ once and twice in $u$. The terms are $\E[D D_{t+1}(F)F_u]$ and
$\E[D^2D_{t+1}(F)[F_u,F_u]+D D_{t+1}(F)F_{uu}]$. The norm bounds its first derivative by $w_{t+1}\mathcal R_t$ and its second by $2w_{t+1}\mathcal R_t$, giving the stated constants.
\end{proof}

These conditions allow unbounded Gaussian states when the policies, transitions and optimal solution have suitable polynomial-growth envelopes and the weights dominate the displayed moments. They do not follow from Gaussianity alone. For a uniform action-Hessian Lipschitz constant in the actual continuation, it suffices additionally to have third-order rollout derivatives with integrable envelopes; a finite smooth dictionary with bounded coefficients can supply policy derivative bounds. Covariance inverses require a positive lower eigenvalue bound. Hard-cap crossings require the separate treatment in Section~\ref{app:regularity}.

\subsection{Node correction, acceptance and policy storage}\label{app:nltargetstorage}
Fix a date and node, suppress $t$, and set $q=Q_t^\rho(s_i,\cdot)$, $u^*=\pi_t^\star(s_i)$, $u^0=\pi_t^0(s_i)$. On a ball of radius $r_B$ around $u^*$ suppose
$2\mu I\preceq Q_{t,uu}^\star\preceq L_*I$ and $\Gamma_H\mathcal R_t\le\mu$. Lemma~\ref{lem:nlcontinuation} gives $\mu I\preceq q''\preceq LI$, with $L=L_*+\Gamma_H\mathcal R_t$. Let $q''$ be $L_h$-Lipschitz there. For the used matrix and computed step assume
\[
 \widehat B\succeq\widehat mI,\quad
 \norm{\widehat B-q''(u^0)}\le\delta_H,\quad
 \norm{\widehat g-q'(u^0)}\le\delta_g,\quad
 \norm{\widehat B d+\widehat g}\le\delta_{\rm lin}.
\]
Thus $\delta_H$ includes curvature-floor bias. Taking uniform node bounds and $\epsilon^0=\max_i\norm{u^0-u^*}$, the raw Newton proposal $n=u^0+d$ and shaped proposal $p$ satisfy
\begin{equation}
 \norm{n-u^*}\le b_N:=\frac{L_h(\epsilon^0)^2/2+\delta_H\epsilon^0+
 \delta_g+\delta_{\rm lin}+\Gamma_g\mathcal R_t}{\widehat m},\qquad
 \norm{p-u^*}\le b_P:=b_N+\eta_{\rm node}.
 \label{eq:nlnewtontarget}
\end{equation}
To prove this, multiply $n-u^*$ by $\widehat B$, add and subtract $q'(u^0)$ and $q''(u^0)(u^0-u^*)$, and integrate $q''$ along $[u^*,u^0]$. The Taylor remainder is bounded by $L_h(\epsilon^0)^2/2$ and $\norm{q'(u^*)}\le\Gamma_g\mathcal R_t$. A radial step cap gives $\eta_{\rm node}=(\norm d-r)_+$; fixed damping gives $(1-\alpha)\norm d$, which need not be second order.

If $\Gamma_g\mathcal R_t<\mu r_B$, the minimizer $x_\rho$ on the closed ball is interior: on its boundary,
$(u-u^*)^\top q'(u)\ge\mu r_B^2-\Gamma_g\mathcal R_t r_B>0$.
Strong convexity gives $\norm{x_\rho-u^*}\le c:=\Gamma_g\mathcal R_t/\mu$.
Assume $u^0,p$ lie in this ball. A paired comparison $\widehat\Delta$ with
$|\widehat\Delta-(q(p)-q(u^0))|\le\tau$ accepts $p$ if $\widehat\Delta<0$ and otherwise returns $u^0$. In both cases the accepted label $y$ obeys $q(y)\le q(p)+\tau$. Hence
\begin{equation}
 \norm{y-u^*}\le b:=\kappa b_P+(\kappa+1)c+\sqrt{2\tau/\mu},
 \qquad\kappa=\sqrt{L/\mu}.
 \label{eq:nlaccepted}
\end{equation}
Indeed, $\mu\norm{y-x_\rho}^2/2\le L\norm{p-x_\rho}^2/2+\tau$.
Exact acceptance therefore preserves second-order target accuracy even upon rejection. A sharper alternative is
\begin{equation}
 \norm{p-x_\rho}\le\eta\norm{u^0-x_\rho},\qquad
 L\eta^2\le\mu/2,\qquad
 \tau<\frac{\mu}{4}\norm{u^0-x_\rho}^2.
 \label{eq:nlacceptmargin}
\end{equation}
Strong convexity and smoothness give $q(p)-q(u^0)\le-\mu\norm{u^0-x_\rho}^2/4$, so the comparison accepts $p$ and the sharper bound $b_P$ applies. Under the other scaling assumptions of Corollary~\ref{cor:nonlinearquadratic}, second-order accepted targets also follow if each node either already has $\norm{u^0-u^*}=O(\epsilon^2)$ or satisfies \eqref{eq:nlacceptmargin}: in the first case both possible labels are second-order accurate, and in the second the accurate proposal is accepted.

\paragraph{Sampling scope of the margin alternative.}\label{app:nlacceptsampling}
Condition on the constructed actions and stored tail, and use $M$ fresh independent paired rollout units, each with conditional mean $q(p)-q(u^0)$. If a pathwise action-gradient envelope on the segment has second moment at most $K^2$, the paired sample-mean error satisfies
\[
 \operatorname{Var}(\widehat\Delta)\le
 K^2\norm{p-u^0}^2/M.
\]
Thus, for a fixed failure probability, $\norm{p-u^0}=O(\epsilon)$ gives an $O(\epsilon/\sqrt M)$ comparison-error bound. Realizing $\tau=O(\epsilon^4)$ by this bound has sufficient sample order $\epsilon^{-6}$. Under \eqref{eq:nlacceptmargin}'s contraction and the additional nondegeneracy $\norm{u^0-x_\rho}=\Theta(\epsilon)$, order $\epsilon^{-2}$ suffices for the comparison margin instead. These are conditional sufficient allocations, not sample-complexity lower bounds; an $O(\epsilon)$ baseline upper bound alone does not imply nondegeneracy. With uniformly bounded gradient-estimator variance, achieving $\delta_g=O(\epsilon^2)$ by independent sample means still has sufficient sample order $\epsilon^{-4}$ at fixed confidence, before bias and the other errors are controlled.

\paragraph{The sequential polynomial--kernel map.}
The thrust storage in Online Supplement~\ref{sup:regressionrefit} first fits direct actions and then fits a kernel to the residual after an intermediate cap. For fixed normalized features and a selected candidate, let $Y$ contain accepted labels, $Z$ the polynomial design, $z(s)$ its feature column, and
\[
 M=Z^\top Z+\Lambda,\quad A=M^{-1}Z^\top,\quad H=ZA,\quad
 R=(K+\lambda_K I)^{-1},\quad K_{ij}=k_j(s_i).
\]
Here $\Lambda\succeq0$, $M$ is invertible, $K\succeq0$, and $\lambda_K>0$; the code's intercept penalty is also positive. Let $b_i^0=\pi_t^0(s_i)$ and apply the hard cap rowwise. Before the final deployment cap the actual real-arithmetic map is
\begin{align}
 f_Y(s)&=z(s)^\top AY+k(s)^\top R\{Y-b^0-\mathcal C_r(HY-b^0)\}\nonumber\\
 &=\mathsf S Y(s)+k(s)^\top R D_{\rm cap}(Y),\label{eq:nlactualstorage}\\
 \mathsf S Y(s)&=[z(s)^\top A+k(s)^\top R(I-H)]Y,\quad
 D_{\rm cap}(Y)=HY-b^0-\mathcal C_r(HY-b^0).\nonumber
\end{align}
Thus replacing the intermediate prediction by $HY$ would analyze a different operator. Set $r(s)^\top=z(s)^\top A+k(s)^\top R(I-H)$ and
$s^{\rm reg}=\sqrt N\norm r_{2,v_t}$. Since $0\preceq H\preceq I$, a bound is
$s^{\rm reg}\le\sqrt N(\norm z_{2,v_t}\norm A+\norm k_{2,v_t}\norm R)$; no full column rank of $Z$ is needed. Polynomial weights and the feature-scale/kernel-width floors give finite derivative envelopes, although narrow kernels can make them large.

Choose any comparator $g^\dagger=z^\top W+k^\top C$ with
$\norm{g^\dagger-\pi_t^\star}_{2,v_t}\le\alpha$. Put
\begin{align}
 \nu^2&=N^{-1}\sum_i v_t(s_i)^2,\quad D_W=AKC-M^{-1}\Lambda W,\nonumber\\
 \beta&=\norm{z^\top-k^\top RZ}_{2,v_t}\norm{D_W}_F
      +\lambda_K\norm{k^\top R}_{2,v_t}\norm C_F,\nonumber\\
 \chi_Y&=\norm{k^\top R D_{\rm cap}(Y)}_{2,v_t}
 \le\norm{k^\top R}_{2,v_t}\norm{D_{\rm cap}(Y)}_F.\nonumber
\end{align}
For maximum node-label error $b$,
\begin{equation}
 \norm{f_Y-\pi_t^\star}_{2,v_t}\le s^{\rm reg}b+F_t^{\rm reg},\qquad
 F_t^{\rm reg}=(1+s^{\rm reg}\nu)\alpha+\beta+\chi_Y+\eta_{\rm arith}.
 \label{eq:nlstoragebudget}
\end{equation}
To see this, $\norm{\mathsf S\Delta Y}_{2,v_t}\le\norm r_{2,v_t}\norm{\Delta Y}_F$, and exact comparator labels satisfy
\[
 \mathsf S(ZW+KC)-g^\dagger=(z^\top-k^\top RZ)D_W-\lambda_K k^\top RC.
\]
The comparator's node residual has Frobenius norm at most $\sqrt N\nu\alpha$. Decompose around those labels and add the intermediate-cap term. Here $\eta_{\rm arith}$ bounds residual rounding and stored-coefficient errors in the smooth map; the float32 intermediate prediction contributes at most $\norm{k^\top R}_{2,v_t}$ times its rounding-error Frobenius norm. Polynomial-only storage omits all kernel terms. Finite candidate selection may depend on targets: the bound applies to the realized frozen candidate, or uniformly over the candidate list, without differentiating selection.

This is a derivative-norm bound obtained from action-value labels, but its approximation term is essential. For finitely many nodes, a smooth bump $\psi$ supported away from them gives the strongly convex objective $q_\epsilon(s,u)=\norm{u-\epsilon\psi(s)}^2/2$. Exact optimal labels at every node are zero, while the optimal policy differs from zero by $O(\epsilon)$ away from them. Smoothness and zero node MSE therefore do not imply second-order policy accuracy. Likewise the coupling term $AKC$ in $\beta$ need not disappear when only the kernel penalty is decreased.

\subsection{Closing the recursion}\label{app:nlclosure}
\begin{proof}[Proof of Theorem~\ref{thm:nonlinearbackward}]
At date $t$, use the actual stored tail $\rho_{>t}=\bar\pi_{>t}$ in Lemma~\ref{lem:nlcontinuation}, so $E_j^\rho=E_j$. Combine \eqref{eq:nlnewtontarget}--\eqref{eq:nlstoragebudget}, and add the deployment discrepancy $\eta_{\rm dep,t}=\norm{\bar\pi_t-f_{Y_t}}_{2,v_t}$. This gives
\begin{align}
 E_t\le{}&s_t^{\rm reg}\left[\kappa_t\left(
 \frac{L_{h,t}(\epsilon_t^0)^2/2+\delta_{H,t}\epsilon_t^0+\delta_{g,t}
 +\delta_{\rm lin,t}+\Gamma_{g,t}\mathcal R_t}{\widehat m_t}
 +\eta_{\rm node,t}\right)\right.\nonumber\\
 &\left.\hspace{13mm}+(\kappa_t+1)\frac{\Gamma_{g,t}\mathcal R_t}{\mu_t}
 +\sqrt{2\tau_t/\mu_t}\right]+F_t^{\rm reg}+\eta_{\rm dep,t}.
 \label{eq:nlexplicitrecursion}
\end{align}
Equation~\eqref{eq:nonlinearbackward} follows with
\begin{align*}
 A_t^{\rm N}&=s_t^{\rm reg}\kappa_t L_{h,t}/(2\widehat m_t),\\
 B_{tj}^{\rm tail}&=s_t^{\rm reg}\Gamma_{g,t}
 \left(\kappa_t/\widehat m_t+(\kappa_t+1)/\mu_t\right)
 (L_j^\star/2)\prod_{k=t+1}^{j-1}C_k^{\rm tr},\\
 C_t^{\rm err}&=s_t^{\rm reg}\max\{\kappa_t/\widehat m_t,\kappa_t,\sqrt{2/\mu_t}\}.
\end{align*}
The same policy norm is produced by storage and required by the earlier-date continuation calculation, closing the backward recursion.
\end{proof}

\paragraph{Local radius conditions.}
Prescribe policy derivative/coefficient envelopes and error radii $\overline E_t$ before the construction. Descending in $t$, insert later radii in \eqref{eq:nlcontinuationbudget}. Check $\Gamma_{H,t}\mathcal R_t\le\mu_t$, $\Gamma_{g,t}\mathcal R_t<\mu_t r_{B,t}$, and $\epsilon_t^0,b_{P,t}\le r_{B,t}$. Require \eqref{eq:nlexplicitrecursion} to be at most $\overline E_t$, with the stored map satisfying the prescribed regularity envelopes. Induction verifies the local regime without assuming the desired policy-error conclusion. These are mathematical sufficient conditions, not certificates supplied by the recorded finite-node diagnostics.

\begin{proof}[Proof of Corollary~\ref{cor:nonlinearquadratic}]
At the terminal decision $\mathcal R_{T-1}=0$, and \eqref{eq:nlexplicitrecursion} is $O(\epsilon^2)$. If all later errors are $O(\epsilon^2)$, the fixed finite sum $\mathcal R_t$ is $O(\epsilon^4)$, so the same bound holds at date $t$. Backward induction proves the result. Under the same weighted advantage and transition bounds, the nominal value error is then $O(\epsilon^4)$. The constants need not be uniform in horizon.
\end{proof}

\subsection{Storage consistency and fixed ridge}\label{app:nlconsistency}
\begin{proposition}[Nonzero consistency defect of the direct fit]\label{prop:nlridgefloor}
For \eqref{eq:nlactualstorage}, let both node baseline and exact optimal labels be $Y^*\ne0$, with $\Lambda\succ0$ and $\lambda_K>0$. Let $\overline Y$ be the final capped node actions. Then
\begin{equation}
 \norm{\overline Y-Y^*}_F\ge\min\{c\norm{Y^*}_F,r\}>0,\qquad
 c=\frac{\lambda_K}{\lambda_{\max}(K)+\lambda_K}
       \frac1{1+\norm{Z\Lambda^{-1/2}}^2}.
 \label{eq:nlridgefloor}
\end{equation}
For polynomial-only storage omit the first factor in $c$.
\end{proposition}
\begin{proof}
Write $E_N=I-H=(I+Z\Lambda^{-1}Z^\top)^{-1}\succ0$ and $d=-E_NY^*$. The raw node error, including the intermediate cap, is $e=d-KR\mathcal C_r(d)$. Since $\norm{KR}=\lambda_{\max}(K)/(\lambda_{\max}(K)+\lambda_K)<1$ and the cap does not increase the Frobenius norm,
$\norm e_F\ge(1-\norm{KR})\norm d_F\ge c\norm{Y^*}_F$.
Finally $\overline Y-Y^*=\mathcal C_r(e)$ rowwise, and
$\norm{\mathcal C_r(e)}_F\ge\min\{1,r/\norm e_F\}\norm e_F$. This proves \eqref{eq:nlridgefloor}, including both caps.
\end{proof}

The thrust polynomial penalty is $0.01$ with intercept penalty $10^{-12}$; each selected kernel penalty is positive. Thus an unchanged direct fit need not preserve even an optimal baseline after storage. The lower bound may be very small and is not an estimate of performance loss in the experiments. It establishes a consistency obstruction to an unconditional asymptotic rate claim. Finite target-based selection cannot remove it automatically: every branch is continuous in its node data, and the minimum of finitely many positive limiting bounds is positive.

A concrete sufficient regime is available when $Z^\top Z/N\succeq\kappa I$, dictionary derivative envelopes are uniform, and the kernel penalty is bounded below. With inactive caps and a polynomial comparator $z^\top W$ satisfying $\norm{z^\top W-\pi_t^\star}_{2,v_t}\le\alpha$, the preceding storage proof gives
\begin{equation}
 \norm{\bar\pi_t-\pi_t^\star}_{2,v_t}
 \le C\bigl(b+\alpha+\lambda_P\norm W_F\bigr),\qquad
 \lambda_P=\norm\Lambda/N.
 \label{eq:nlridgeschedule}
\end{equation}
Indeed its reproduction defect is $-(z^\top-k^\top RZ)M^{-1}\Lambda W$, and $\norm{M^{-1}\Lambda}\le\lambda_P/\kappa$. Direct fitting with bounded $W$, $b,\alpha=O(\epsilon^2)$ and $\lambda_P=O(\epsilon^2)$ preserves second-order accuracy. A baseline-residual fit instead permits $\lambda_P=O(\epsilon)$ if the first-order correction has coefficients $W=O(\epsilon)$ and approximation error $O(\epsilon^2)$. These representation and conditioning assumptions are essential; this latter construction is a sufficient alternative, not a change to the reported thrust algorithm. On unbounded Gaussian state domains, weighted closeness alone does not establish cap inactivity.

\section{Nonsmooth safeguards and belief regularity}\label{app:gaussian}
Section~\ref{app:regularity} distinguishes branchwise curvature under hard caps from valid population derivative identities and gives a smooth alternative. Section~\ref{app:beliefconditions} records the information and state conditions required for an exact Gaussian belief. The metric bound in Section~\ref{app:metricimplementation} remains available without Hessian consistency, subject to its gradient-error and smoothness assumptions.

\subsection{Derivative chain rule and nonsmooth safeguards}\label{app:regularity}
For $F(u,\omega)=\mathcal T_{\theta,t}(s,u,\omega)$, when the required value derivatives and differentiation under expectation are justified, the corresponding chain-rule expressions are
\begin{align}
g_t&=\E\left[\ell_u+F_u^\top\nabla V_{t+1}^\rho(F)\right], \label{eq:gradchain}\\
B_t&=\E\left[\ell_{uu}+F_u^\top\nabla^2V_{t+1}^\rho(F)F_u
+\sum_j\partial_jV_{t+1}^\rho(F)\,\nabla_{uu}^2 F_j\right].
\label{eq:hesschain}
\end{align}
Stage-cost derivatives include dependencies through simulated arguments. The last Hessian term is generally nonzero; dropping it approximates the pathwise derivative \citep{heess2015}.

\paragraph{Differentiability of the implemented continuation.}
The radial cap is nonsmooth at radius $r$. If positive-probability noise paths cross that surface as $u$ varies over $U$, the almost-sure $C^2$-on-$U$ hypothesis of Proposition~\ref{prop:clderivatives} fails, even if the surface has zero probability at one fixed action. Automatic differentiation then computes branchwise curvature without an established population-Hessian unbiasedness guarantee.

First derivatives have a weaker sufficient condition: almost-sure differentiability at the fixed action, together with a local random Lipschitz bound having finite conditional expectation, permits dominated convergence of difference quotients. Lipschitz continuity alone does not ensure differentiability at every fixed action. Second derivatives require additional control. For example, let $Z\sim\N(0,1)$ and $f(u,Z)=\max\{-1,\min\{u+Z,1\}\}$. The pathwise second derivative is zero almost surely at fixed $u$, whereas
\[
 \frac{d^2}{du^2}\E f(u,Z)=\phi(1+u)-\phi(1-u),
\]
which equals $-0.222548$ at $u=1/2$; $\phi$ is the standard-normal density. Differentiating the selected branches misses the moving-boundary contribution. Finite differences of a fixed finite sample away from its kinks need not detect this population discrepancy. The reported derivative checks therefore retain their tested sample and continuation scope.

\paragraph{Smooth safeguards and exact pathwise derivatives.}
The following construction makes the regularity route explicit without changing the interpretation of the reported hard-cap experiments.

\begin{lemma}[Smooth radial caps and continuation regularity]\label{lem:smoothcap}
For $r>0$ and an integer $p\geq1$, define
\begin{equation}
 \mathcal C_{r,p}(d)=
 \frac{d}{[1+(\|d\|/r)^{2p}]^{1/(2p)}}.
 \label{eq:smoothcapfamily}
\end{equation}
Then $\mathcal C_{r,p}$ is $C^\infty$, its norm is strictly below $r$, and, for the hard cap $\mathcal C_r$ in \eqref{eq:meshdeploy},
\begin{equation}
 \sup_d\|\mathcal C_{r,p}(d)-\mathcal C_r(d)\|
 =r(1-2^{-1/(2p)})\longrightarrow0.
 \label{eq:smoothcaperror}
\end{equation}
Fix a finite horizon and an open neighborhood of the current action. Suppose every map evaluated in the actual continuation, including its feasible action parameterization, dynamics, costs and filter, is $C^2$ on the reachable domain. If the complete sampled cost's first and second action derivatives have integrable local envelopes, Proposition~\ref{prop:clderivatives} applies to this smooth continuation. A sufficient way to obtain these envelopes is polynomial growth of the maps and their derivatives through order two, uniformly on that domain, together with base noise having moments of every order and a fixed initial information state. Any inverses or Cholesky factors used in this condition require covariance matrices bounded away from singularity.
\end{lemma}
\begin{proof}
Since $(\|d\|/r)^{2p}=(d^\top d)^p/r^{2p}$ is a polynomial and the denominator is positive, the map is smooth even at $d=0$. Put $x=\|d\|/r$ and $h_p(x)=x(1+x^{2p})^{-1/(2p)}$. Then $h_p(x)<1$ for finite $x$ and
$h_p'(x)=(1+x^{2p})^{-1-1/(2p)}$.
The normalized approximation error is $x-h_p(x)$ for $x\leq1$ and $1-h_p(x)$ for $x\geq1$. It increases on the first interval and decreases on the second, proving \eqref{eq:smoothcaperror} at $x=1$.
Finite composition preserves $C^2$ regularity. Under the stated polynomial conditions, induction over the finite rollout bounds the states and their first two action sensitivities by a polynomial in the finitely many base-noise magnitudes, uniformly on a smaller bounded action neighborhood. The same holds for the cost derivatives. Finite noise moments make these envelopes integrable, and conditional dominated convergence gives the two pathwise identities.
\end{proof}

At $p=1$, \eqref{eq:smoothcapfamily} is the alternative $d/\sqrt{1+\|d\|^2/r^2}$; at $\|d\|=r$ it returns norm $r/\sqrt2$. Larger fixed $p$ approximates the hard cap more closely, but the lemma gives no uniform second-derivative bound as $p\to\infty$. Uniform action-map convergence therefore does not justify taking a hard-cap limit inside the Hessian identity. Nor does it by itself bound full-policy cost differences.

All safeguards in the evaluated future state-to-action map must satisfy the regularity conditions. For a box $[\ell_j,h_j]$, one smooth feasible coordinate map is
$S_j(v)=c_j+a_j\tanh((v-c_j)/a_j)$, where $c_j=(h_j+\ell_j)/2$ and $a_j=(h_j-\ell_j)/2>0$. It takes values strictly inside the box and changes actions relative to hard clipping. A $C^2$ saturation with constant exterior pieces is another possible design if exact endpoint values are needed. Merely smoothing the radial cap while retaining a hard final action projection does not satisfy the lemma.

The node QP, acceptance decision and regression procedure are not differentiated through: their completed coefficients are fixed when an earlier action is varied. Their own smoothness is therefore unnecessary for this lemma. In contrast, the smooth map must be the one used both in future rollouts and in deployment. Changing it requires rebuilding the backward policy; differentiating a smooth surrogate while deploying a hard map introduces a different-continuation error. Section~\ref{app:nlregime} evaluates the $p=1$ smooth cap in a controlled non-LQG example. The main learned-controller experiments retain hard safeguards, and the sufficient growth and covariance conditions have not been certified for those tasks. The policy-comparison identities remain valid for the existing nonsmooth policies under their own integrability assumptions.

\paragraph{Accuracy and regularity of the deployment cap.}
If $\norm{\pi_t^{\theta,\star}(s)-\pi_t^0(s)}\le r$, nonexpansiveness of projection implies
$\norm{\pi_t^0(s)+\mathcal C_r(f(s)-\pi_t^0(s))-\pi_t^{\theta,\star}(s)}\le\norm{f(s)-\pi_t^{\theta,\star}(s)}$.
This controls action values, not policy derivatives. For a smooth residual $d(s)=f(s)-\pi_t^0(s)$, a transverse crossing of $\norm d=r$ has a derivative jump $-nn^\top Dd$, with $n=d/r$. Moreover, a bounded baseline and finite Gaussian-kernel sum cannot make a nonconstant polynomial residual uniformly bounded along an unbounded state ray. A global inactivity claim therefore needs an explicit domain and coefficient argument.

For the smooth family \eqref{eq:smoothcapfamily}, let $d=f-\pi_t^0$ and suppose $\norm d_{2,1}\le E\le r$. Then the deployment term in Theorem~\ref{thm:nonlinearbackward} satisfies
\begin{equation}
 \eta_{\rm dep,t}=\norm{\mathcal C_{r,p}(d)-d}_{2,v_t}
 \le(4p+6)r^{-2p}E^{2p}\norm d_{2,v_t}.
 \label{eq:nlsmoothcapaccuracy}
\end{equation}
For completeness, $R(x)=\mathcal C_{r,p}(x)-x$ on $\norm x\le r$ obeys
$\norm R\le r^{-2p}\norm x^{2p+1}/(2p)$,
$\norm{DR}\le(1+1/(2p))r^{-2p}\norm x^{2p}$, and
$\norm{D^2R}\le(4p+2)r^{-2p}\norm x^{2p-1}$.
The chain rule for $R\circ d$, bounding one factor $Dd$ by $E$, gives \eqref{eq:nlsmoothcapaccuracy}. Thus a fixed smooth cap has higher-order distortion for uniformly small residuals. Weighted closeness on an unbounded domain alone does not imply the unweighted condition $\norm d_{2,1}\le E$, and a hard-cap crossing does not satisfy this smooth argument.

\subsection{Exact-belief conditions}\label{app:beliefconditions}
\paragraph{Information.}
New predictor information about the latent state must enter the likelihood or satisfy Assumption~\ref{ass:context}; a current outcome cannot predict itself before observation.
\paragraph{State.}
Recurrent predictor states and random future covariates require their own retained states and transition laws. Filtering $X_t$ alone does not supply these omitted dynamics.
\paragraph{Improvement.}
Exact filtering, derivatives, local optimization, and action-value decrease are separate requirements. Theorems~\ref{thm:improve} and~\ref{thm:backward} also require their acceptance/model conditions; regressed policies retain the interpolation term in Corollary~\ref{cor:backwarddefect}.

\clearpage
\section{A controlled non-LQG quadratic-accuracy regime}\label{app:nlregime}
Figure~\ref{fig:nlregime} displays the error curves underlying Table~\ref{tab:nlregime}.

This experiment connects the correction, storage, and smooth deployment operations in Corollary~\ref{cor:nonlinearquadratic}. The reference policy is perturbed by a fixed smooth direction, and the deployed future map is used in the preceding-date calculation.

\paragraph{Problem and reference.}
There are two scalar decisions, with unrestricted actions,
\begin{align}
 s_1&=\tanh(s_0)+u_0+0.4\xi,\qquad \xi\sim\N(0,1),\nonumber\\
 \ell_0(s_0,u_0)&=0.35u_0^2,\qquad
 \ell_1(s_1,u_1)=\tfrac12(s_1+u_1)^2+\tfrac14u_1^4 .
 \label{eq:nlregimemodel}
\end{align}
The last optimum is the unique root of $u^3+u+s=0$:
\[
 \pi_1^\star(s)=\frac{2}{\sqrt3}
 \sinh\!\left[\frac13\operatorname{arsinh}(-3\sqrt3s/2)\right].
\]
Writing $p=\pi_1^\star(s)$ gives $(V_1^\star)'=s+p$ and
$(V_1^\star)''=3p^2/(1+3p^2)\in[0,1)$, so
$0.7\le Q_{0,uu}^\star\le1.7$. We solve the first-date stationarity equation using 256-point Gauss--Hermite quadrature; reference state derivatives follow by implicit differentiation.

\paragraph{Backward construction and storage.}
Set $\pi_t^{0,\epsilon}=\pi_t^\star+\epsilon h_t$, where
$h_0(s)=0.5+\sin s$, $h_1(s)=\sin s$, and
$\epsilon\in\{0.2,0.1,0.05,0.025,0.0125,0.00625\}$.
At each date, one Newton step uses the gradient and Hessian of the action value under the actual stored future policy. The proposal displacement is transformed by
$\mathcal C_{1,1}(d)=d/\sqrt{1+d^2}$.
All construction-node targets pass the conditional-cost comparison within $10^{-12}$ numerical tolerance.
Storage uses 101 equally spaced nodes on $[-1.5,1.5]$ at date zero and 161 on $[-5,5]$ at date one. We apply the same smooth cap to the fitted displacement from the baseline before deployment and before evaluating the earlier continuation.

The principal comparison fits the residual target with the single feature $h_t(s)$, minimizing
$n^{-1}\sum_i(c h_t(s_i)-[y_i-\pi_t^{0,\epsilon}(s_i)])^2+\lambda c^2$.
The first-order correction direction is therefore representable. We compare $\lambda=\epsilon$ with fixed $\lambda=0.05$.
A separate direct-action diagnostic uses features $[\pi_t^\star,h_t]$, with ridge $\lambda=\epsilon^2$ or fixed $\lambda=0.01$. Including the reference-policy feature removes its dictionary approximation error and isolates direct-ridge bias; this diagnostic is not the thrust storage dictionary. Both fits use the normalized penalty convention above.

\paragraph{Measurements.}
For $v(s)=1+s^2$, we measure
\begin{equation}
 E_t^{\rm grid}=\sum_{k=0}^2\frac1{k!}
 \max_{s\in G_t}\frac{|D^k(\pi_t-\pi_t^\star)(s)|}{1+s^2},
 \label{eq:nlregimegrid}
\end{equation}
separately for the baseline, correction target, raw fit, and deployed map. The evaluation grids have 181 points on $[-1.5,1.5]$ and 241 on $[-5,5]$. These sampled compact-domain errors do not bound the global weighted supremum in \eqref{eq:nlpolicynorm}. The continuation diagnostic is
$\max_{s\in G_0}|Q_{0,u}^{\bar\pi_1}(s,\pi_0^\star(s))|/(1+s^2)$; the optimal-tail gradient vanishes.

\begin{figure}[htbp]
\centering
\includegraphics[width=\textwidth]{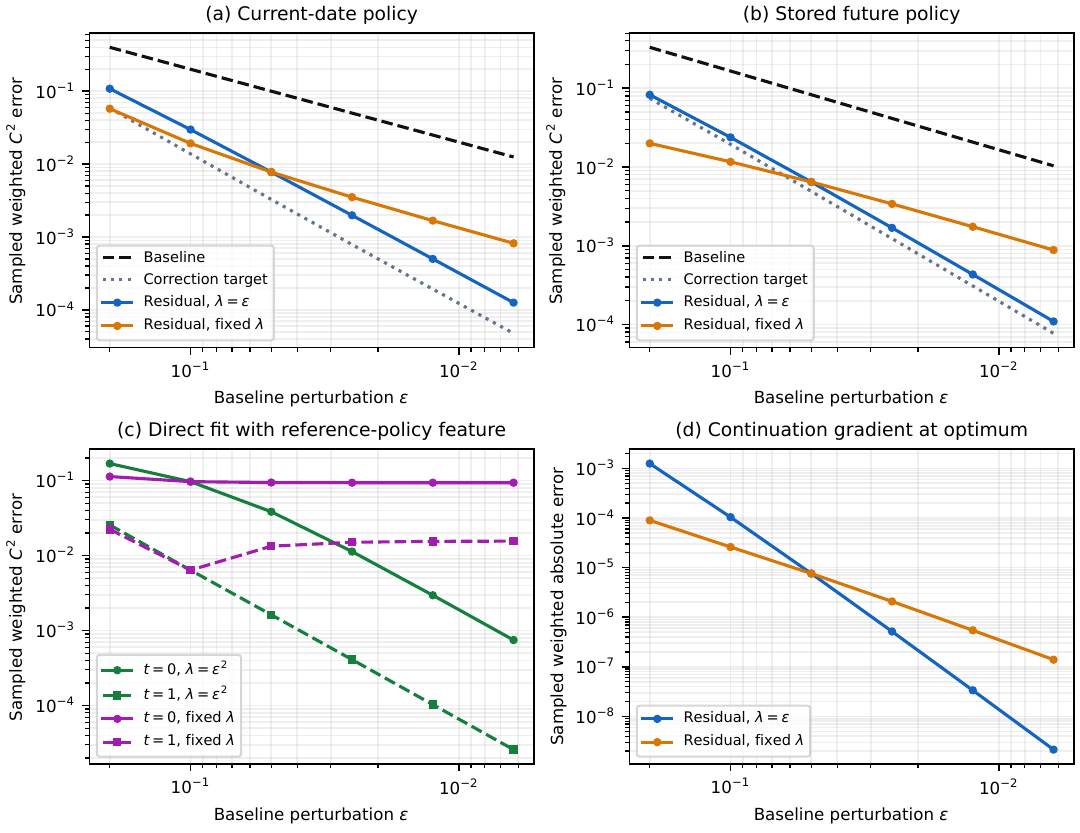}
\caption{Controlled non-LQG order measurements. Panels (a)--(b) separate baseline, target, and deployed residual-policy errors. Panel (c) isolates direct-ridge bias using a dictionary containing the reference policy. Panel (d) uses the actual stored, smooth deployed tail. Errors are sampled quantities defined in the text; smaller perturbations are to the right.}
\label{fig:nlregime}
\end{figure}

\paragraph{Observed regime.}
Table~\ref{tab:nlregime} gives orders over the final halving
$0.0125\to0.00625$; all halving intervals appear in the numerical archive.
Decreasing residual ridge yields approximately second-order deployed policies at both dates and fourth-order continuation-gradient error.
The individual action, first-derivative, and second-derivative errors also have orders 1.98--1.99. The deployment cap's difference from the raw residual fit has order 2.95--2.96, consistent with the higher-order distortion in \eqref{eq:nlsmoothcapaccuracy}.
At $\epsilon=0.00625$, date-zero target and deployed errors are
$4.76\times10^{-5}$ and $1.26\times10^{-4}$; the corresponding date-one errors are $7.72\times10^{-5}$ and $1.09\times10^{-4}$.

\begin{table}[htbp]
\centering\small
\caption{Final-halving orders of sampled errors. A value near two denotes quadratic accuracy; near zero denotes a plateau. Direct fits include the reference policy as a feature.}
\label{tab:nlregime}
\begin{tabular}{lrrrr}
\toprule
Storage & Target $t=0$ & Deployed $t=0$ & Deployed $t=1$ & Tail gradient\\
\midrule
Residual, $\lambda=\epsilon$ & 2.017 & 1.993 & 1.983 & 3.970\\
Residual, fixed $\lambda$ & 2.018 & 1.033 & 0.982 & 1.966\\
Direct, $\lambda=\epsilon^2$ & 2.018 & 1.984 & 1.995 & 4.022\\
Direct, fixed $\lambda$ & 0.450 & $-0.0001$ & $-0.010$ & $-0.022$\\
\bottomrule
\end{tabular}
\end{table}

Fixed residual ridge leaves an approximately first-order bias, rather than a nonzero absolute floor. Fixed direct ridge has a floor: even at $\epsilon=0$, deployed errors are 0.09374 and 0.01562 at dates zero and one. Its nonvanishing future-policy error also limits the preceding target.
The current Newton target remains second order even with a first-order residual tail; backward refresh reduces the continuation contribution, without generically increasing the full current-step order beyond two.

\paragraph{Numerical checks and scope.}
Changing from 192 to 256 quadrature points changes reference actions by at most $2.13\times10^{-14}$ and their first two derivatives by $5.75\times10^{-11}$.
At $\epsilon=0.0125$ with decreasing residual ridge, the date-zero target's sampled $C^2$ error changes by $9.29\times10^{-13}$; doubling evaluation-grid density changes the principal measurements by at most $6.0\times10^{-5}$ relatively.
Analytic action derivatives and implicit reference-policy derivatives agree with independent automatic differentiation within $2.8\times10^{-17}$ and $1.9\times10^{-16}$, respectively.
The minimum date-zero construction curvature is above 0.8968.

This is a local perturbation experiment with a baseline constructed from the reference policy, a representable residual direction, and deterministic quadrature. It illustrates the consistent regime of Corollary~\ref{cor:nonlinearquadratic}, including actual smooth deployment, while leaving global derivative envelopes and the learned-controller conditions unverified. The four storage settings are diagnostics of order and bias, not performance comparisons of learned actors. The numerical archive
\texttt{nonlinear\_quadratic\_regime\_20260928}
contains the implementation, per-component errors, and quadrature/grid checks.

\clearpage
\sloppy
\noindent{\Large\bfseries Online Supplement}\par\medskip
\counterwithout{table}{section}
\counterwithout{figure}{section}
\setcounter{section}{0}
\setcounter{table}{0}
\setcounter{figure}{0}
\renewcommand{\theHsection}{supp.\arabic{section}}
\renewcommand{\theHtable}{supp.\arabic{section}.\arabic{table}}
\renewcommand{\theHfigure}{supp.\arabic{section}.\arabic{figure}}
\renewcommand{\theHequation}{supp.\arabic{section}.\arabic{equation}}

\renewcommand{\thesection}{S\arabic{section}}
\renewcommand{\thetable}{S\arabic{table}}
\renewcommand{\thefigure}{S\arabic{figure}}
\section*{Reading guide}
This supplement provides the experimental details needed to interpret the main paper. Section~\ref{sup:completeexperiments} specifies the task models, fitting data, controller budgets, and paired evaluation. Section~\ref{sup:regressionrefit} isolates policy storage through target validation and fresh episode panels. Section~\ref{sup:constructionprotocol} documents independent OLF solver checks, restricted feedback-MPC student construction, and the matched policy-storage comparison in Section~\ref{sup:matchedstorage}. Section~\ref{sup:currentnative} gives current native deployment measurements and local numerical checks. Section~\ref{sup:scalardp} reports the separate scalar docking DP reference and regression diagnostic. Section~\ref{sup:lqgpanel} preserves the full quadratic-thrust reference panel. The Technical Appendix contains the proofs, safeguard analysis, and controlled non-LQG accuracy experiment (Section~\ref{app:nlregime}).

\section{Core experimental protocols}\label{sup:completeexperiments}
All tasks use stochastic partially observed models fitted from observations. Filtering is exact within the fitted Gaussian families except for the arm's approximate EKF. The docking protocol in Section~\ref{sup:dockingprotocol} tests information acquisition, thrust in Sections~\ref{sec:rocketprotocol}--\ref{sup:thrustresults} supplies exact LQG and nonquadratic comparisons, and the arm in Section~\ref{sup:robotprotocol} tests nonlinear dynamics. Independent OLF MPC optimizes future action sequences, whereas restricted feedback MPC optimizes belief-dependent future actions; neither uses DPO continuation.

Policy costs use held-out paired episodes, with each controller following its own action and observation history. Common phases and exogenous noise pair methods. Fitting, target construction, and model selection use the fitted model; the plant supplies observation data and held-out evaluation. Unless stated otherwise, intervals are descriptive, unadjusted 95\% Student-$t$ intervals over eight fitted-model seed means. Episodes and antithetic pairs do not count as independent fits. Feedback-MPC students in Section~\ref{sup:feedbackdistill} use the same model seeds as the reported BG policies but different policy representations; Section~\ref{sup:matchedstorage} repeats the comparison with matched representation families.

\paragraph{AI-assisted implementation.}\label{sup:aidevelopment}
ChatGPT/Codex (OpenAI) assisted with implementing and debugging simulation, policy-construction, evaluation, and plotting code under the author's direction. The reported numerical results and plots derive from computational runs and saved outputs using the models, algorithms, and evaluation protocols specified in this supplement. ChatGPT/Codex and Claude (Anthropic) also supported critical examination of ideas and mathematical arguments and manuscript editing; the mathematical claims are justified by the stated assumptions and proofs in the main paper and Technical Appendix.

\subsection{Active-sensing docking}\label{sup:dockingprotocol}
\paragraph{Task and protocol.}
This thrust-control variant has a 12-decision Gaussian state $X_t=(p_t,v_t,w_t)$ (position, velocity, and hidden disturbance). The two continuous controls are normalized thrust $f_t\in[-1,1]$ and sensor power $a_t\in[0,1]$. Thrust changes the state drift; sensor power reduces the variance of the \emph{next} position observation:
\begin{equation}
\operatorname{Var}(Y_{t+1}\mid X_{t+1},a_t)
=c_t^2\left(r_{\rm floor}^2+
\frac{r_{\rm base}^2}{1+\kappa a_t}\right),\qquad
c_t=\begin{cases}
10,&\text{late-noise condition and }t\geq6,\\
1,&\text{otherwise}.
\end{cases}
\label{eq:dockingsensing}
\end{equation}
All methods know the noise schedule and use it in fitting/simulation. Costs cover position, velocity, thrust, and sensing; success requires $|p_T|\leq0.25$, $|v_T|\leq0.35$. Uncertainty reduction earns no separate reward. The last sensor action is fixed at zero because no later decision uses its observation. Each of eight models is fitted by Gaussian observation likelihood to 8,192 randomized-action trajectories without latent labels. DPO uses 1,600 updates; BG uses 256 nodes/date, 128 derivative paths, and 256 independent acceptance paths. Evaluation uses 512 paired episodes per seed with each controller's own belief history.

\paragraph{Independent planning comparisons.}
OLF-MPC (L-BFGS) and OLF-MPC (CEM) optimize future thrust sequences and replan after observations, without DPO initialization or continuation. Their fixed future thrust cannot benefit from future observations, so sensing only adds expense; its optimal value is zero and is eliminated analytically. Replanning does not restore that omitted ex ante information value. Sensing changes future estimation uncertainty: the separate sensor input buys state information while model parameters remain fixed. Independent feedback MPC uses PyTorch L-BFGS on a sample-average objective through simulated observations/filter updates. It optimizes datewise sigmoid-parameterized sensing and tanh-transformed affine belief-mean thrust feedback, not an exact POMDP solution. Table~\ref{tab:compactdocking} in the main paper reports BG policies stored by physical-action fitting; Section~\ref{sup:feedbackdistill} documents the separately trained feedback-MPC students.

A separately validated acados SQP solves the same zero-sensing OLF objective with 128 planning paths. All 98,304 decisions across both regimes and eight seeds return successful solver status. Its closed-loop costs, 4.5675 and 15.1203, remain close to the L-BFGS and CEM OLF results. Faster solution of this objective does not recover the omitted value of future information. The numerical belief-space DP in Section~\ref{sup:scalardp} concerns a separate scalar companion, not this three-state task.

\subsection{Thrust identification and comparison protocol}\label{sec:rocketprotocol}
The planar thrust model has latent state $X_t=(p_t,v_t,w_t)\in\R^6$ and two continuous commands $u_t\in\R^2$. Its fitted equations are
\begin{align}
p_{t+1}&=p_t+\Delta v_t+\tfrac12\Delta^2(G_\theta u_t+w_t),\nonumber\\
v_{t+1}&=v_t+\Delta(G_\theta u_t+w_t),\nonumber\\
w_{t+1}&=\diag(\rho_\theta)w_t+f_\theta(C_t)
 +\diag(\sigma_{w,\theta})\xi^X_{t+1},\nonumber\\
Y_{t+1}&=p_{t+1}+\diag(\gamma_\theta(C_t))w_{t+1}
 +\diag(\sigma_{y,\theta})\nu_{t+1}.
\label{eq:rocketmainmodel}
\end{align}
Here $\Delta=0.25$, $T=16$, and $C_t=(\sin(\varphi+0.43t),\cos(\varphi+0.43t))$ has random initial phase. The kinematics and initial Gaussian prior are supplied. Observation likelihood learns $G_\theta$, wind persistence, process/observation noise scales, and the context maps $f_\theta,\gamma_\theta$, supplied by a shared network with two 24-unit tanh layers. The plant's $G_\star=\left(\begin{smallmatrix}1&0.35\\0.25&0.75\end{smallmatrix}\right)$ is withheld from fitting and policy construction.

Each of eight seeds supplies 2,048 randomized-action observation trajectories for fitting, 320 for validation, and 512 for predictive testing. Fitting uses 800 Adam updates (batch 128, learning rate 0.006), selecting validation checkpoints every 25 updates. Only contexts, actions, and noisy observations are used. The same eight identified-model checkpoints serve both effort conditions and objectives.

Controllers use the fitted posterior $\N(m_t,P_t)$. Initial uncertainty, process noise, and observation noise are present in the plant and the fitted model. The fitted Kalman recursion is exact, but its accuracy for the plant is evaluated empirically. Here phase and time determine $P_t$ independently of commands, so thrust tests passive belief feedback. Conditional on context, its dynamics remain linear-Gaussian under either objective.

The quadratic stage and terminal costs are
$\ell_{\rm q}(X,u)=\tfrac12(X^\top Q_cX+u^\top R_cu)$ and
$\Phi_{\rm q}(X)=\tfrac12X^\top Q_fX$, with
\begin{align*}
Q_c&=\diag(0.10,0.10,0.03,0.03,0.005,0.005),\\
Q_f&=\diag(6,6,2,2,0.02,0.02),\\
R_c^{\rm mild}&=\diag(0.12,0.12),\qquad
R_c^{\rm aniso}=\diag(0.02,0.50).
\end{align*}
These are different effort costs in fixed physical coordinates, not a change of units. The nonquadratic objective adds smooth position and coupled-command penalties:
\begin{align}
\ell_{\rm nq}(X,u)&=\ell_{\rm q}(X,u)
 +0.025(p_1^4+p_2^4)+0.015(u_1+0.5u_2)^4,\nonumber\\
\Phi_{\rm nq}(X)&=\Phi_{\rm q}(X)+0.03(p_1^4+p_2^4).
\label{eq:rocketnonquadratic}
\end{align}
For each objective and frozen model, a two-layer, 48-unit tanh DPO actor maps $(m_t,\diag(P_t),C_t,t/T)$ to unbounded commands. Phase and time determine the full covariance schedule. Training uses 600 Adam updates (batch 128, learning rate 0.003); DPO+150 continues the optimizer. Model and actor gradient norms are clipped at 5.

\paragraph{BG construction and independent comparators.}
For thrust, each date uses 128 DPO-rollout belief nodes, perturbed by 0.2 times coordinate-wise empirical standard deviations with a 0.05 scale floor. Corrections use 128 antithetic derivative paths, 256 independent antithetic acceptance paths, Hessian floor 0.01, and displacement cap 2. Static BG uses DPO continuation; backward BG uses the fitted later-date corrections. Their nodes, random numbers, features, and safeguards are matched. LQG features are affine in $m_t$ with phase harmonics; nonquadratic features add selected squared, cubic, and cross terms. Ridge penalties are 0.001 and 0.01, respectively. Nonquadratic deployment also caps DPO-relative displacement at 2. Execution evaluates the fitted map without a new acceptance test, so construction-node acceptance does not certify off-node improvement.

OLF-MPC (L-BFGS) and OLF-MPC (CEM) implement OLF control: they optimize the entire remaining open-loop command sequence, execute its first action, and replan with their own shifted previous plans. They use no DPO continuation, initialization, action ball, or acceptance test. L-BFGS allows 80 iterations; CEM uses full sequence covariance, 512 candidates, and 24 generations. Exact Gaussian quadratic/fourth moments remove objective Monte Carlo noise; LQG action-independent constants are omitted. For nonquadratic thrust, restricted feedback MPC instead optimizes datewise affine belief-mean policies by PyTorch L-BFGS with exact Gaussian moments, without DPO/BG actions, continuation, or initialization. It uses the same frozen model and evaluation paths, executing its first action and replanning after each observation. These task-specific planners have different budgets and are neither exact nonlinear feedback oracles nor reproductions of PETS, POPLIN, or belief-space iLQG/DDP.

\subsection{Thrust reference checks}\label{sup:thrustresults}
Section~\ref{sup:lqgresults} checks fitted-model action accuracy against Riccati feedback. Section~\ref{sup:quarticresults} specifies the controlled continuation comparison after adding quartic costs.
\subsubsection{LQG reference}\label{sup:lqgresults}
Backward Riccati recursion gives the exact fitted-model feedback policy used to test Proposition~\ref{prop:lqgbackwardexact}. The action panel uses common held-out DPO histories at five dates with cold planner starts, while policy costs use each controller's own history. BG reduces fitted-optimal action MSE by more than 99.99\%, to $1.94\times10^{-6}/1.72\times10^{-6}$ in mild/anisotropic conditions. Its plant-optimal MSEs, $0.002029/0.001966$, approach the fitted-model--plant oracle discrepancy; these squared errors do not give an additive error decomposition.

The cached fitted-LQG controller builds float32 condensed kernels, solves selected inverse columns in float64, and stores float32 coefficients. The independent action reference uses float32 Riccati recursion. This rounding explains residual self-errors of $10^{-11}$--$10^{-10}$; neither smaller numerical residuals nor sampled cost reversals indicate improvement on the mathematical optimum.

\subsubsection{Nonquadratic continuation comparison}\label{sup:quarticresults}
The quartic costs in \eqref{eq:rocketnonquadratic} remove the Riccati reference. Static and backward BG share nodes, derivative and acceptance streams, features, and safeguards; only the continuation used to generate targets changes. Section~\ref{sup:regressionrefit} repeats this comparison with matched target-validated representations. Section~\ref{sec:rocketdiagnostics} gives derivative checks; neither numerical MPC nor BG is assumed to be the optimal feedback policy.

\subsection{Nonlinear robot arm}\label{sup:robotprotocol}
\paragraph{Model and task.}
We use a synthetic torque-driven planar two-link arm with state $X=(q,v)\in\R^4$, observed goal $g\in\R^2$, and two bounded controls $u\in[-2,2]^2$. The horizon is $T=16$ with $\Delta=0.14$. The supplied physical family contains coupled inertia, Coriolis, gravity, drag, and actuator terms. Writing its acceleration as $a_\theta(q,v,u)$, the stochastic transition and observation are
\begin{align}
q_{t+1}&=q_t+\Delta v_t+\tfrac12\Delta^2\bigl(a_\theta(q_t,v_t,u_t)+\Sigma_{w,\theta}\xi^X_{t+1}\bigr),\nonumber\\
v_{t+1}&=v_t+\Delta\bigl(a_\theta(q_t,v_t,u_t)+\Sigma_{w,\theta}\xi^X_{t+1}\bigr),\qquad
Y_{t+1}=q_{t+1}+\Sigma_{y,\theta}\nu_{t+1}.
\label{eq:robottransition}
\end{align}
The diagonal noise scales are learned along with physical coefficients. The unit-link fingertip map is $r(q)=(\cos q_1+\cos(q_1+q_2),\sin q_1+\sin(q_1+q_2))$. Initial mean angles are uniform on $[-1.2,1.2]^2$, initial mean velocities are zero, and $P_0=\diag(0.10^2,0.10^2,0.24^2,0.24^2)$. Goals are $r(q^{\rm goal})$ for angles uniform on $[-0.9,0.9]^2$. An initial noisy angle measurement updates this prior. With $e=r(q)-g$, the task costs are
\begin{align}
\ell(x,u,g)&=0.17\|e\|^2+0.025\|v\|^2+0.012u_1^2+0.055u_2^2\nonumber\\
&\quad+0.025\sum_{j=1}^2e_j^4+0.012(u_1+0.65u_2)^4,\nonumber\\
\ell_T(x,g)&=5\|e\|^2+0.35\|v\|^2+0.28\sum_{j=1}^2e_j^4.
\label{eq:robotcost}
\end{align}
Terminal attainment requires $\|e_T\|<0.20$ and $\|v_T\|_\infty<0.35$, a threshold diagnostic distinct from total cost. The arm is an in-house simulator, not a standard Reacher or hardware benchmark.

\paragraph{Learning and comparisons.}
Each of eight seeds fits 13 physical/noise parameters from 2,048 randomized-action observation trajectories, with 512 validation and 512 test trajectories. Fitting uses 380 Adam updates (batch 128) and the EKF Gaussian innovation score, observing only torques, noisy angles, and the prior. Hidden velocities and plant parameters are withheld. This is parameter learning in a supplied nonlinear family. The EKF maintains full covariance; the two-layer, 64-unit DPO actor and regressed policies use means, covariance diagonals, goal, and time. That compressed input need not be sufficient.

DPO uses 500 full-rollout updates (batch 128), with a +150 arm. Both BG variants share 96 DPO nodes/date, 48 derivative paths, and 96 independent acceptance paths; curvature floor, DPO-relative radius, and ridge are $0.02,0.75,0.02$. Residual features include means, goals, fingertip errors, covariance diagonals, squares, and a joint-angle cross term. Backward BG differentiates complete model/EKF rollouts through regressed later policies, retaining state-input derivatives while freezing coefficients; static BG retains DPO. Node acceptance precedes regression, with no learned critic.

The independent OLF controllers OLF-MPC (L-BFGS) and OLF-MPC (CEM) optimize the full remaining bounded open-loop torque sequence using 24 conditional state/process paths; they use, respectively, 30 L-BFGS iterations and eight CEM iterations with 128 candidates. They invoke neither DPO nor BG continuation and replan after every observation, carrying the previous plan between dates. The nonlinear OLF solver has no global-optimality certificate.

Independent feedback MPC first solves an open-loop plan, then uses PyTorch L-BFGS to optimize affine functions of future belief-mean deviations followed by a scaled tanh torque map through model/EKF trajectories. The bounded action law is therefore nonlinear in the belief mean. Its zero-gain ablation removes observation response while retaining the same 24 paths, initial 30 open-loop iterations, and 30 additional iterations. Neither uses DPO/BG. This ablation matches sampling and iteration limits, not time, and tests feedback response beyond an extra solve.

The zero-gain comparison lowers mean plant cost from $6.36050$ to $6.28979$ when future belief response is enabled: a paired difference of $-0.07071$ $[-0.08023,-0.06118]$, with similar attainment. Extra open-loop optimization alone gives an unresolved difference of $-0.02133$ $[-0.05262,+0.00996]$ relative to the original OLF planner. These fixed-budget ablations support a contribution from feedback response without establishing an optimal nonlinear policy. Evaluation uses 256 paired episodes per seed.

\section{Policy storage and target validation}\label{sup:regressionrefit}
The reported backward BG policies store accepted local actions using target-validated maps in thrust and the arm and a physical-action fit in docking. Map choices use fitted-model target-action diagnostics, never plant cost: thrust and arm use construction-node cross-validation, while docking uses a terminal held-out target pilot. These choices were developed after earlier diagnostics, so the study is exploratory. A fresh exogenous-noise evaluation panel reuses the same eight learned-model seeds and is therefore an independent episode check, not a new-model replication.

\paragraph{Nonquadratic thrust.}
The stored policy uses a 27-coefficient polynomial direct-action fit, optionally augmented by a 128-center local RBF correction. Kernel residuals subtract the radius-two-capped polynomial prediction from target actions; the combined map is capped again relative to DPO at deployment. Polynomial ridge is $0.01$ with intercept penalty $10^{-12}$; kernel penalties are selected from $\{0.01,0.1,1\}$. Technical Appendix~\ref{app:nltargetstorage} gives this operator and its storage-error bound. Four-fold target-action cross-validation admits the additional component datewise only when it lowers MSE by at least 5\%; all subsequent BG targets are recomputed using the resulting corrected tail. The construction uses 128 nodes, 128 derivative paths, and 256 independent acceptance paths per date. On the fresh mild/anisotropic panels, backward BG costs $1.606786/2.052751$, versus static BG's $1.609062/2.057187$. Representative seed-0 generated-C date-zero action times are $5.820/5.280\,\mu$s for backward BG and $5.587/5.385\,\mu$s for static BG. Backward construction takes $2.699/2.632$ seconds per fitted model on the local RTX~4070 SUPER, excluding shared model/DPO learning and node generation.

On fresh mild-thrust episodes, the paired backward-minus-static BG cost is $-0.00228$ (interval $[-0.00364,-0.00091]$); on anisotropic episodes it is $-0.00444$ (interval $[-0.00577,-0.00310]$). On the original task panel, the backward-minus-online-feedback-MPC difference is $+0.002213$ for mild thrust (interval $[0.001521,0.002905]$) and $+0.011277$ for anisotropic thrust (interval $[0.007760,0.014793]$). These are exploratory 95\% Student-$t$ intervals across eight paired fitted-model seed means.

\paragraph{Docking.}
The actor retains its 16 features, smooth output map, and action bounds. Policy fitting minimizes deployed physical-action error, with coefficient penalty $10^{-4}$; a terminal-date $2\times2$ ablation attributes most target-MSE reduction to the action loss rather than the penalty. Rebuilding the corrected tail gives backward BG original-panel costs $0.822330/4.025040$ for regular/late-noise docking and fresh-panel costs $0.821157/3.940916$. Original-panel success is $82.57\%/50.68\%$ and fresh-panel success is $82.25\%/49.80\%$. Static BG costs $0.824931/4.031797$ on the original panel. The original-panel backward-minus-static regular cost difference is $-0.002601$ $[-0.005942,+0.000740]$, while the fresh-panel difference is $-0.005105$ $[-0.007677,-0.002534]$. Relative to restricted online feedback MPC on original paths, backward BG is cheaper by $0.01335$ $[0.00455,0.02215]$ in regular docking but costlier by $0.07158$ $[0.01710,0.12606]$ under late noise. These are descriptive eight-seed paired intervals, not comparisons to an optimal belief policy. Static/frozen-target/backward builds average $4.44/5.44/5.27$ seconds per model on the local RTX~4070 SUPER; separate A6000 student times below are not pooled with them. Seed-0 native-C backward actions take $1.389/1.848\,\mu$s at date zero.

\paragraph{Two-link arm.}
Four-fold construction-target cross-validation selects among residual maps and up to 32-center Nystr\"om maps, then the corrected tail is rebuilt. On separate fitted-model target nodes, action MSE is $0.043954$; this measures storage of accepted actions rather than plant performance. Backward BG costs $6.678051$ on the original and $6.688447$ on the fresh episode panel. Original/fresh attainment is $59.23\%/57.08\%$. Static BG costs $6.701478$ with $62.26\%$ attainment on original paths. Restricted online feedback MPC costs $6.289789$. The stored arm maps average 31,065 bytes; backward/static BG date-zero native-C times are $3.837/3.659\,\mu$s over eight seeds. Backward construction takes $9.47$ seconds per model on the local RTX~4070 SUPER. The construction time conditions on the fitted model, trained actor, and prepared nodes, excluding their preparation costs.

\section{Feedback-MPC student construction}\label{sup:constructionprotocol}
The feedback-MPC student benchmark freezes the eight fitted models and DPO actors in each condition. Student labels come from restricted feedback optimization at DPO-visited beliefs, and the fitted policy executes without online planning. Common model learning, actor training, and node generation are excluded from reported policy-construction times. The historical student representations differ from the task-section backward BG maps, so those cost differences are descriptive. Section~\ref{sup:matchedstorage} gives the separate common-storage comparison.

\subsection{Independent OLF solver verification}\label{sup:olfverification}
The nonquadratic-thrust OLF planner independently minimizes exact expected quartic cost from zero, without DPO/BG continuation. It uses cached population objectives and up to four consecutive 80-iteration L-BFGS solves, stopping at maximum absolute gradient $10^{-5}$. All 256 date/seed/condition batches pass, with maximum $9.70\times10^{-6}$. Effort-coefficient bounds put Hessian eigenvalues above $\mu=0.12$ or $0.02$, so $\|\nabla f\|^2/(2\mu)$ bounds suboptimality by at most $1.69\times10^{-8}$. This verifies the OLF solve, not feedback optimality. The arm and docking feedback teachers remain finite-budget local solvers.

\subsection{Restricted feedback-MPC distillation}\label{sup:feedbackdistill}
\paragraph{Historical deployment comparison.}
Table~\ref{tab:feedbackdistillmain} compares the original task-section backward BG policies, independently fitted feedback-MPC students, and online feedback MPC. BG and student storage maps differ, so these costs are descriptive rather than a matched-capacity contrast. The principal common-storage comparison is main-paper Figure~\ref{fig:matchedstorage} and Section~\ref{sup:matchedstorage} below. Online feedback MPC replans along its own visited beliefs, so its gap from a student also includes visitation effects.

\begin{table}[htbp]
\centering\small
\caption{Eight-seed mean plant cost for the task-section backward BG policies, separately constructed feedback-MPC students, and online feedback MPC (lower is better). BG costs match the preceding task tables; student maps differ in capacity. Student teacher budgets are 320 feedback iterations for thrust, 30 projected iterations for the arm, and 60 iterations for docking. Online feedback MPC replans along its own trajectories.}
\label{tab:feedbackdistillmain}
\begin{tabular}{lrrr}
\toprule
Task / condition & Backward BG & Feedback-MPC student & Online feedback MPC\\
\midrule
Thrust, mild & 1.65568 & 1.65656 & 1.65346\\
Thrust, anisotropic & 2.15022 & 2.15654 & 2.13894\\
Robot arm & 6.67805 & 6.73788 & 6.28979\\
Docking, regular & 0.82233 & 0.82867 & 0.83568\\
Docking, late noise & 4.02504 & 4.05233 & 3.95346\\
\bottomrule
\end{tabular}
\end{table}

\paragraph{Frozen inputs and target generation.}
For each condition, the feedback teacher independently optimizes future belief-dependent response rules and labels the first action at each DPO-visited belief. Online feedback MPC instead replans along its own trajectories; its gap from the offline student combines target quality, policy fitting, and visitation effects.

The arm uses 96 nodes per date over 16 dates, 24 conditional paths, and feedback checkpoints at 4, 16, 30, 80, and 320 iterations after an independent 30-iteration OLF initialization. Its DPO-anchored \texttt{RefinedPolicy} has ridge $0.02$ and an action radius of $0.75$; the main comparison uses radius-projected 30-iteration labels, matching the direct online planner's feedback-iteration budget. Evaluation uses 256 paired episodes per seed. Docking uses 256 nodes per date over 12 dates, 128 antithetic planning paths, and a 60-iteration L-BFGS cap. It fits first physical actions into a DPO-anchored 16-feature smooth residual policy with ridge $10^{-4}$ and raw-action radius two. The feedback teacher jointly optimizes sensing and affine posterior-mean thrust. Evaluation uses 512 paired episodes per seed. The feedback teachers have no global optimality or convergence certificates.

\begin{table}[htbp]
\centering\footnotesize
\setlength{\tabcolsep}{3pt}
\caption{Feedback-MPC student benchmark: eight-seed mean plant cost, attainment/success, and mean sensing, with median offline construction seconds on one A6000 per seed. Arm values use the projected 30-iteration student. The online planner has no comparable offline build; its policies are replanned along their own trajectories. Backward BG values are reported in the main task tables with different stored maps.}
\label{tab:feedbackstudentdetails}
\begin{tabular}{llrrrr}
\toprule
Task & Controller & Cost & Success/attain (\%) & Mean sensor & Build (s)\\
\midrule
Arm & Feedback-MPC student & 6.73788 & 56.84 & -- & 57.83\\
 & Online feedback MPC & 6.28979 & 66.02 & -- & --\\
\addlinespace
Docking, regular & Feedback-MPC student & 0.82867 & 81.96 & 0.3322 & 272.27\\
 & Online feedback MPC & 0.83568 & 82.42 & 0.3387 & --\\
\addlinespace
Docking, late noise & Feedback-MPC student & 4.05233 & 50.63 & 0.8332 & 266.10\\
 & Online feedback MPC & 3.95346 & 49.98 & 0.5240 & --\\
\bottomrule
\end{tabular}
\end{table}

Table~\ref{tab:feedbackstudentdetails} separates stored-student performance from online feedback planning. Increasing arm feedback iterations from 30 to 320 changes student cost from $6.73788$ to $6.73866$, while median construction time rises from $57.83$ to $436.23$ seconds. Regular-docking student cost minus online feedback MPC is $-0.00701$ $[-0.01610,+0.00208]$; under late noise it is $+0.09887$ $[+0.05022,+0.14752]$. These are exploratory two-sided 95\% Student-$t$ intervals over eight paired seed means, without multiplicity adjustment. Episodes within a seed are not independent model fits.

\paragraph{Sensing fit and policy-class limits.}
In regular docking, the feedback student's training-node sensing MSE averaged over dates and seeds is $0.0199$, with the residual cap active for $0.0093$ of node actions. Under late noise these quantities are $0.2952$ and $0.3609$. At date three, the late-noise feedback teacher's mean sensing is $0.2084$ while the fitted student's is $0.9962$. This mismatch is consistent with restricted DPO-centered storage but does not establish the cause of the closed-loop cost difference.

\paragraph{Nonquadratic thrust.}
The thrust student uses 128 nodes per date over 16 dates, ridge $0.01$, a DPO-centered radius-two deployment cap, and 256 paired episodes per seed. Restricted affine feedback teachers start from independent OLF solutions with zero future gains and use checkpoints at 4, 16, 80, and 320 feedback iterations. OLF solves serve only as planner initialization; no OLF policy is distilled or reported. Raw and radius-projected feedback-student costs agree at the shown precision.

\begin{table}[htbp]
\centering\small
\setlength{\tabcolsep}{4pt}
\caption{Thrust feedback-MPC students: eight-seed mean plant cost across teacher iteration limits. The task-section backward BG policy is shown for context; it uses a different storage map. Lower is better.}
\label{tab:feedbackthruststudents}
\begin{tabular}{lrrrrr}
\toprule
Condition & Backward BG & Feedback 4 & Feedback 16 & Feedback 80 & Feedback 320\\
\midrule
Mild & 1.655677 & 1.738798 & 1.742606 & 1.656518 & 1.656558\\
Anisotropic & 2.150216 & 2.245274 & 2.219976 & 2.156664 & 2.156543\\
\bottomrule
\end{tabular}
\end{table}

Table~\ref{tab:feedbackthruststudents} shows large gains between the 16- and 80-iteration teacher checkpoints, with little further change at 320. Median A6000 construction times for feedback-80/feedback-320 are $39.42/139.78$ seconds (mild) and $39.66/127.20$ seconds (anisotropic). One anisotropic seed had an archived path-replay discrepancy of $1.37\times10^{-4}$ after server transfer; its paired-path mean discrepancy was about $1.4\times10^{-6}$. It was retained after a documented numerical-tolerance recovery. These student costs are near the corresponding task-section BG means, but the different policy maps prevent an isolated target-quality comparison.

\subsection{Matched policy-storage comparison}\label{sup:matchedstorage}
The task-section policies and the points in Figure~\ref{fig:introcostlatency} use their originally selected maps. To examine how map capacity affects the comparison, a separate $2\times2$ experiment crosses target construction (backward BG versus feedback-MPC teacher actions fitted into a student) with a compact or richer shared map family. The eight fitted models, DPO actors, DPO-visited nodes, action transformations and limits, fitting loss and ridge penalty, and target-validation rule where applicable are held fixed within each task and map family. Only the feedback side distills teacher actions: BG directly regresses its accepted local actions. Each BG policy is rebuilt backward at each map capacity using its own stored later-date policy, so the larger map changes both storage and subsequent continuation targets. The richer thrust and arm maps share a candidate family and selection rule across methods, although target validation can select different realized features. Thus the comparison controls a major representation asymmetry but does not isolate the causal effect of target construction alone.

Thrust compares a 27-coefficient polynomial map with an optional 128-center local RBF residual, admitted datewise by four-fold target-action validation after at least 5\% MSE improvement; ridge is $0.01$ and the DPO-relative action radius is two. Its archived feedback-320 teacher labels are held fixed for both student maps. The arm compares a 20-feature residual with a family allowing up to 32 Nystr\"om centers, selected by four-fold validation after at least 2\% improvement; ridge is $0.02$ and the action radius is $0.75$. Its teacher uses 30 OLF-initialization and 30 feedback iterations. Docking compares the fixed 16-feature map with the same map plus 32 Gaussian RBF features; both methods fit physical actions with ridge $10^{-4}$ and the same action transform and raw-action radius two. Its feedback teacher uses 60 L-BFGS iterations. Arm and docking teachers were recomputed locally; the two student maps in each condition use identical teacher node actions. The methods retain their respective finite target-generation budgets and have no global optimality certificates.

\begin{table}[htbp]
\centering\footnotesize
\setlength{\tabcolsep}{3pt}
\caption{Matched policy-storage comparison on held-out plant episodes. Each entry is a mean over eight fitted-model seeds; $\Delta$ is feedback-MPC student minus backward BG within the same map family, so positive values favor BG. Brackets give exploratory, unadjusted paired-seed 95\% Student-$t$ intervals. Thrust uses a separate 512-episode panel and rebuilt maps, as detailed below. Arm entries use evaluation stream 12; stream 22 is reported in the text.}
\label{tab:matchedstorage}
\begin{tabular}{llrrl}
\toprule
Task / condition & Map & BG cost & Student cost & $\Delta$ [95\% interval]\\
\midrule
Thrust, mild & Compact & 1.629484 & 1.629304 & $-0.000180$ [$-0.000419,+0.000059$]\\
 & Rich & 1.627984 & 1.627820 & $-0.000164$ [$-0.000405,+0.000078$]\\
Thrust, anisotropic & Compact & 2.086755 & 2.087783 & $+0.001027$ [$-0.003374,+0.005429$]\\
 & Rich & 2.083432 & 2.084063 & $+0.000631$ [$-0.002964,+0.004227$]\\
\addlinespace
Arm & Compact & 6.714405 & 6.737905 & $+0.023501$ [$+0.014112,+0.032889$]\\
 & Rich & 6.678051 & 6.687144 & $+0.009093$ [$-0.000731,+0.018918$]\\
\addlinespace
Docking, regular & Compact & 0.822330 & 0.828730 & $+0.006401$ [$+0.001554,+0.011248$]\\
 & Rich & 0.818491 & 0.820978 & $+0.002487$ [$-0.001124,+0.006099$]\\
Docking, late noise & Compact & 4.025040 & 4.050587 & $+0.025547$ [$+0.011568,+0.039526$]\\
 & Rich & 4.027685 & 4.059673 & $+0.031989$ [$+0.008778,+0.055199$]\\
\bottomrule
\end{tabular}
\end{table}

\paragraph{Arm teacher projection and student storage.}\label{sup:armstoragechain}
The retained matched-comparison logs separate the feedback teacher's raw action, its projection into the DPO-centered radius-$0.75$ action ball, and the deployed student prediction. Across eight seeds, projection affects $6.62\%$ of the construction actions (seed-mean range $4.17$--$8.20\%$). Raw-teacher versus DPO action MSE is $0.113911$, and raw-teacher versus projected-label MSE is $0.018848$. Relative to the projected labels, deployed compact/rich student MSE is $0.031271/0.024672$. Each squared error averages over action components and 96 DPO-visited nodes per date, then 16 dates and eight seeds. These are in-sample diagnostics of teacher projection and storage, not optimal-action errors or an additive decomposition of plant-cost loss. They differ from the BG accepted-target diagnostic on separate fitted-model nodes in Section~\ref{sup:regressionrefit}. The remaining planner gap motivates joint examination of storage and visitation mismatch; these logs do not identify their relative contributions.

\paragraph{Thrust policy and evaluation provenance.}
The task-section thrust table uses target-validated maps on 256 episodes per seed; the fresh-panel storage diagnostic reevaluates those same maps on another 256 episodes. The matched comparison rebuilds compact/rich BG maps and refits both students from the fixed feedback-320 labels, then evaluates all four maps on 512 new paired episodes. For model seed $s$, add $10{,}000s$ to each component of the base phase/noise seed pairs: $(7{,}100{,}000,7{,}100{,}001)$ for the task panel, $(7{,}500{,}000,7{,}500{,}001)$ for the fresh diagnostic, and $(8{,}500{,}000,8{,}500{,}001)$ for the matched panel. Checkpoint and input hashes link each panel to its policy records in the reproducibility archive. Cross-panel absolute cost differences mix evaluation and, for the matched study, policy changes; method contrasts use paired paths within each panel.

The rich-minus-compact change in $\Delta$ is $+0.000017$ [$-0.000154,+0.000187$] in mild thrust and $-0.000396$ [$-0.001403,+0.000611$] in anisotropic thrust. The matched-map thrust costs therefore establish neither BG superiority nor equivalence. In the arm, this change is $-0.014407$ [$-0.019118,-0.009697$] on stream 12. An independent 256-path evaluation stream on the \emph{same eight models} gives compact/rich differences $+0.026307$ [$+0.014936,+0.037677$] and $+0.010736$ [$+0.002632,+0.018841$], with a change of $-0.015570$ [$-0.023032,-0.008108$]. Rich-map student attainment remains below BG by $2.25/2.44$ percentage points across the two streams. The second stream is an episode replication, not another set of model fits.

In regular docking, the rich map reduces $\Delta$ from $+0.006401$ to $+0.002487$, a change of $-0.003913$ [$-0.006951,-0.000876$]. Its student cost changes by $-0.007752$ [$-0.013294,-0.002210$], while BG changes by $-0.003839$ [$-0.009460,+0.001782$]. Under late noise, BG has lower cost with both maps, but the change in $\Delta$ is $+0.006442$ [$-0.007282,+0.020166$]; its direction is unresolved. The new compact docking BG exactly reproduces the earlier physical-action-fit policy on all evaluated paths. The compact student was refitted after recomputing teacher actions on an RTX 4070 SUPER and is not bitwise identical to the earlier A6000 student; the earlier raw teacher actions were not retained. This distinction does not affect matching within the new $2\times2$ experiment.

Thrust and docking use 512 paired plant episodes per seed; each arm stream uses 256. Controllers follow their own closed-loop histories under common exogenous noise. Paired differences and their intervals treat the eight model/DPO seeds, not episodes, as independent units. Map selection never uses held-out plant cost. These intervals are descriptive and unadjusted for multiple comparisons. Neither a zero-containing interval nor the capacity interaction certifies policy equivalence or a pure representation-only effect.

\begin{table}[htbp]
\centering\small
\setlength{\tabcolsep}{5pt}
\caption{Additional offline construction for rich-map policies on one local RTX 4070 SUPER, mean seconds per model seed. The teacher labels are generated once and shared across both student map fits. Shared model/DPO learning, node preparation, held-out evaluation, and online action selection are excluded.}
\label{tab:matchedstoragebuild}
\begin{tabular}{lrrrr}
\toprule
Task / condition & Rich BG & Teacher & Rich student fit & Teacher + fit\\
\midrule
Arm & 9.343 & 80.624 & 2.151 & 82.776\\
Docking, regular & 5.085 & 309.602 & 3.791 & 313.393\\
Docking, late noise & 4.985 & 294.505 & 3.656 & 298.162\\
\bottomrule
\end{tabular}
\end{table}

Table~\ref{tab:matchedstoragebuild} gives teacher-plus-fit/BG time ratios of $8.86$, $61.63$, and $59.81$ for the arm, regular docking, and late-noise docking. These compare the tested implementations and finite budgets at their achieved costs; they do not estimate the minimum computation needed to attain a common accuracy. Native export and exploratory tuning are also excluded. Thrust feedback-320 teacher actions were reused from an A6000 archive whereas the matched map builds ran on the RTX 4070 SUPER; their times are not added into a same-device total. The native execution points in main-paper Figure~\ref{fig:introcostlatency} and the historical student and stage accounting in Figure~\ref{fig:historicaldeployment} refer to the original policies, not the matched-map policies in Table~\ref{tab:matchedstorage}.

\section{Numerical and deployment audits}\label{sup:currentnative}
Section~\ref{sup:nativekernels} checks native policy execution and solver status, while Section~\ref{sec:rocketdiagnostics} reports derivative checks and their numerical limits. Neither certifies global feedback optimality. The independent OLF convergence check is given in Section~\ref{sup:olfverification}.

\subsection{Historical deployment and offline stage accounting}\label{sup:offlineaccounting}
\begin{figure}[!htbp]
\centering
\includegraphics[width=\linewidth]{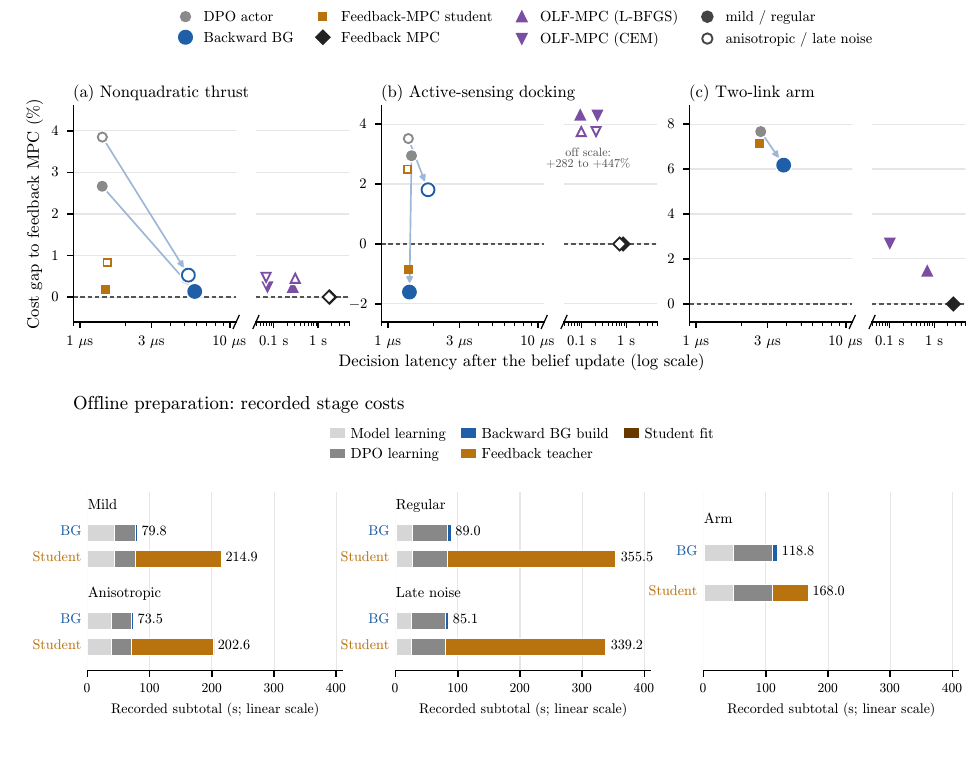}
\caption{\small Historical distinct-map deployment and mixed-hardware preparation in \textbf{(a)} nonquadratic thrust, \textbf{(b)} active-sensing docking, and \textbf{(c)} the two-link arm. \emph{Top:} eight-seed mean plant-cost gap $100(J/J_{\mathrm{FB}}-1)$ relative to online feedback MPC versus decision latency. Blue circles are the original task-section backward BG policies; orange squares are separately constructed feedback-MPC students with different storage maps. Filled/open symbols distinguish mild/anisotropic thrust and regular/late-noise docking. Arrows show DPO-to-BG correction; docking OLF costs are off scale. \emph{Bottom:} recorded eight-seed mean stage times for these policies: common model and DPO learning, followed by backward BG construction or feedback-teacher optimization and student fitting. Bar labels give recorded subtotals. These recorded subtotals are not same-device speed comparisons. Timing scope, hardware, and excluded stages are documented in Sections~\ref{sup:offlineaccounting}--\ref{sup:nativekernels}.}
\label{fig:historicaldeployment}
\end{figure}

The lower row of Figure~\ref{fig:historicaldeployment} accounts for the same deployed policies as the upper row: target-validated backward BG in thrust and the arm, the rebuilt physical-action fit in docking, and feedback-320, feedback-30 projected, and feedback-60 students, respectively. It does not substitute the earlier BG storage maps. Both methods inherit the same fitted model and DPO actor within each condition. The bars include their original learning times, then either the full recorded backward BG build (including its own regression and checks) or feedback-teacher construction and student fitting. No separate distillation is applied to BG.

For additive accounting, each segment is the mean recorded time over eight seeds, and each bar label is the mean of per-seed sums. These statistics differ from the construction-time medians elsewhere. The recorded subtotals for BG/student are $79.8/214.9$ seconds in mild thrust, $73.5/202.6$ in anisotropic thrust, $89.0/355.5$ in regular docking, $85.1/339.2$ in late-noise docking, and $118.8/168.0$ in the arm. Model learning and feedback-teacher construction use archived A6000 runs; thrust DPO learning and the current BG rebuilds use local RTX 4070-series runs, while docking/arm DPO learning uses A6000 runs. Hence these sums are historical stage accounting, not synchronized, same-host training benchmarks, and do not establish a hardware-controlled speed ratio.

Node preparation, data collection, held-out evaluation, native export, and exploratory tuning are excluded. Missing components are not treated as zero-cost work. BG correction and regression are shown jointly because separate timers are not available for every task. The arm student-fit timer covers both raw and projected variants; the plotted policy uses the projected variant, and the small recorded joint fit cost is retained without an unmeasured allocation. Thrust teacher time includes its OLF initialization and recorded checks. The per-seed stage ledger, source hashes, and exclusions are retained with the figure. The lower panels quantify recorded preparation of the tested policies rather than minimum achievable implementation cost.

\subsection{Native policy kernels and acados}\label{sup:nativekernels}
Native policy times use one AMD Ryzen 9 7950X3D CPU thread with preloaded weights and beliefs, excluding filtering, Python-call overhead, construction, and compilation. Main-paper marker $\mathrm C$ denotes the eight-seed median at date zero; $\mathrm C_0$ denotes seed-0 audits of target-validated thrust and action-fit docking. Repeated calls are not independent fits. Native/PyTorch action discrepancies are at most $6.06\times10^{-6}$ across these audits and $1.05\times10^{-6}$ for target-validated arm maps. Section~\ref{sup:regressionrefit} gives map storage and timing details; archived CSVs retain dates, seed scopes, and repetitions.

\paragraph{Execution of the matched checkpoints.}
Table~\ref{tab:matchednative} measures the saved compact/rich BG and student maps whose costs appear in Table~\ref{tab:matchedstorage}. It covers all 160 policies from 40 condition--seed sets, with eight measured seeds in every condition. Arm Eval.~1 and Eval.~2 use the same policies and share one timing row.

The audit uses one pinned logical CPU (core index 2) on an AMD Ryzen 9 7950X3D, GCC 16.2.0 with \texttt{-O3 -fno-fast-math}, and double arithmetic retaining the stored parameter values. Each native call evaluates the actor, features, selected kernel, cap, and action transform from a supplied belief. Filtering, data loading, offline work, and per-action Python/FFI overhead are excluded. All four methods cycle through the same 64 construction beliefs per date, with 2,000 warm-up calls and seven repetitions of 20,000 calls; method order rotates. Each table entry is the median across seed medians, with the seed range in brackets. Repetitions are timing measurements, not additional trained seeds.

\begin{table}[htbp]
\centering\footnotesize
\setlength{\tabcolsep}{3pt}
\caption{Native matched-policy action time at date zero, in $\mu$s: seed median [minimum, maximum]. $n$ is the number of measured model seeds. These are the realized maps from Table~\ref{tab:matchedstorage}, with all eight cost-evaluation seeds measured in each condition.}
\label{tab:matchednative}
\begin{tabular}{lcrrrr}
\toprule
& & \multicolumn{2}{c}{Compact map} & \multicolumn{2}{c}{Rich map}\\
Task / condition & $n$ & BG & Student & BG & Student\\
\midrule
Thrust, mild & 8 & 1.49 [1.46,1.50] & 1.48 [1.46,1.49] & 5.46 [5.43,5.50] & 5.46 [5.43,5.50]\\
Thrust, anisotropic & 8 & 1.47 [1.46,1.48] & 1.47 [1.46,1.50] & 5.44 [1.46,5.47] & 5.44 [5.43,5.46]\\
Arm & 8 & 2.39 [2.34,2.44] & 2.39 [2.35,2.42] & 3.39 [3.34,3.45] & 3.37 [3.34,3.44]\\
Docking, regular & 8 & 1.47 [1.43,1.51] & 1.48 [1.43,1.51] & 2.50 [2.44,2.52] & 2.49 [2.44,2.52]\\
Docking, late noise & 8 & 1.48 [1.43,1.66] & 1.48 [1.44,1.57] & 2.51 [2.46,2.63] & 2.51 [2.46,2.57]\\
\bottomrule
\end{tabular}
\par\smallskip\begin{minipage}{\linewidth}\footnotesize
In anisotropic thrust, date-zero validation for BG seed~6 retained the polynomial-only map; this accounts for the $1.46\,\mu$s minimum in the rich candidate family.
\end{minipage}
\end{table}

Across 147,456 audited policy--state inputs at all dates, native C and float64 PyTorch evaluations of the same frozen maps agree within $2.30\times10^{-14}$. The maximum discrepancy from archived float32 evaluations is $1.33\times10^{-5}$, accounted for by float32 arithmetic differences on these inputs. Checkpoint hashes, raw actions, initial diagnostic criteria and outcomes, and precision checks remain in the archive. Timing also covers dates 8 and 15 in thrust/arm and 6 and 11 in docking, with all repetitions and seed ranges archived. Realized kernel selection can vary across methods, dates and seeds, so a shared candidate class need not imply identical latency.

\paragraph{Feedback-student execution in Figure~\ref{fig:historicaldeployment}.}
The student points use a separate date-zero native-C audit of all eight saved feedback-target policies per condition; the DPO actors were timed as controls. Each kernel cycles through 64 common DPO-visited beliefs for seven repeats of 50,000 calls after 2,000 warm-up calls, rotating method order; the plotted student times are medians of eight seed medians. Native/PyTorch parity at all dates in the 40 condition--seed combinations has maximum error $5.30\times10^{-7}$. Student labels use feedback-320 in thrust, projected feedback-30 in the arm, and feedback-60 in docking (Section~\ref{sup:feedbackdistill}); their costs are the archived paired evaluations. The plotted backward BG points instead use the task-section policies and their $\mathrm C$ or $\mathrm C_0$ timings, including the target-validated thrust/arm maps and the physical-action docking fit.

\paragraph{Thrust acados inputs and protocol.}
The solver audit fixes one held-out belief for each condition, seed, and date $0,8,15$ (48 beliefs). Acados timings take each seed's median over nine repetitions, then the median across eight seeds. Plans start from independent model-LQR initialization at every audited date; previous plans are not reused. Table~\ref{tab:currentacados} reports date-zero times and the accompanying solver-status counts.

\begin{table}[htbp]
\centering\small
\caption{Date-zero acados thrust audit, median milliseconds over eight seeds. Native time covers the solver; full decision time also covers Python-side preparation and action extraction. The last column counts status-zero and iteration-limit cases among the eight seeds. These fixed-belief checks have no closed-loop plant-cost evaluation.}\label{tab:currentacados}
\begin{tabular}{lrrr}
\toprule
Planner / condition & Native & Full decision & Status 0 / 2\\
\midrule
OLF SQP / mild &0.187&1.104&8 / 0\\
OLF SQP / anisotropic &0.263&1.207&8 / 0\\
Feedback RTI / mild &1.93&3.41&8 / 0\\
Feedback RTI / anisotropic &2.54&4.02&8 / 0\\
Feedback SQP / mild &141&143&0 / 8\\
Feedback SQP / anisotropic &148&149&0 / 8\\
\bottomrule
\end{tabular}
\end{table}

Table~\ref{tab:currentacados} reports different solution qualities. OLF SQP succeeds on all 48 beliefs, with maximum independent Newton-reference objective discrepancy $7.11\times10^{-15}$. Feedback SQP reaches its 80-iteration limit at dates zero and eight (32 of 48 cases); separate maximum KKT stationarity residual is $6.82\times10^{-5}$. RTI completes one QP step, rather than certifying nonlinear convergence. At anisotropic date zero, its first-action distance to the numerical feedback reference has seed median 0.137. The PyTorch feedback reference is itself numerical, with maximum recorded gradient component $4.89\times10^{-4}$. The main paper therefore does not attach the reference's closed-loop cost to the faster RTI time.

\paragraph{Other native planner coverage.}
Docking OLF SQP is separately validated in closed loop and in the matched date-zero timing panel in Table~\ref{tab:compactdocking} of the main paper. Docking feedback acados instead has only fixed-belief feasibility diagnostics: all 16 RTI cases fail the declared constraint check; capped SQP passes eight regular and one late-noise case, all with iteration-limit status. Parallel diagnostic wall times are not standalone decision-latency measurements. No native feedback-acados arm result is included. The measured PyTorch planners remain distinct implementations.

The main paper also identifies PyTorch planner measurements on one A6000 by the marker $\mathrm G$. Those frozen-checkpoint, batch-one measurements pool dates $0,8,15$ on one common held-out DPO history per seed, with one warmup and three synchronized calls per date; each seed is summarized by its median, followed by the median across eight seeds. Planners start cold. Loading, shared filtering, construction, and cached coefficient preparation are excluded. The marker $\mathrm P$ instead denotes the docking seed-0, date-zero CPU pilot. These boundaries remain distinct from native kernel and acados times.

The source records are the 2026-09-27 native multiseed audit, its matched rocket acados report, and the docking OLF validation. The reproducibility package includes the aggregate CSVs and the audit reports in \texttt{timing/}.

\subsection{Derivative and numerical checks}\label{sec:rocketdiagnostics}
\paragraph{Derivative checks.}
Common-noise finite differences give thrust refreshed-tail gradient errors $2.07\times10^{-5}/3.66\times10^{-5}$. In the arm, the corresponding gradient error is $3.13\times10^{-4}$ and the smooth DPO-tail Hessian error is $2.12\times10^{-5}$. The actual refreshed-tail Hessian discrepancy reaches $0.0874$: radial and action caps prevent interpreting this as smooth-Hessian validation. Technical Appendix~\ref{app:regularity} specifies the distinction between a local curvature metric and the Hessian of the expected objective.

Arm test NLL improves from $-1.32966$ to $-1.54294$, only $0.00025$ above the plant-coefficient EKF score; both filters are approximate. Similar plant/fitted-model costs and persistent within-model planner gaps support investigating policy and representation error without identifying a unique cause.

\paragraph{Reproduction records.}
The experiment archive retains fitted-model and actor checkpoints, stored policy coefficients, random-seed and stream identifiers, paired episode outcomes, solver budgets, implementation versions, and file hashes. Current construction stores aggregate acceptance and safeguard statistics and raw node-fit errors; those raw errors precede deployment caps and are not off-node policy guarantees. Proposal, acceptance, and held-out evaluation paths are independent. The scalar DP reference and its grid/quadrature limits are specified separately in Section~\ref{sup:scalardp}.

\section{Numerical DP reference for scalar active-sensing docking}\label{sup:scalardp}
This stylized one-dimensional companion supplies a tractable numerical Bellman reference for an active-sensing belief state. It is a new problem, not an exact reduction of the three-state docking experiment in main paper, Section~\ref{sec:dockingmain}. The hidden state has Gaussian belief $\N(m_t,P_t)$ and obeys
\begin{align}
X_{t+1}&=X_t+0.5f_t+0.08\xi^X_{t+1},&
Y_{t+1}&=X_{t+1}+\eta_{t+1}(a_t),\nonumber\\
\operatorname{Var}(\eta_{t+1}(a_t))
&=0.025^2+\frac{0.45^2}{1+25a_t},&
(f_t,a_t)&\in[-0.4,0.4]\times[0,1].
\label{eq:scalardpmodel}
\end{align}
The process and observation noises are independent centered Gaussians, with $\xi^X_{t+1}\sim\N(0,1)$. There are 12 decisions and the initial belief is $\N(0.5,0.15)$. The physical stage and terminal costs are
$\ell(x,f,a)=0.02x^2+0.03f^2+0.03(a+0.5a^2)$ and
$\Phi(x)=0.3x^2+5x^4$. A zero sensor action retains a noisy baseline observation; it does not suppress observations. The final sensor action is zero because its observation arrives after the last control decision. The Gaussian belief transition and conditional Gaussian moments are exact for this model.

\paragraph{Numerical reference and policy costs.}
Belief-state DP uses a rectangular grid with 201 mean points, 81 variance points, 17 thrust choices, nine sensor choices, and seven Gauss--Hermite nodes. Its initial Bellman value is 0.066349. A combined refinement to 401 mean points, 161 variance points, 33 thrust choices, 17 sensor choices, and 11 quadrature nodes gives 0.066089. Thus the reference remains numerical: interpolation, finite action grids, quadrature, and extrapolation are not certified error bounds. Table~\ref{tab:scalardpcost} compares deployed policies on 512 common physical-noise paths for one trained DPO/BG seed; each controller follows its own belief trajectory. The restricted feedback MPC here optimizes belief-dependent future thrust \emph{and} sensing, whereas the three-state comparator in main paper, Section~\ref{sec:dockingmain} uses a datewise sensing schedule. Near equality to DP in this scalar task therefore does not validate the three-state comparator as optimal.

\begin{table}[htbp]
\centering\small
\caption{Scalar active-sensing companion: one training seed and 512 paired physical closed-loop paths. Gaps subtract the interpolated numerical DP policy's sampled cost; $\pm$ is one paired Monte Carlo standard error. Feedback MPC and OLF have finite optimization budgets.}
\label{tab:scalardpcost}
\begin{tabular}{lrr}
\toprule
Controller & Mean physical cost & Gap from DP \\
\midrule
Numerical DP policy & 0.066471 & reference \\
Restricted feedback MPC & 0.066518 & $0.000046\pm0.000341$ \\
Backward BG & 0.069550 & $0.003079\pm0.000779$ \\
DPO & 0.072665 & $0.006194\pm0.001231$ \\
OLF-MPC (L-BFGS) & 0.076225 & $0.009754\pm0.002713$ \\
DP with $a_t\equiv0$ & 0.076252 & $0.009781\pm0.002700$ \\
\bottomrule
\end{tabular}
\end{table}

Eight independently trained DPO actors and their one-sweep BG refinements were also evaluated on 20,000 \emph{common Gaussian innovation paths} per seed, integrating physical costs conditionally on each belief. Their mean conditional costs are 0.070864 (DPO) and 0.068430 (BG), with training-seed standard deviations 0.000847 and 0.000439. The numerical DP Bellman value is 0.066349. BG lowers DPO cost in all eight seeds; the mean paired reduction is $0.002434\pm0.000325$, where $\pm$ now denotes one \emph{training-seed} standard error. The ratio of mean improvements,
$(0.070864-0.068430)/(0.070864-0.066349)=53.9\%$,
describes the fraction of the DPO-to-numerical-DP gap closed by BG in this conditional-cost evaluation. It is distinct from the approximately 50\% ratio in the one-seed physical-path comparison of Table~\ref{tab:scalardpcost}.

\paragraph{Final-date local solve and regression.}
At $t=11$ there is no future feedback tail. With $\mu=m+0.5f$ and $S=P+0.08^2$, the exact conditional one-step objective is
\begin{equation}
q_{11}(m,P,f)
=0.02(m^2+P)+0.03f^2
 +0.3(\mu^2+S)+5(\mu^4+6\mu^2S+3S^2).
\label{eq:scalardpterminalq}
\end{equation}
Its bounded minimizer on $[-0.4,0.4]$ is unique and is computed by bisection. We replayed the \emph{original} BG construction at 64 DPO-visited nodes per seed (512 nodes total), using the saved policy checkpoints and the original derivative and acceptance random streams. The replay exactly reproduces every archived date-11 acceptance fraction, node action MSE, and accepted gain. All proposed terminal corrections were accepted; none reached the thrust trust-step bound. Table~\ref{tab:scalardpterminal} evaluates the original accepted target and the deployed regression map against the exact minimizer at those same nodes. One standard error across the eight training seeds is shown for conditional cost excess.

\begin{table}[htbp]
\centering\small
\caption{Final-date action diagnostic on the original 64 construction nodes per seed. Thrust RMSE and exact conditional one-step cost excess are first averaged over nodes, then over eight training seeds. The target is the accepted local BG proposal before regression.}
\label{tab:scalardpterminal}\label{tab:compactscalar}
\begin{tabular}{lrr}
\toprule
Action at $t=11$ & Thrust RMSE vs.\ exact & $q_{11}-q_{11}^\star$ \\
\midrule
DPO baseline & --- & $0.001400\pm0.000054$ \\
Accepted BG target & 0.003605 & $0.00000460\pm0.00000055$ \\
Fitted BG policy & 0.039431 & $0.000530\pm0.000080$ \\
\bottomrule
\end{tabular}
\end{table}

The target thrust lies within 0.01 of the exact action at 96.5\% of these nodes. Its mean one-step cost excess is less than 1\% of the fitted policy's mean excess. Thus the terminal construction finds nearly optimal node actions, while their regression into the executable policy loses most of that local precision. This isolates a representation bottleneck \emph{at the terminal construction nodes}; it does not allocate the entire policy's cost gap among theoretical error terms. Replacing only the deployed BG terminal thrust by the exact action, leaving all prior BG actions unchanged, reduces mean conditional full-policy cost from 0.068430 to 0.067486 across the eight seeds, a paired mean improvement of $0.000944\pm0.000182$ across seeds; all eight improve. Relative to the mean BG-to-Bellman gap of 0.002081, this is about 45\% descriptively, but a rollout--Bellman ratio is not an additive regret decomposition.

\paragraph{A second backward sweep.}
Applying a second sweep to the same DPO visitation nodes, with independent derivative and acceptance streams, improves only three of eight trained seeds. Mean conditional cost changes from 0.068430 after one sweep to 0.068711 after two, a paired difference of $+0.000281\pm0.000180$ across training seeds. This experiment does not support a reliable improvement from repeated sweeps; changing the future tail, local target, and regression together does not isolate Newton iteration.

\section{Quadratic thrust reference panel}\label{sup:lqgpanel}
This numerical calibration complements the affine LQG analysis in Technical Appendix~\ref{app:theory}. Table~\ref{tab:compactlqg} preserves the full reference comparison; the nonquadratic thrust, docking, arm, and matched-storage comparisons remain in the main paper. Action MSE uses common held-out DPO histories. The DPO-to-BG plant-cost differences are $-0.03694$ $[-0.03948,-0.03439]$ (mild) and $-0.06831$ $[-0.07126,-0.06536]$ (anisotropic), with exploratory 95\% paired $t$ intervals across eight seeds.
\begin{table}[H]
\centering\small
\caption{Quadratic thrust. Entries are mild / anisotropic eight-seed means; costs use 256 paired episodes per seed. Action MSE is relative to fitted Riccati feedback. The time marker specifies the implementation defined above.}
\label{tab:compactlqg}
\setlength{\tabcolsep}{4pt}
\begin{tabular}{lrrr}
\toprule
Controller & Plant cost & Fitted action MSE & Action time\\
\midrule
DPO &1.417321 / 1.849765&$4.58\!\times\!10^{-2}$ / $5.91\!\times\!10^{-2}$&$1.607/1.583\,\mu\mathrm s^{\mathrm C}$\\
DPO+150 &1.416492 / 1.844699&$4.74\!\times\!10^{-2}$ / $5.67\!\times\!10^{-2}$&$1.618/1.597\,\mu\mathrm s^{\mathrm C}$\\
Backward BG &1.380383 / 1.781456&$1.94\!\times\!10^{-6}$ / $1.72\!\times\!10^{-6}$&$0.063/0.063\,\mu\mathrm s^{\mathrm C}$\\
OLF-MPC (L-BFGS) &1.380386 / 1.781459&$9.53\!\times\!10^{-12}$ / $1.35\!\times\!10^{-10}$&$36.653/50.959\,\mathrm{ms}^{\mathrm G}$\\
OLF-MPC (CEM) &1.381653 / 1.784709&0.001451 / 0.023809&$4.459/4.269\,\mathrm{ms}^{\mathrm G}$\\
Fitted Riccati &1.380386 / 1.781459&$\simeq0$ / $\simeq0$&$<0.01\,\mu\mathrm s^{\mathrm C}$\\
Plant Riccati &1.377876 / 1.777820&0.002034 / 0.001971&$<0.01\,\mu\mathrm s^{\mathrm C}$\\
\bottomrule
\end{tabular}
\par\smallskip\footnotesize
LQG BG evaluates a stored map affine in the belief mean with phase-harmonic features, without a baseline-actor call; the DPO rows evaluate neural actors. Riccati kernel readings are near the measurement floor. Fitted-controller residuals of $10^{-11}$--$10^{-10}$ arise from different floating-point reference/execution paths; its mathematical self-error is zero. The plant oracle uses its own plant belief, so its fitted-reference error need not vanish.
\end{table}

\end{document}